\documentclass[UTF8,10pt]{article}
\usepackage[left=1.91cm,right=1.91cm,top=2.54cm,bottom=2.54cm]{geometry}
\usepackage{amsfonts}
\usepackage{amsmath}
\usepackage{amssymb}
\usepackage{pgfplots}
\usepackage{tikz}
\usetikzlibrary{arrows}
\usepackage{titlesec}
\usepackage{bm}
\usepackage{appendix}
\usepackage{tikz-cd}
\usepackage{mathtools}
\usepackage{amsthm}
\usepackage{enumitem}
\usepackage[colorlinks,linkcolor=black,anchorcolor=black,citecolor=black]{hyperref}
\usepackage[numbers]{natbib}
\usepackage{setspace}
\usepackage{cases}
\usepackage{fancyhdr}
\usepackage{txfonts}

\theoremstyle{definition}
\newtheorem{thm}{Theorem}[section]
\newtheorem{nota}[thm]{Notation}

\newtheorem{lem}[thm]{Lemma}
\newtheorem{prop}[thm]{Proposition}

\newtheorem*{rmk}{Remark}

\newcommand{\R}{\mathbb{R}}

\newcommand{\N}{\mathbb{N}}

\newcommand{\T}{\mathbb{T}}
\newcommand{\TT}{\mathcal{T}}
\newcommand{\TP}{\overline{\partial}{}}
\newcommand{\TL}{\overline{\Delta}{}}

\newcommand{\curl}{\text{curl }}
\newcommand{\dive}{\text{div }}
\newcommand{\la}{\lambda}

\newcommand{\q}{\quad}
\newcommand{\p}{\partial}
\newcommand{\bp}{(b_0\cdot\partial)}

\newcommand{\dd}{\mathfrak{D}}
\newcommand{\DD}{\mathcal{D}}
\newcommand{\grk}{\mathring{\tilde{g}}}

\newcommand{\nab}{\nabla}
\newcommand{\ak}{\tilde{a}}
\newcommand{\vk}{\tilde{v}}
\newcommand{\Ak}{\tilde{A}}
\newcommand{\ek}{\tilde{\eta}}
\newcommand{\Jk}{\tilde{J}}
\newcommand{\ar}{\mathring{a}}
\newcommand{\Ar}{\mathring{A}}
\newcommand{\ark}{\mathring{\tilde{{a}}}}
\newcommand{\Ark}{\mathring{\tilde{{A}}}}
\newcommand{\Jrk}{\mathring{\tilde{{J}}}}
\newcommand{\er}{\mathring{\eta}}

\newcommand{\erk}{\mathring{\tilde{{\eta}}}}
\newcommand{\vr}{\mathring{v}}
\newcommand{\Jr}{\mathring{J}}

\newcommand{\pArk}{\nabla_{\mathring{\tilde{A}}}}

\newcommand{\lapArk}{\Delta_{\mathring{\tilde{A}}}}

\newcommand{\lapAk}{\Delta_{\tilde{A}}}

\newcommand{\divAr}{\text{div}_{\mathring{\tilde{A}}}}

\newcommand{\curlAr}{\text{curl}_{\mathring{\tilde{A}}}}
\newcommand{\pa}{\nabla_a}

\newcommand{\pAk}{\nabla_{\tilde{A}}}
\newcommand{\diva}{\text{div}_{\tilde{a}}}
\newcommand{\divA}{\text{div}_{\tilde{A}}}
\newcommand{\curla}{\text{curl}_{\tilde{a}}}
\newcommand{\curlA}{\text{curl}_{\tilde{A}}}
\newcommand{\lkk}{\Lambda_{\kk}}
\newcommand{\Er}{\mathcal{E}}
\newcommand{\Ir}{\mathring{I}}

\newcommand{\lap}{\Delta}

\newcommand{\di}{\text{div}\,}

\newcommand{\cp}{\overline{\partial}{}}
\newcommand{\dx}{\,dx}
\newcommand{\dy}{\,dy}
\newcommand{\dz}{\,dz}
\newcommand{\dt}{\,dt}
\newcommand{\dS}{\,dS}
\newcommand{\vol}{\text{vol}\,}

\newcommand{\kk}{\kappa}
\newcommand{\eps}{\varepsilon}

\newcommand{\FF}{\mathbf{F}}

\newcommand{\NN}{\mathcal{N}}
\newcommand{\BB}{\mathfrak{B}}
\newcommand{\XX}{\mathbf{X}}

\newcommand{\PP}{\mathcal{P}}
\newcommand{\MM}{\mathcal{M}}

\newcommand{\vvv}{\textbf{v}_0}
\newcommand{\bbb}{\textbf{b}_0}
\newcommand{\qqq}{\textbf{q}_0}

\newcommand{\gr}{\mathring{g}}

\newcommand{\io}{\int_{\Omega}}

\newcommand{\ig}{\int_{\Gamma}}
\newcommand{\nn}{\hat{n}}
\newcommand{\nnk}{\tilde{n}}
\newcommand{\mm}{{(m)}}
\newcommand{\mmn}{{(m+1)}}
\newcommand{\mml}{{(m-1)}}

\newcommand{\nnr}{\mathring{\tilde{n}}}
\newcommand{\vrk}{\mathring{\tilde{v}}}
\newcommand{\RR}{\mathcal{R}}
\newcommand{\LL}{\overset{L}{=}}
\newcommand{\gk}{\tilde{g}}
\newcommand{\itt}{\int_0^T\int_0^t}
\numberwithin{equation}{section}
\usepackage{xcolor}

\newcommand{\Nf}{\mathcal{N}_0}

\begin{document}
\bibliographystyle{plain}
\title{\textbf{Local Well-Posedness of the Free-Boundary Incompressible Magnetohydrodynamics with Surface Tension}}
\author{Xumin Gu\thanks{School of Mathematics, Shanghai University of Finance and Economics, Shanghai 200433, China. Email: \texttt{gu.xumin@shufe.edu.cn}.},\,\, Chenyun Luo\thanks{Department of Mathematics, The Chinese University of Hong Kong, Shatin, NT, Hong Kong.
Email: \texttt{cluo@math.cuhk.edu.hk}.\,\,\,\,\,\,\,
},\,\, Junyan Zhang \thanks{Department of Mathematics, National University of Singapore, Singapore 119076.
Email: \texttt{zhangjy@nus.edu.sg}}}
\date{\today}
\maketitle

\begin{abstract}
We prove the local well-posedness of the 3D free-boundary incompressible ideal magnetohydrodynamics (MHD) equations with surface tension, which describe the motion of a perfect conducting fluid in an electromagnetic field. We adapt the ideas developed in the remarkable paper \cite{coutand2007LWP} by Coutand and Shkoller to generate an approximate problem with artificial viscosity indexed by $\kappa>0$ whose solution converges to that of the MHD equations as $\kappa\to 0$. However, the local well-posedness of the MHD equations is no easy consequence of Euler equations thanks to the strong coupling between the velocity and magnetic fields.  This paper is the continuation of the second and third authors' previous work \cite{luozhangMHDST3.5} in which the a priori energy estimate for incompressible free-boundary MHD with surface tension is established.  But the existence is not a trivial consequence of the a priori estimate as it cannot be adapted directly to the approximate problem due to the loss of the symmetric structure.
\end{abstract}

		

\setcounter{tocdepth}{1}

\tableofcontents

\section{Introduction}

We consider the following 3D incompressible ideal MHD system which describes the motion of a conducting fluid with free surface boundary in an electro-magnetic field under the influence of surface tension
\begin{equation}\label{MHD}
\begin{cases}
(\partial_t +u\cdot\nabla) u-B\cdot\nabla B+\nabla P=0,~~P:=p+\frac{1}{2}|B|^2~~~& \text{in}~\DD; \\
(\partial_t +u\cdot\nabla) B-B\cdot\nabla u=0,~~~&\text{in}~\DD; \\
\dive u=0,~~\dive B=0,~~~&\text{in}~\DD,
\end{cases}
\end{equation} with boundary conditions
\begin{equation}\label{MHDB}
\begin{cases}
(\partial_t +u\cdot\nabla)|_{\p\DD}\in\mathcal{T}(\p\DD), \\
P=\sigma\mathcal{H}~~~& \text{on}~\p\DD, \\
B\cdot n=0~~~& \text{on}~\p\DD.
\end{cases}
\end{equation} Here $\DD:=\cup_{0\leq t\leq T}\{t\}\times\DD_t$ and $\DD_t\subseteq\R^3$ is the bounded domain occupied by the conducting fluid (plasma) whose boundary $\p\DD_t$ moves with the velocity of the fluid. Here $u=(u_1,u_2,u_3)$ is the fluid velocity, $B=(B_1,B_2,B_3)$ is the magnetic field, $p$ is the fluid pressure and $P:=p+\frac12|B|^2$ is the total pressure. The quantity $\mathcal{H}$ is the mean curvature of the free surface $\p\DD_t$, $\sigma>0$ is a given constant, called surface tension coefficient and $n$ denotes the exterior unit normal to $\p\DD_t$. Throughout the manuscript, we will use the notation $D_t:=\p_t+u\cdot\nabla$ to denote the material derivative.

The first boundary condition shows that the boundary of the plasma moves with the velocity of the fluid. It can be equivalently expressed as the velocity of $(\p\DD_t)$ is equal to $u\cdot n$.  The second boundary condition shows that the normal stress tensor is balanced by the surface tension. Here we note that $\mathcal{H}$ is determined by the unknown moving domain and thus not known a priori. The third boundary condition implies that the plasma liquid is a perfect conductor. We also note that $\dive B=0$ and $B\cdot n|_{\p\DD_t}=0$ are both required only for initial data and they automatically propagate to any positive time. Therefore, the system \eqref{MHD}-\eqref{MHDB} is not over-determined. 

Under the conditions above, we have the following conservation of physical energy \cite[Section 1]{luozhangMHDST3.5}.
\begin{equation}
\frac{d}{dt}\left(\frac{1}{2}\int_{\DD_t} |u|^2 +  |B|^2\dx+\sigma\int_{\p\DD(t)} dS(\p\DD(t))\right)=0
\end{equation}

Given a simply connected domain $\DD_0\subseteq\R^3$ and initial data $u_0$ and $B_0$ satisfying $\dive u_0=0$ and $\dive B_0=0, B_0\cdot n|_{\{t=0\}\times \p\DD_0}=0$, we want to find a set $\DD$ and vector fields $u$ and $B$ solving \eqref{MHD}-\eqref{MHDB} with initial data
\begin{equation}
\DD_0=\{x: (0,x)\in \DD\},\q (u,B)=(u_0, B_0),\q \text{in}\,\,\{t=0\}\times \Omega_0.
\end{equation}

\begin{rmk}
When the surface tension is neglected, the classical Rayleigh-Taylor sign condition $-\nabla_n P\geq c_0>0$ is necessary for the well-posedness. Ebin \cite{ebin1987} and Hao-Luo \cite{haoluo2018} constructed the counterexamples for Euler equations and MHD equations respectively to show that the free-boundary problems can be ill-posed when the Rayleigh-Taylor sign condition is violated.
\end{rmk}

\subsection{History and Background}

\subsubsection{Physical background: Plasma-Vacuum model}

The free-boundary problem \eqref{MHD}-\eqref{MHDB} originates from the plasma-vacuum free-interface model, which is an important theoretic model both in the laboratory and in astrophysical magnetohydrodynamics. The plasma is confined in a vacuum with another magnetic field $\hat{B}$, and there is a free interface $\Gamma(t)$, moving with the motion of plasma, between the plasma region $\Omega_+(t)$ and the vacuum region $\Omega_{-}(t)$. Such a model requires that \eqref{MHD} holds in the plasma region $\Omega_+(t)$ and the pre-Maxwell system holds in vacuum $\Omega_{-}(t)$:
\begin{equation}\label{vacuum}
\curl\hat{B}=\mathbf{0},~~~\dive\hat{B}=0.
\end{equation}  
On the interface $\Gamma(t)$, the perfect conducting condition is required that there is no jump in the \textit{the normal component}:
\begin{equation}\label{interface}
B\cdot n=\hat{B}\cdot n=0,~~~[P]:=p+\frac12|B|^2-\frac12|\hat{B}|^2=\sigma\mathcal{H}
\end{equation}  where $n$ is the exterior unit normal to $\Gamma(t)$.  Finally, there is a rigid wall $W$ wrapping the vacuum region on which the following boundary condition holds
\[
\hat{B}\times n_w=\mathcal{J}
\] where $\mathcal{J}$ is the given outer surface current density (as an external input of energy) and $n_w$ is the exterior unit normal of $W$. Note that for the ideal MHD equations, $B\cdot n=0$ on $\p\DD_t$ and $\di B=0$ in $\DD_t$ should be the constraints on initial data and these constraints propagate. See more details in \cite[Chapter 4, 6]{MHDphy}. 

\begin{rmk} We also note that, in the study of MHD surface waves in the free-interface models, the effect of surface tension is crucially important for modeling MHD flows in liquid metals, e.g., the liquid-metal film flows, jets and droplets, etc. See also \cite{MHDSTphy1,MHDSTphy2,MHDSTphy3} and references therein. Even if we consider the MHD flows in astrophysical plasmas, where the surface tension effect and magnetic diffusion are usually neglected, it is still useful to keep surface tension as a stabilization effect in numerical simulations of the magnetic Rayleigh-Taylor instability \cite{MHDSTphy4, MHDSTphy5}.
\end{rmk}
In this manuscript, we consider the case that $\hat{B}=\mathbf{0}$, i.e., we neglect the magnetic field in the vacuum. It characterizes the free-surface motion of an isolated liquid plasma under the influence of surface tension.

\subsubsection{Review of previous results}

In the absence of magnetic field $B$, the system \eqref{MHD}-\eqref{MHDB} is reduced to the free-boundary incompressible Euler equations. The study of free-surface incompressible Euler equations has blossomed in the past several decades. In the case of no surface tension ($\sigma=0$), the first breakthrough is Wu \cite{wu1997LWPww,wu1999LWPww} in which she proved the local well-posedness (LWP) for the irrotational case without surface tension. See also \cite{ambrose05,lannes2005ww,mz2009,alazard2011st,alazard2014cauchy} for the LWP with or without surface tension.  In the case of nonzero vorticity, Christodoulou-Lindblad \cite{christodoulou2000motion} first proved the a priori estimates and then Lindblad \cite{lindblad2003,lindblad2005well} proved the LWP by using Nash-Moser iteration. Later Coutand-Shkoller \cite{coutand2007LWP,coutand2010LWP} proved the LWP by using tangential smoothing and the energy estimates without loss of regularity in the case of both $\sigma=0$ and $\sigma>0$. See also Zhang-Zhang \cite{zhangzhang08Euler} for the study of incompressible water wave. In the case of nonzero surface tension, we refer to \cite{Schweizer05,coutand2007LWP,coutand2008VX,shatah2008a,shatah2008b,shatah2011} for LWP, and \cite{ignatova2016,DK2017,DKT} for low regularity estimates. 

However, the study of free-boundary MHD equations is far less developed as opposed to Euler equations. \textit{The strong coupling between the magnetic field and the motion of fluid destroys some good properties of Euler equations such as the propagation of the irrotational assumption.} Most of the known results focus on the case of zero surface tension. When the surface tension is neglected, extra stabilization such as the Rayleigh-Taylor sign condition is required. Lee \cite{leeMHD1,leeMHD2} proved the LWP for viscous-resistive MHD and the vanishing viscosity-resistivity limit. For the free-boundary problem of ideal incompressible MHD under the Rayleigh-Taylor sign condition, Hao-Luo \cite{haoluo2014} proved the a priori estimates and \cite{haoluo2019} proved the linearized LWP. Then the first author and Wang \cite{gu2016construction} proved the LWP. The second and the third authors \cite{luozhangMHD2.5} proved the minimal regularity $H^{\frac52+\eps}$ estimates for a small fluid domain. For the plasma-vacuum model under the Rayleigh-Taylor sign condition, Hao \cite{hao2017} proved the a priori estimates when $\mathcal{J}=\mathbf{0}$ and the first author \cite{guaxi1,guaxi2} proved the LWP for the axisymmetric case with a non-zero vacuum magnetic field in a non-simply connected domain. We note that there is another non-collinearity condition\footnote{Such condition comes from the study of the stability of current-vortex sheet which is a two-fluid (plasma-plasma) model in free-boundary MHD.} $|B\times\hat{B}|\geq c_0>0$ which gives extra $1/2$-order regularity of the free interface than the Rayleigh-Taylor sign condition for the plasma-vacuum model. Under this condition, Morando-Trebeschi-Trakhinin \cite{iMHDlinear} proved the LWP for the linearized plasma-vacuum system and Sun-Wang-Zhang \cite{sun2017well} proved the nonlinear LWP. Coulombo-Morando-Secchi-Trebeschi \cite{iMHDVX} proved the a priori estimates for 3D incompressible current-vortex sheets and Sun-Wang-Zhang \cite{sun2015well} proved the LWP. So far, the energy estimates and well-posedness of the plasma-vacuum model in general cases (i.e., $\mathcal{J}\neq 0$ and without axisymmetric assumption) under the Rayleigh-Taylor sign condition are still open problems. 

In the case of nonzero surface tension, there are very few results for the free-boundary MHD system and most previous works focus on the resistive or viscous MHD. To the best of our knowledge, The second and the third authors' previous work \cite{luozhangMHDST3.5} which proved the $H^{7/2}$ a priori estimate is the only available result for incompressible ideal MHD with surface tension. We also refer to Chen-Ding \cite{chendingMHDlimit} for the inviscid-non-resistive limit under the condition $B|_{\p\DD_t}=\mathbf{0}$, Wang-Xin \cite{wangxinMHD} for GWP of incompressible resistive MHD around a transversal uniform magnetic field, and Padula-Solonnikov \cite{Solonnikov}, Guo-Zeng-Ni \cite{GuoMHDSTviscous} for incompressible viscous-resistive MHD. 

Finally, for compressible MHD, we refer to Secchi-Trakhinin \cite{secchi2013well} for the LWP of plasma-vacuum model under non-collinearity condition, and Chen-Wang \cite{chengqMHD}, Trakhinin \cite{trakhininMHD2009} and Wang-Yu \cite{wangyuMHD2d} for compressible current-vortex sheets in 3D and 2D. Very recently, Trakhinin-Wang proved the LWP of free-boundary compressible ideal MHD under Rayleigh-Taylor sign condition \cite{trakhininMHD2020} or with surface tension \cite{trakhininMHD2021}. All these results are proved by Nash-Moser iteration and thus there is no energy estimate without regularity loss. The third author proved the LWP \cite{ZhangCRMHD2} and the incompressible limit \cite{ZhangCRMHD1} of compressible resistive MHD under the Rayleigh-Taylor sign condition with energy estimates of no regularity loss. Finding suitable energy estimates without regularity loss for compressible ideal MHD with or without surface tension is also a wide-open problem. The plasma-vacuum model in compressible MHD under the Rayleigh-Taylor sign condition is also unsolved. See Trakhinin \cite{trakhininMHD2016} for detailed discussion.

In this manuscript, we prove the local well-posedness with energy estimates of no regularity loss for the free-boundary problem in incompressible ideal MHD with surface tension. Our result is a necessary step to study the plasma-vacuum model under the influence of surface tension, which is an original theoretical model in the study of confined plasma in both laboratory and astrophysical MHD.

\subsection{Reformulation in Lagrangian coordinates}

We reformulate the MHD equations in Lagrangian coordinates and thus the free-surface domain becomes fixed. Let $\Omega\subseteq\R^3$ be a bounded domain. Denoting coordinates on $\Omega$ by $y=(y_1,y_2,y_3)$, we define $\eta:[0,T]\times \Omega\to\DD$ to be the flow map of the velocity $u$, i.e., 
\begin{equation}
\p_t \eta (t,y)=u(t,\eta(t,y)),\q
\eta(0,y)=y.
\end{equation}
We introduce the Lagrangian velocity, magnetic field and pressure respectively by
\begin{equation}
v(t,y)=u(t,\eta(t,y)),\q
b(t,y)=B(t,\eta(t,y)),\q
q(t,y)=P(t,\eta(t,y)).
\end{equation}
Let $\p$ be the spatial derivative with respect to $y$ variable. We introduce the cofactor matrix $a=[\p\eta]^{-1}$. Specifically, $a^{\mu\alpha}=\frac{\p y^{\mu}}{\p x^\alpha}$, where $x^\alpha(t,y)=\eta^\alpha(t,y)$. Also, we define $J:=\det[\p\eta]$, which is well-defined since $\eta(t,\cdot)$ is almost the identity map when $t$ is sufficiently small. It's worth noting that $a$ verifies the Piola's identity and $J=1$ in the incompressible case, i.e., 
\begin{equation}\label{piola}
\p_\mu a^{\mu\alpha} =\p_\mu (Ja^{\mu\alpha}) = 0\text{ and }J=1.
\end{equation}
Here, the Einstein summation convention is used for repeated upper and lower indices. Above and throughout, all Greek indices range over 1, 2, 3, and the Latin indices range over 1, 2. 

For the sake of simplicity and clean notation, here we consider the model case when 
\begin{equation}
\Omega=\T^2\times (0,1),
\label{Omega}
\end{equation}
 where $ \partial\Omega=\Gamma_0\cup\Gamma$ and $\Gamma=\T^2\times \{1\}$ is the top (moving) boundary, $\Gamma_0=\T^2\times\{0\}$ is the fixed bottom. We mention here that $\Omega$ is known as the reference domain, which allows us to work in one coordinate patch. We refer the interested readers to \cite{coutand2007LWP} for detailed discussions. Using a partition of unity, e.g., \cite{DKT}, a general domain can also be treated with the same tools we shall present. However, choosing $\Omega$ as above allows us to focus on the real issues of the problem without being distracted by the cumbersomeness of the partition of unity. Let $N$ stand for the outward unit normal of $\p\Omega$. In particular, we have $N=(0,0,-1)$ on $\Gamma_0$ and $N=(0,0,1)$ on $\Gamma$. 

Under this setting, the system \eqref{MHD}-\eqref{MHDB} can be reformulated as:
\begin{equation}\label{MHDL1}
\begin{cases}
\partial_tv_{\alpha}-b_{\beta}a^{\mu\beta}\partial_{\mu}b_{\alpha}+a^{\mu}_{\alpha}\partial_{\mu}q=0~~~& \text{in}~[0,T]\times\Omega;\\
\partial_t b_{\alpha}-b_{\beta}a^{\mu\beta}\partial_{\mu}v_{\alpha}=0~~~&\text{in}~[0,T]\times \Omega ;\\
a^{\mu\alpha}\partial_{\mu}v_{\alpha}=0,~~a^{\mu\alpha}\partial_{\mu}b_{\alpha}=0~~~&\text{in}~[0,T]\times\Omega;\\
v^3=b^3=0~~~&\text{on}~\Gamma_0;\\
a^{\mu\alpha}N_{\mu}q+\sigma(\sqrt{g}\Delta_g \eta^{\alpha})=0 ~~~&\text{on}~\Gamma;\\
a^{\mu\nu}b_{\nu}N_{\mu}=0 ~~~&\text{on}~\Gamma,\\
\end{cases}
\end{equation}
where $N$ is the unit outer normal vector to $\p\Omega$,  and $\Delta_g$ is the Laplacian of the metric $g_{ij}$ induced on $\Gamma(t) =\eta(t, \Gamma)$ by the embedding $\eta$. Specifically,  we have:
\begin{equation}\label{gij}
g_{ij}=\TP_i\eta^{\mu}\TP_j\eta_{\mu},~\Delta_g(\cdot)=\frac{1}{\sqrt{g}}\TP_i(\sqrt{g}g^{ij}\TP_j(\cdot)),\text{ where } g:=\det (g_{ij}).
\end{equation} 
Here, we use $\TP$ to emphasis that the derivative is tangential to $\Gamma$. In particular, $\TP=(\TP_1,\TP_2)=(\p_1, \p_2)$.

By the second equation of \eqref{MHDL1} and the divergence-free condition on $b$, we get $\p_t(a^{\mu\alpha}b_{\mu})=0$ which implies $a^{\mu\alpha}b_{\mu}=b_0^{\alpha}$ and thus $b^{\alpha}=b_0^{\mu}\p_\mu \eta^{\alpha}=(b_0\cdot\p)\eta^{\alpha}$. See Gu-Wang \cite[(1.13)-(1.15)]{gu2016construction} for the proof. Therefore, the system \eqref{MHDL1} can be equivalently written as the following system of $(\eta,v,q)$
\begin{equation}\label{MHDL}
\begin{cases}
\p_t\eta=v~~~& \text{in}~[0,T]\times\Omega;\\
\p_tv-\bp^2\eta+\pa q=0~~~& \text{in}~[0,T]\times\Omega;\\
\dive_a v=0,&\text{in}~[0,T]\times\Omega;\\
\dive b_0=0~~~&\text{in}~\{t=0\}\times \Omega ;\\
v^3=b_0^3=0~~~&\text{on}~\Gamma_0;\\
a^{3\alpha}q+\sigma(\sqrt{g}\Delta_g \eta^{\alpha})=0 ~~~&\text{on}~\Gamma;\\
b_0^3=0 ~~~&\text{on}~\Gamma,\\
(\eta,v)=(Id,v_0)~~~&\text{on}~\{t=0\}{\times}\overline{\Omega}.
\end{cases}
\end{equation}
\begin{nota}  The differential operator $\pa:=(\pa^1, \pa^2, \pa^3)$ with $\pa^{\alpha}:=a^{\mu\alpha}\p_\mu$, and for a smooth vector field $X$, we denote by $\dive_a X:= \nab_a \cdot X = a^{\mu\alpha} \p_\mu X_\alpha$ the Eulerian divergence of $X$ and by $\di X :=\p\cdot X=\delta^{\mu\alpha} \p_\mu X_\alpha$ the flat divergence. 
\end{nota}
\begin{rmk}
The initial data of $q$ is determined by $v_0$ and $b_0$. In particular, $q_0$ satisfies an elliptic equation 
\begin{align}\label{qdata}
\begin{cases}
-\Delta q_0=(\p v_0)(\p v_0)-(\p b_0)(\p b_0),~~\text{in}\,\,\Omega,\\
q_0 = 0, ~~\text{on}\,\,\Gamma, \\
\frac{\p q_0}{\p N} = 0, ~~\text{on}\,\,\Gamma_0.
\end{cases}
\end{align}
The boundary condition on $\Gamma$ is derived by restricting the boundary condition $a^{3\alpha} q + \sigma (\sqrt{g}\Delta_g \eta^\alpha)=0$ at $t=0$. Then it becomes
$
 \delta^{3\alpha} q_0+\sigma \TL \eta_0^\alpha =0, 
$
where $\TL := \TP_1^2+\TP_2^2$, and this yields $q_0=0$ on $\Gamma$ since $\TP^2 \eta_0^3 =0$.  On the other hand, the boundary condition on $\Gamma_0$ is derived from by restricting $\p_tv-\bp^2\eta+\pa q=0$ on $\{t=0\}\times \Gamma_0$  and then taking the normal component, where we have used the fact that $a^{13}=a^{23}=0$ and $a^{33}=1$ on $\Gamma_0$. 
\end{rmk}

\subsection{Main result}
We prove the local well-posedness of \eqref{MHDL} in the presenting manuscript. We denote $\|f\|_{s}: = \|f(t,\cdot)\|_{H^s(\Omega)}$ for any function $f(t,y)\text{ on }[0,T]\times\Omega$ and $|f|_{s}: = |f(t,\cdot)|_{H^s(\Gamma)}$ for any function $f(t,y)\text{ on }[0,T]\times\Gamma$. Let $\Pi$ be the canonical normal projection defined on the tangent bundle of the moving interface. Our main result is:

\begin{thm}\label{lwp}
Let $v_0\in H^{4.5}(\Omega)\cap H^5(\Gamma)$ and $b_0\in H^{4.5}(\Omega)$ be divergence-free vector fields with $(b_0\cdot N)|_{\Gamma}=0$, and define $q_0$ as in \eqref{qdata}. Then there exists some $T>0$, only depending on $\sigma,v_0,b_0$, such that the system \eqref{MHDL} with initial data $(v_0,b_0,q_0)$ has a unique strong solution $(\eta,v,q)$ with the energy estimates
\begin{equation}\label{lwpenergy0}
\sup_{0\leq t\leq T}E(t)\leq \mathcal{C},
\end{equation}where $\mathcal{C}$ is a constant depends on $\|v_0\|_{4.5}, \|b_0\|_{4.5}$, $|v_0|_5$, and
\begin{equation}\label{lwpenergy}
\begin{aligned}
E(t)&:=\left\|\eta(t)\right\|_{4.5}^2+\sum_{j=0}^3\left\|\left(\p_t^jv(t),\p_t^j\bp\eta(t)\right)\right\|_{4.5-j}^2+\left\|\left(\p_t^4 v(t), \p_t^4 \bp\eta(t)\right)\right\|_{0}^2\\
&+\sum_{j=0}^3\left|\TP\left(\Pi\TP^{3-j}\p_t^j v(t)\right)\right|_0^2+\left|\TP\big(\Pi\TP^3 \bp\eta(t)\big)\right|_0^2.
\end{aligned}
\end{equation}
Moreover, the $H^5(\Gamma)$-regularity of $v$ on the free-surface can also be recovered, in the sense that there exists some $0<T_1<T$, depending only on $\sigma^{-1},v_0,b_0$, such that
\begin{equation}\label{vhigh}
\sup_{0\leq t\leq T_1}|\eta(t)|_{5}^2+|v(t)|_{5}^2\leq C(\sigma^{-1}, \|v_0\|_{4.5}, \|b_0\|_{4.5}).
\end{equation}
\end{thm}

\begin{rmk}[\textbf{Smoothing effect of $b_0\cdot \p$}] It can be seen that in \eqref{lwpenergy} $v$ and $\bp\eta$ are of the same interior regularity (i.e., $H^{4.5}(\Omega)$). This suggests that $\bp$ and $\p_t$ behave the same when falling on the flow map $\eta$.  This observation turns out to be very important when studying the energy of the approximate equations \eqref{MHDLkk0} defined below. 
\end{rmk}

\subsection{Strategy of the proof}
\subsubsection{Necessity of the tangential smoothing}\label{1smooth}
 In \cite{luozhangMHDST3.5}, the second and third authors proved the a priori estimates of \eqref{MHDL}. However, it is often highly nontrivial to prove the local well-posedness for a free-boundary problem of an inviscid fluid, especially when equipped with the Young-Laplace boundary condition, by a simple iteration scheme and fixed-point argument for the linearized equations. The reason is that the linearization breaks the subtle cancellation structure on the free surface and thus causes the loss of tangential derivatives of the flow map $\eta$, which also occurs for incompressible Euler equations with surface tension. 

In their remarkable work \cite{coutand2007LWP}, Coutand and Shkoller introduced an approximate system in the Lagrangian coordinates by smoothing the nonlinear coefficients in the tangential direction. This can be adapted to study the MHD equations and the tangential smoothing preserves the essential transport-type structure of the original equations. Specifically,  we define $\lkk$ to be the standard mollifier with parameter $\kk>0$ on $\R^2$ as in \eqref{lkk0}. Let $\ek:=\lkk^2 \eta$ and $\ak=[\p\ek]^{-1}$. Then we set the nonlinear $\kk$-approximation problem by replacing $a$ with $\ak$. However, such construction does not apply to MHD because we also need to control $\|[\lkk^2,\bp]\eta\|_{4.5}$ in which there is a normal derivative $b_0^3\p_3$ that is not compatible with the tangential mollification in the interior. Motivated by Gu-Wang \cite{gu2016construction}, we first mollify the flow map on the boundary, then extend it into the interior by the harmonic extension, i.e., 
\begin{equation} \label{tangential smoothing0}
\begin{cases}
-\Delta \ek=-\Delta\eta~~~&\text{ in }\Omega,\\
\ek=\lkk^2\eta~~~&\text{ on }\Gamma.
\end{cases}
\end{equation}

Define $\ak:=[\p\ek]^{-1},~\Jk:=\det[\p\ek]$ and $\Ak:=\Jk\ak$, then we have the Piola's identity $\p_\mu\Ak^{\mu\alpha}=0$. The nonlinear approximate system is defined to be
\begin{equation}\label{MHDLkk0}
\begin{cases}
\p_t\eta=v~~~& \text{in}~[0,T]\times\Omega;\\
\p_tv-\bp^2\eta+\pAk q=0~~~& \text{in}~[0,T]\times\Omega;\\
\diva v=0,&\text{in}~[0,T]\times\Omega;\\
\dive b_0=0~~~&\text{in}~\{t=0\}\times \Omega ;\\
v^3=b_0^3=0~~~&\text{on}~\Gamma_0;\\
\Ak^{3\alpha}q = -\sigma \sqrt{g}(\Delta_g \eta\cdot \nnk )\nnk^\alpha +\kk\left((1-\TL) (v\cdot \nnk)\right)\nnk^\alpha~~~&\text{on}~\Gamma;\\
b_0^3=0 ~~~&\text{on}~\Gamma,\\
(\eta,v)=(Id,v_0)~~~&\text{in}~\{t=0\}{\times}\overline{\Omega}.
\end{cases}
\end{equation}
In this paper, we will
(i). derive the uniform-in-$\kk$ a priori estimates of the system \eqref{MHDLkk0}, and then
 (ii). solve the nonlinear $\kk$-approximation system \eqref{MHDLkk0}.

\subsubsection{Necessity of the artificial viscosity}\label{2viscosity}

There is an artificial viscosity term $\kk\left((1-\TL) (v\cdot \nnk)\right)\nnk^\alpha$ in the smoothed surface tension equation on the boundary. This was first introduced by Coutand-Shkoller in \cite{coutand2007LWP} where the authors mentioned that the artificial viscosity term appears to be necessary to prove the existence of an inviscid fluid with non-trivial vorticity and surface tension. This term also appears in the subsequent work that studies the free-surface fluid with surface tension, e.g.,  Cheng-Coutand-Shkoller \cite{coutand2008VX} for the vortex sheets, Coutand-Hole-Shkoller \cite{coutand2013LWP} for the compressible Euler, and very recently Trakhinin-Wang \cite{trakhininMHD2021} for the compressible MHD.

\begin{rmk}
Very recently, the first author and Lei \cite{GuLeielastoST} proved the LWP of incompressible elastodynamics with surface tension by proving the inviscid limit of the visco-elastodynamics system in standard Sobolev spaces. We also note that the inviscid-non-resistive limit of free-boundary MHD (under $B|_{\p\DD_t}=\mathbf{0}$) was recently proved by Chen-Ding \cite{chendingMHDlimit} in co-normal Sobolev spaces. However, the analogous inviscid limit in standard Sobolev space does not apply to MHD due to the existence of MHD boundary layers. 
\end{rmk}

An essential reason for introducing such an artificial viscosity term is that the presence of surface tension forces us to control all of the time derivatives. In particular, the pressure $q$ satisfies an elliptic equation and it appears that one can only get control of it by considering the Neumann boundary condition instead of the Dirichlet boundary condition due to the presence of surface tension. The Neumann boundary condition contains the time derivative of $v$, and thus we have to include the time derivatives in our energy. 

However, the full-time derivatives of $v$ and $\bp\eta$ only have $L^2(\Omega)$ regularity which introduces two new difficulties. First, we cannot get estimates of the full-time derivatives of $q$ via the elliptic equation due to the low spatial regularity. Second, we do not have any control for the terms containing full-time derivatives on the boundary due to the failure of the Sobolev trace lemma. For the original system, one can use the subtle cancellation structure developed in \cite{DK2017,luozhangMHDST3.5} to resolve this difficulty. But such cancellation structure no longer holds for the nonlinear $\kk$-approximate problem due to the presence of tangential smoothing. Therefore, introducing the artificial viscosity term could produce $\kk$-weighted higher-order terms on the boundary, which enables us to finish the energy control.

\begin{rmk}
The Young-Laplace boundary condition only gives us the information in the Eulerian normal direction. Therefore, the artificial viscosity can only be imposed in the smoothed Eulerian normal direction $\kk\left((1-\TL)(v\cdot\nnk)\right)\nnk^{\alpha}$ instead of all the components, otherwise the system would be over-determined.
\end{rmk}

\subsubsection{Difference from the case without surface tension}\label{3GW}

The first author and Wang \cite{gu2016construction} proved the LWP of incompressible MHD without surface tension, in which the pressure $q$ can be controlled by the elliptic equation with Dirichlet (zero) boundary condition, and thus one can avoid the estimates of all time derivatives which turn out to be very complicated in the presenting manuscript. This tells an essential difference from the case without the surface tension. 

On the other hand, as mentioned in \cite{coutand2007LWP, DK2017,luozhangMHDST3.5}, surface tension has a stronger stabilization effect than the Rayleigh-Taylor sign condition in the case without surface tension. The presence of surface tension allows us to control the boundary norms of the normal component of $v$ and $\bp\eta$ by comparing with the corresponding Eulerian normal projections instead of using the normal trace theorem to reduce to interior tangential estimates. We refer Section \ref{sect bdy estimate 3.3} for details. This property allows us to gain extra $1/2$ derivatives in the interior, and there is no need to introduce the Alinhac good unknowns and correction terms as in \cite{gu2016construction}.

\subsubsection{Illustration on the energy functional}\label{4energy}
Let $\Pi$ be the canonical normal project defined on the tangent bundle of the moving interface and $\nnk$ be the (Eulerian) unit normal (We refer to Lemma \ref{geometric} for the precise definition). 
The energy functional of the nonlinear approximate problem \eqref{MHDLkk0} is defined to be
\[
E_\kk=E_{\kk}^{(1)}+E_{\kk}^{(2)}+E_{\kk}^{(3)}, 
\]
where
\begin{align*}
E_\kk^{(1)}:=&\left\|\eta(\kk)\right\|_{4.5}^2+\sum_{j=0}^{3}\left\|\left(\p_t^{j}v(\kk),\p_t^{j}\bp\eta(\kk)\right)\right\|_{4.5-j}^2+\left\|\left(\p_t^4 v(\kk),\p_t^4 \bp\eta(\kk)\right)\right\|_{0}^2\\
&+\sum_{j=0}^{3}\left|\TP\left(\Pi\TP^{3-j}\p_t^j v(\kk)\right)\right|_0^2+\left|\TP\big(\Pi\TP^3 \bp\eta(\kk)\big)\right|_0^2, \\
E_{\kk}^{(2)}:= &\sigma\int_0^T\bigg(\sum_{j=1}^{4}\left|\sqrt{\kk}\p_t^j v(\kk)\cdot\nnk(\kk)\right|_{5-j}^2+\left|\sqrt{\kk}\bp v(\kk) \cdot \nnk(\kk)\right|_4^2\bigg)\dt,\nonumber\\
E_{\kk}^{(3)}:=&\sum_{k=0}^4\int_0^T \Big (\big\|\sqrt{\kk} \p_t^{k}v(\kk)\big\|_{5.5-k}^2+\big\|\sqrt{\kk} \p_t^k \bp \eta(\kk)\big\|_{5.5-k}^2\Big)\dt.
\end{align*}

The energy constructed above looks much more complicated than \eqref{lwpenergy}, but it is quite natural. First, $E_\kk^{(1)}$ constitutes the non-weighted energies that are needed to close the a priori estimate for the MHD equations without the artificial viscosity (cf. Luo-Zhang \cite{luozhangMHDST3.5}). Then $E_{\kk}^{(2)}$ consists of the $\kk$-weighted higher-order energy terms produced by the artificial viscosity when dealing with the tangential estimates. 

Besides, extra error terms are generated when all the derivatives fall on the smoothed Eulerian normal $\nnk$ in the construction of $E_{\kk}^{(2)}$. Since $E_\kk^{(2)}$ only gives us higher-order control of the normal component instead of all components. Most of the top order error terms should be treated by moving them to the interior with the help of the Sobolev trace lemma, and we use $E_{\kk}^{(3)}$ to record all of them. Nevertheless, due to the strong coupling structure (see Subsection \ref{5CS} for more details) between the velocity and the magnetic field, the terms in $E_\kk^{(3)}$ must be controlled together via the Hodge-type div-curl estimate, and thus we have to include the associated magnetic terms in $E_\kk^{(3)}$ as well. 

When closing the energy estimates of $E_{\kk}^{(2)}$ and $E_{\kk}^{(3)}$, and $\TP^3\p_t$-tangential estimates, one needs the control of $\sqrt{\kk}$-weighted $H^5(\Gamma)$-norms of $\eta,v$ and $\bp\eta$ recorded in Lemma \ref{super control}. These $\sqrt{\kk}$-weighted bounds can be established by considering $\TP^4,\TP^4\p_t,\TP^4\bp$-differentiated smoothed Young-Laplace boundary condition. See also Coutand-Shkoller \cite[Lemma 12.6]{coutand2007LWP}.

\begin{rmk}
In the proof of Lemma \ref{super control}, the self-adjointness of $\lkk$ is used to keep the structure and close the energy estimates. This is the reason that we need to mollify $\eta$ twice in \eqref{tangential smoothing0}. 
\end{rmk}

\subsubsection{Difference between Euler equations and MHD with surface tension}\label{5CS}

As mentioned in \cite{luozhangMHD2.5,luozhangMHDST3.5}, \textit{the Cauchy invariance for the Euler equations no longer holds for MHD equations, which makes it impossible to get a higher regularity of the flow map $\eta$ than that of the velocity $v$.} Without such property, one cannot control the $\TP^4$-tangential energy estimate directly as what Coutand-Shkoller did in \cite{coutand2007LWP} for the incompressible Euler equations. 

 In addition to this, the strong coupling between $v$ and $b=\bp\eta$ yields that the Sobolev norms of their vorticities have to be controlled together (see equation \eqref{curlbeq}). As a consequence, 
the full Sobolev norms of $v$ and $b$ (and their time derivatives) have to be studied simultaneously. 

Finally, the $|\sqrt{\kk}\eta|_6$ regularity for 3D incompressible Euler equations in \cite{coutand2007LWP} cannot be achieved either. But this does not affect the proof for the MHD system unless one wants to get a $H^6(\Gamma)$-posteriori estimates for the flow map $\eta$.

\subsubsection{Penalization method to solve the linearized problem}\label{6penalize}

Finally, it remains to solve the nonlinear approximation problem. With the help of tangential smoothing, it is not difficult for us to finish the iteration from the linearized approximate problem to the nonlinear one. But it is still difficult to solve the linearized approximate problem by the fixed-point argument even if one can get the a priori estimates without the loss of regularity. The reason is that we do not have any suitable equation for $q$ and thus the structure of the linearized system is no longer preserved in the verification of the fixed-point argument. Motivated by \cite{coutand2007LWP}, we use the penalization method to solve the linearized system. We introduce a penalized pressure defined by $q_{\lambda}:=-\lambda^{-1}\divAr w_{\lambda}$ and prove the existence of $L^2$-weak solution to the penalized problem by Galerkin's method. Then we take the weak limit by passing $\lambda\to 0$ to get the weak solution of the linearized approximate problem. Finally, one can prove the weak solution is strong by $H^1$-estimates together with the inverse theorem of div-curl decomposition (cf. Lemma \ref{hodge} (2)).

\begin{rmk}
The penalization method is not needed in the compressible case because the free-boundary compressible MHD is a first-order symmetric hyperbolic system with characteristic boundary conditions and the corresponding linearized problem can be solved by the duality argument in Lax-Phillips \cite{lax60}. We refer to Trakhinin-Wang \cite{trakhininMHD2020,trakhininMHD2021} for details.
\end{rmk}

\begin{rmk}
We cannot directly prove the weak solution of the penalized problem is a strong solution as in \cite{coutand2007LWP} since the divergence part cannot be controlled because of the presence of the magnetic field. That is why we first take the weak limit and then verify the $H^1$-estimates for the linearized system.
\end{rmk}

\begin{rmk}
In the a priori estimates and iteration process of the linearized approximate problem, the energy control is much simpler than the uniform-in-$\kk$ estimates of the nonlinear approximate problem \eqref{MHDLkk} because we no longer require the energy is $\kk$-independent. Therefore, one can use the elliptic estimates for equations with merely BMO-coefficients proved by Dong-Kim \cite{DKell} (see also Disconzi-Kukavica \cite[Proposition 3.4]{DK2017}) to get the boundary control. See Section \ref{sect exist kk-prob} for details.
\end{rmk}

\subsection{Organization of the paper}

The presenting manuscript is organized as follows. In Section \ref{prelemma} we record the lemmas that are repeatedly used in the proof. Then we introduce the nonlinear $\kk$-approximation problem and do the div-curl-boundary estimates in Section \ref{sect nonlinearkk}. The non-weighted energy $E_\kk^{(1)}$ and $\sqrt{\kk}$-weighted boundary norms are treated in Section \ref{tgkk} and $\sqrt{\kk}$-weighted interior norms are treated in Section \ref{sect E3kk}. Then the uniform-in-$\kk$ estimates for the nonlinear $\kk$-approximate problem are closed in Section \ref{close}. In Section \ref{linearlwpkk} we solve the linearized approximate system by penalization method. In Section \ref{sect exist kk-prob} we use Picard iteration to solve the nonlinear $\kk$-approximate problem. Finally, the original system's local well-posedness and energy estimates are established in Section \ref{lwp1}.

\bigskip

The following notations will be frequently used in the rest of this manuscript. 
\\

\noindent\textbf{List of Notations: }
\begin{itemize}
\item $\Omega:=\T^2\times(0,1)$. $\Gamma:=\T^2\times\{1\}$ is the free boundary and $\Gamma_0:=\T^2\times\{0\}$ is the fixed bottom.
\item $\|\cdot\|_{s}$:  We denote $\|f\|_{s}: = \|f(t,\cdot)\|_{H^s(\Omega)}$ for any function $f(t,y)\text{ on }[0,T]\times\Omega$.
\item $\|\cdot\|_{L_t^2H_y^s}$: We denote $\|f\|_{L_t^2H_y^s}:=\|f(t,y)\|_{L^2(0,T;H^s(\Omega))}$ for any function $f(t,y)\text{ on }[0,t]\times\Omega$. 
\item $|\cdot|_{s}$:  We denote $|f|_{s}: = |f(t,\cdot)|_{H^s(\Gamma)}$ for any function $f(t,y)\text{ on }[0,T]\times\Gamma$.
\item $P(\cdot)$:  A generic non-decreasing continuous function in its arguments;
\item $\TP,\TL$: $\TP=\p_1,\p_2$ denotes the tangential derivative and $\TL:=\p_1^2+\p_2^2$ denotes the tangential Laplacian.
\item  For a smooth scalar function $f$, $\nabla^{\alpha}_a f:=a^{\mu\alpha}\p_{\mu} f$. Also, for a smooth vector field $X$,  $\dive_a X:=a^{\mu\alpha}\p_\mu X_{\alpha}$ and $(\curl_aX)_\lambda:=\epsilon_{\lambda\tau\alpha}a^{\mu\tau}\p_\mu X^{\alpha}$, where $\epsilon_{\lambda\tau\alpha}$ is the sign of the 3-permutation $(\lambda\tau\alpha)\in S_3.$
\item Let $X$ be given as above.  The flat divergence and curl are given as $\di X := \delta^{\mu\alpha} \p_\mu X_\alpha$, and $(\curl X)^\la:= \epsilon^{\lambda\tau\alpha}\p_\tau X_\alpha$. 
\end{itemize}

\section{Preliminary lemmas}\label{prelemma}

\subsection{Geometric identities}
The following geometric identities will be used repeatedly (and silently) throughout this manuscript. 

\begin{lem}\label{geometric}
Let $\nn$ be the unit outer normal to $\eta(\Gamma)$, namely $\Pi:=\nn\otimes\nn$, and $\mathcal{T},\mathcal{N}$ be the tangential and normal bundle of $\eta(\Gamma)$ respectively. Denote $\Pi:\mathcal{T}|_{\eta(\Gamma)}\to\mathcal{N}$ to be the canonical normal projection. Denote $\TP_A$ to be $\p_t$ or $\TP_1,\TP_2$. Then we have the identities
\begin{align}
\label{n} \nn := n\circ\eta=&\frac{a^T N}{|a^T N|},\\
\label{an}|a^T N|=&|(a^{31},a^{32},a^{33})|=\sqrt{g},\\
\label{Pi} \Pi^{\alpha}_{\lambda}=&\nn^{\alpha}\nn_{\lambda}=\delta^{\alpha}_{\lambda}-g^{kl}\TP_k\eta_\alpha\TP_l\eta_\lambda,\\
\label{Pipro} \Pi^{\alpha}_{\lambda}=&\Pi^{\alpha}_{\mu}\Pi^{\mu}_{\lambda},\\
\label{lapg0} -\Delta_g(\eta^{\alpha}|_{\Gamma})=&\mathcal{H}\circ\eta \nn^\alpha, \\
\label{lapg}\sqrt{g}\Delta_g\eta^{\alpha}=&\sqrt{g}g^{ij}\Pi^{\alpha}_{\lambda}\TP_i\TP_j\eta^{\lambda}=\sqrt{g}g^{ij}\TP_i\TP_j\eta^{\alpha}-\sqrt{g}g^{ij}g^{kl}\TP_k\eta^{\alpha}\TP_l\eta^{\mu}\TP_i\TP_j\eta_{\mu},\\
\label{tplapg}\TP_{A}(\sqrt{g}\Delta_g\eta^{\alpha})=&\TP_i\bigg(\sqrt{g}g^{ij}\Pi^{\alpha}_{\lambda}\TP_A\TP_j\eta^{\lambda}+\sqrt{g}(g^{ij}g^{kl}-g^{ik}g^{lj})\TP_j\eta^{\alpha}\TP_k\eta_{\lambda}\TP_A\TP_l\eta^{\lambda}\bigg),\\
\label{tpn} \TP_{A} \nn_{\mu} =& -g^{kl}\cp_k \TP_{A} \eta^\tau \nn_{\tau} \cp_l \eta_\mu,\\
\label{pt ggij} \p_t (\sqrt{g}g^{ij})=& \sqrt{g} (g^{ij}g^{kl}-2g^{lj}g^{ik})\TP_k v^\lambda \TP_l \eta_\lambda.
\end{align}
\end{lem}
\begin{proof}
See Disconzi-Kukavica \cite[Lemma 2.5]{DK2017}.
\end{proof}

\begin{rmk}
Recall that $g_{ij}=\TP_i\eta_{\mu}\TP_j\eta^{\mu}$ and $g=\det[g_{ij}]$ and $[g^{ij}]=[g_{ij}]^{-1}$. This means that $g_{ij},~g$ and $g^{ij}$ are rational functions of $\TP\eta$ and so is $\Pi$. 
\end{rmk}

\begin{nota} \label{Q notation} We shall use the notation $Q(\p\eta)$ and $Q(\TP\eta)$ to denote the rational functions of $\p\eta$ and $\TP\eta$, respectively. This $Q$ notation allows us to record error terms in a concise way and so it will be used frequently throughout the rest of this paper. For example, for any tangential derivative $\TP_A$, we have $\TP_AQ(\TP\eta)=\tilde{Q}^i_{\alpha}(\TP\eta)\TP_A\TP_i\eta^{\alpha}$ where the term $\tilde{Q}^i_{\alpha}(\TP\eta)$ is also a rational function of $\TP\eta$. For more details of such notation, we refer readers to \cite[Sect. 11]{coutand2007LWP} and \cite[Remark 2.4]{DK2017}.
\end{nota}

\subsection{Sobolev inequalities}

First, we list the Kato-Ponce estimates which will be used in div-curl estimates.

\begin{lem}[\textbf{Kato-Ponce type inequalities}]\label{katoponce}  Let $J=(I-\Delta)^{1/2},~s\geq 0$. Then the following estimates hold:

(1) $\forall s\geq 0$, we have 
\begin{equation}\label{product}
\begin{aligned}
\|J^s(fg)\|_{L^2}&\lesssim \|f\|_{W^{s,p_1}}\|g\|_{L^{p_2}}+\|f\|_{L^{q_1}}\|g\|_{W^{s,q_2}},\\
\|\p^s(fg)\|_{L^2}&\lesssim \|f\|_{\dot{W}^{s,p_1}}\|g\|_{L^{p_2}}+\|f\|_{L^{q_1}}\|g\|_{\dot{W}^{s,q_2}},
\end{aligned}
\end{equation}with $1/2=1/p_1+1/p_2=1/q_1+1/q_2$ and $2\leq p_1,q_2<\infty$;

(2) $\forall s\geq 1$, we have
\begin{equation}\label{kato3}
\|J^s(fg)-(J^sf)g-f(J^sg)\|_{L^p}\lesssim\|f\|_{W^{1,p_1}}\|g\|_{W^{s-1,p_2}}+\|f\|_{W^{s-1,q_1}}\|g\|_{W^{1,q_2}}
\end{equation} for all the $1<p<p_1,p_2,q_1,q_2<\infty$ with $1/p_1+1/p_2=1/q_1+1/q_2=1/p$.
\end{lem}

\begin{proof}
See Kato-Ponce \cite{kato1988commutator}.
\end{proof}

\begin{lem}[\textbf{Trace lemma for harmonic function}]\label{harmonictrace}
Suppose that $s\geq 0.5$ and $u$ solves the boundary-valued problem
\[
\begin{cases}
\Delta u=0~~~&\text{ in }\Omega,\\
u=g~~~&\text{ on }\Gamma
\end{cases}
\] where $g\in H^{s}(\Gamma)$. Then it holds that 
\[
|g|_{s}\lesssim\|u\|_{s+0.5}\lesssim|g|_{s}
\]
\end{lem}
\begin{proof}
The LHS follows from the standard Sobolev trace lemma, while the RHS is the property of Poisson integral, which can be found in \cite[Proposition 5.1.7]{taylorPDE1}.
\end{proof}

\subsection{Elliptic estimates}

First we illustrate the div-curl elliptic estimate.
\begin{lem} [\textbf{Hodge-type decomposition and the inverse theorem}]\label{hodge}
~

(1) Let $X$ be a smooth vector field and $s\geq 1$, then it holds that
\begin{equation}\label{divcurls}
\|X\|_s\lesssim\|X\|_0+\|\curl X\|_{s-1}+\|\dive X\|_{s-1}+|\TP X\cdot N|_{s-1.5}.
\end{equation}

(2) Let $\Omega\subseteq\R^3$ be a bounded $H^{k+1}$-domain with $k>1.5$. Given $\FF,G\in H^{l-1}(\Omega)$ with $\dive \FF=0$. Consider the equations
\begin{equation}\label{dc}
\curl X=\FF,~~\dive X=G ~~~\text{in }\Omega.
\end{equation} If $\FF$ satisfies $\int_{\gamma}\FF\cdot N\dS=0$ for each connected component $\gamma$ of $\p\Omega$ and $h\in H^{l-0.5}(\p\Omega)$ satisfies $\int_{\p\Omega}h\dS=\io G\dy$, then $\forall 1\leq l\leq k$, there exists a solution $X\in H^l(\Omega)$ to \eqref{dc} with boundary condition $X\cdot N|_{\p\Omega}=h$ such that
\begin{equation}\label{divcurli}
\|X\|_{H^l(\Omega)}\leq C(|\p\Omega|_{H^{k+0.5}})\left(\|\FF\|_{H^{l-1}(\Omega)}+\|G\|_{H^{l-1}(\Omega)}+|h|_{H^{l-0.5}(\p\Omega)}\right).
\end{equation}Such solution is unique if $\Omega$ is the disjoint union of simply connected open sets.
\end{lem}

\begin{proof}
(1) This follows from the well-known identity $-\Delta X=\curl\curl X-\nabla\dive X$ and integrating by parts. (2) This is the main result of Cheng-Shkoller \cite{hodgeinverse}.
\end{proof}

Next, the following $H^1$-elliptic estimates will be applied to control $\|\p_t^3 q\|_1$.
\begin{lem}[\textbf{Low regularity elliptic estimates}]\label{H1elliptic}
Assume $\BB^{\mu\nu}$ satisfies $\|\BB\|_{L^{\infty}}\leq K$ and the ellipticity $\BB^{\mu\nu}(x)\xi_{\mu}\xi_{\nu}\geq \frac{1}{K}|\xi|^2$ for all $x\in\Omega$ and $\xi\in\mathbb{R}^3$. Assume $W$ to be an $H^1$ solution to 
\begin{equation}\label{lowelliptic}
\begin{cases}
\p_{\nu}(\BB^{\mu\nu}\p_{\mu}W)=\dive \pi &\text{ in }\Omega \\
\BB^{\mu\nu}\p_{\nu}WN_{\mu}=h &\text{ on }\p\Omega,
\end{cases}
\end{equation} where $\pi,\dive \pi \in L^2(\Omega)$ and $h\in H^{-0.5}(\p\Omega)$ with the compatibility condition $$\int_{\p\Omega}(\pi\cdot N-h)dS=0.$$ If $\|\BB-I\|_{L^{\infty}}\leq\eps_0$ which is a sufficently small constant depending on $K$, then we have:
\begin{equation}
\|W-\overline{W}\|_{1}\lesssim\|\pi\|_{0}+|h-\pi\cdot N|_{-0.5},\text{ where }\overline{W}:=\frac{1}{|\Omega|}\int_{\Omega} W dy,
\end{equation}
\end{lem}
\begin{proof} See \cite[Lemma 3.2]{ignatova2016}.
\end{proof}

\subsection{Properties of tangential mollification}

Let $\zeta=\zeta(y_1,y_2)\in C_c^{\infty}(\R^2)$ be a standard cut-off function such that $\text{Spt }\zeta=\overline{B(0,1)}\subseteq\R^2,~~0\leq\zeta\leq 1$ and $\int_{\R^2}\zeta=1$. The corresponding dilation is $$\zeta_{\kk}(y_1,y_2)=\frac{1}{\kk^2}\zeta\left(\frac{y_1}{\kk},\frac{y_2}{\kk}\right),~~\kk>0.$$ Now we define
\begin{equation}\label{lkk0}
\lkk f(y_1,y_2,y_3):=\int_{\R^2}\zeta_{\kk}(y_1-z_1,y_2-z_2)f(z_1,z_2, y_3)\dz_1\dz_2.
\end{equation}

The following lemma records the basic properties of tangential smoothing.
\begin{lem}[\textbf{Regularity and Commutator estimates}]\label{tgsmooth} Let $f$ be a smooth function. For $\kk>0$, we have: (1) The following regularity estimates:
\begin{align}
\label{lkk11} \|\lkk f\|_s&\lesssim \|f\|_s,~~\forall s\geq 0;\\
\label{lkk1} |\lkk f|_s&\lesssim |f|_s,~~\forall s\geq -0.5;\\ 
\label{lkk2} |\TP\lkk f|_0&\lesssim \kk^{-s}|f|_{1-s}, ~~\forall s\in [0,1];\\  
\label{lkk3} |f-\lkk f|_{L^{\infty}}&\lesssim \sqrt{\kk}|\TP f|_{0.5}\\
\label{lkk33} |f-\lkk f|_{L^p}&\lesssim \kk|\TP f|_{L^p},\\
\label{lkk333} |f-\lkk f|_{L^2}&\lesssim \sqrt{\kk}|\TP^{\frac12}f|_0.
\end{align}

(2) Commutator estimates: Define the commutator $[\lkk,f]g:=\lkk(fg)-f\lkk(g)$. Then it satisfies
\begin{align}
\label{lkk4} |[\lkk,f]g|_0 &\lesssim|f|_{L^{\infty}}|g|_0,\\ 
\label{lkk5} |[\lkk,f]\TP g|_0 &\lesssim |f|_{W^{1,\infty}}|g|_0, \\ 
\label{lkk6} |[\lkk,f]\TP g|_{0.5}&\lesssim |f|_{W^{1,\infty}}|g|_{0.5}.
\end{align}
\end{lem}
\begin{proof}
We refer \cite{coutand2007LWP,gu2016construction,ZhangCRMHD2} for the proof except for \eqref{lkk333}. The inequality \eqref{lkk333} can be proved by integrating $\TP^{\frac12}$ by parts and then using Minkowski inequality
\begin{align*}
\left| f-\lkk f\right|_{0}=&\left|\int_{\R^2\cap B(0,\kk)}\zeta_{\kk}(z)\left(f(y-z)-f(y)\right)dz\right|_{L_y^2}\\
=&\kk \left|\int_{\R^2\cap B(0,\kk)}\TP^{\frac12}\zeta_{\kk}(z)\TP^{\frac12}f(y-\theta z)dz\right|_{L_y^2}\\
\lesssim&\kk \left|\TP^{\frac12} f\right|_{0}\left|\TP^{\frac12}\zeta_{\kk}\right|_{L^1(\R^2\cap B(0,\kk))}\lesssim \kk^2\left|\TP^{\frac12} f\right|_{0}\left|\TP^{\frac12}\zeta_{\kk}\right|_{L^2}.
\end{align*}
Then by interpolation, we have
\begin{align*}
\left|\TP^{\frac12}\zeta_{\kk}\right|_{L^2}\lesssim\left|\zeta_{\kk}\right|_{L^2}^{\frac12}\left|\TP\zeta_{\kk}\right|_{L^2}^{\frac12}\lesssim\left(\frac{1}{\kk}\left|\zeta\right|_{L^2}\right)^{\frac12}\left(\frac{1}{\kk^2}|\TP\zeta|_{L^2}\right)^{\frac12}\lesssim\kk^{-\frac32},
\end{align*}and thus
\begin{align*}
\left| f-\lkk f\right|_{0}\lesssim\sqrt{\kk}\left|\TP^{\frac12}f\right|_0.
\end{align*}
\end{proof}


\section{The nonlinear approximate system}\label{sect nonlinearkk}
For $\kk>0$, we denote $\lkk$ to be the standard mollifier on $\R^2$ defined as \eqref{lkk0}. Define $\ek$ to be the smoothed version of $\eta$ solved by the following elliptic system
\begin{equation}\label{ekk}
\begin{cases}
-\Delta\ek=-\Delta\eta,~~~&\text{in}\,\,\Omega,\\
\ek=\lkk^2\eta~~~&\text{ on }\p\Omega,
\end{cases}
\end{equation} and $\ak:=[\p\ek]^{-1}$, $\Jk:=\det[\p\ek]$, $\Ak:=\Jk\ak$ and $\nnk = n\circ \ek$. Now we introduce the nonlinear $\kk$-approximation system of \eqref{MHDL}.
\begin{equation}\label{MHDLkk}
\begin{cases}
\p_t\eta=v~~~& \text{in}~[0,T]\times\Omega;\\
\p_tv-\bp^2\eta+\pAk q=0~~~& \text{in}~[0,T]\times\Omega;\\
\diva v=0,&\text{in}~[0,T]\times\Omega;\\
\dive b_0=0~~~&\text{in}~\{t=0\}\times \Omega ;\\
v^3=b_0^3=0~~~&\text{on}~\Gamma_0;\\
\Ak^{3\alpha}q = -\sigma \sqrt{g}(\Delta_g \eta\cdot \nnk )\nnk^\alpha +\kk\left((1-\TL) (v\cdot \nnk)\right)\nnk^\alpha~~~&\text{on}~\Gamma;\\
b_0^3=0 ~~~&\text{on}~\Gamma,\\
(\eta,v)=(Id,v_0)~~~&\text{in}~\{t=0\}{\times}\overline{\Omega}.
\end{cases}
\end{equation}Here $\TL:=\TP_1^2+\TP_2^2$ is the (flat) tangential Laplacian. The re-formulated boundary condition on $\Gamma$ is used here since we find that it is more convenient to apply when studying \eqref{MHDLkk}. We remark here that in absence of $\kk\left((1-\TL) (v\cdot \nnk)\right)\nnk^\alpha$ the boundary condition is just a reformulation of 
\begin{align}\Ak^{3\alpha}q=-\sigma\sqrt{g}\Delta_g \eta^{\alpha}.
\label{original}
\end{align}
Invoking \eqref{n} and the identity $\Jk |\ak^T N|=\sqrt{\tilde{g}}$, where $\tilde{g} = g(\ek)$, we have
\begin{equation}
 \Ak^{3\alpha}/\sqrt{\tilde{g}} = \Jk\ak^{\mu\alpha}N_\mu / \Jk |\ak^T N|=\nnk^{\alpha}, \label{n tilde id}
 \end{equation}
 and so \eqref{original} becomes
\begin{align*}
q\nnk^\alpha = -\sigma\frac{\sqrt{g}}{\sqrt{\tilde{g}}}\Delta_g \eta^\alpha.
\end{align*}
Also, because $\nnk\cdot \nnk =1$ (Euclidean dot product), we obtain
\begin{equation*}
q\nnk^\alpha = q(\nnk\cdot\nnk)\nnk^\alpha = -\sigma\frac{\sqrt{g}}{\sqrt{\tilde{g}}}(\Delta_g \eta\cdot \nnk)\nnk^\alpha.
\end{equation*}
In view of \eqref{n tilde id}, this is equivalent to 
\begin{align*}
\Ak^{3\alpha}q = -\sigma \sqrt{g}(\Delta_g \eta\cdot \nnk )\nnk^\alpha 
\end{align*} 
By adding the artificial viscosity term $\kk\left((1-\TL) (v\cdot \nnk)\right)\nnk^\alpha$ on the RHS, the boundary condition of \eqref{MHDLkk} is then achieved:
\begin{equation}
\Ak^{3\alpha}q = -\sigma \sqrt{g}(\Delta_g \eta\cdot \nnk )\nnk^\alpha +\kk\left((1-\TL) (v\cdot \nnk)\right)\nnk^\alpha.
\label{BC}
\end{equation}
In addition, since $\Ak^{3\alpha} \nnk_\alpha=\sqrt{\tilde{g}}$, \eqref{BC} can be written as
\begin{align}
\sqrt{\tilde{g}}q = -\sigma \sqrt{g}(\Delta_g \eta\cdot \nnk ) +\kk(1-\TL) (v\cdot \nnk).
\label{BC1}
\end{align}
Despite being equivalent to each other, \eqref{BC} and \eqref{BC1} will be adapted to different scenarios. In fact, \eqref{BC} will be used  in Section \ref{tgkk} for the tangential energy estimate, whereas we find \eqref{BC1} more convenient when dealing with the boundary estimate in Section \ref{sect bdy estimate 3.3}.

 Let's state the main theorem.  Our goal is to derive the uniform-in-$\kk$ a priori estimates for the nonlinear approximation system \eqref{MHDLkk}. 
\begin{prop}\label{nonlinearkk}
Given the divergence-free vector fields $v_0\in H^{4.5}(\Omega)\cap H^{5}(\Gamma)$ and $b_0\in H^{4.5}(\Omega)$ satisfying $b_0^3=0$ on $\Gamma\cup \Gamma_0$, there exists some $T_1>0$ independent of $\kk>0$, such that the solution $(\eta(\kk),v(\kk),q(\kk))$ to \eqref{MHDLkk} satisfies the following uniform-in-$\kk$ estimates
\begin{equation}\label{energykk}
\sup_{0\leq t\leq T_1}E_\kk(t)\leq \mathcal{C},
\end{equation}
where $\mathcal{C}$ is a constant depends on $\|v_0\|_{4.5}, \|b_0\|_{4.5}, |v_0|_5$, provided the following a priori assumption hold for all $t\in[0,T_1]$
\begin{align}
\label{smallak} \|\Jk(t)-1\|_{3.5}+\|Id-\Ak(t)\|_{3.5}+\|Id-\Ak^T\Ak(t)\|_{3.5}\leq \eps.
\end{align} Here the energy functional $E_\kk$ of \eqref{MHDLkk} is defined to be
\begin{equation}\label{Ekk}
E_\kk=E_{\kk}^{(1)}+E_{\kk}^{(2)}+E_{\kk}^{(3)}, 
\end{equation}
where
\begin{equation}
\begin{aligned}
E_\kk^{(1)}:=&\left\|\eta(\kk)\right\|_{4.5}^2+\sum_{j=0}^{3}\left\|\left(\p_t^{j}v(\kk),\p_t^{j}\bp\eta(\kk)\right)\right\|_{4.5-j}^2+\left\|\left(\p_t^4 v(\kk),\p_t^4 \bp\eta(\kk)\right)\right\|_{0}^2\\
&+\sum_{j=0}^{3}\left|\TP\left(\Pi\TP^{3-j}\p_t^j v(\kk)\right)\right|_0^2+\left|\TP\big(\Pi\TP^3 \bp\eta(\kk)\big)\right|_0^2, 
\end{aligned}
\end{equation} and 
\begin{equation}
\begin{aligned}
E_{\kk}^{(2)}:= &\sigma\int_0^T\bigg(\sum_{j=1}^{4}\left|\sqrt{\kk}\p_t^j v(\kk)\cdot\nnk(\kk)\right|_{5-j}^2+\left|\sqrt{\kk}\bp v(\kk) \cdot \nnk(\kk)\right|_4^2\bigg)\dt,\\
E_{\kk}^{(3)}:=&\sum_{k=0}^4\int_0^T \Big (\big\|\sqrt{\kk} \p_t^{k}v(\kk)\big\|_{5.5-k}^2+\big\|\sqrt{\kk} \p_t^k \bp \eta(\kk)\big\|_{5.5-k}^2\Big)\dt.
\end{aligned}
\end{equation}
\end{prop}

The proof of this theorem is organized as follows: The rest of this section is devoted to the estimate of the full Sobolev of the pressure $q$, and the velocity field $v$ and the magnetic field $\bp\eta$ as well as their time derivatives. In Section \ref{tgkk} we study the tangential energy estimate of $v$ and $\bp\eta$, which ties to the control of the boundary Sobolev norms of the time derivatives of $v$ and $\bp\eta$ that arose from the div-curl estimate. The terms in the weighted boundary top order energy $E_{\kk}^{(2)}$ are created during this process owing to the artificial viscosity.    Lastly, we investigate the weighted top order energy functional $E_{\kk}^{(3)}$ in Section \ref{sect E3kk}. We need this energy to control the error terms generated by the artificial viscosity on the boundary when all derivatives land on the Eulerian normal $\nnk$. 

Let $T\leq T_\kk$, where $[0,T_\kk]$ is the interval of existence for the solution of the $\kk$-problem for some fixed $\kk$. The key step for showing \eqref{energykk} is to prove
\begin{align}
\sup_{0\leq t\leq T}E_\kk(t) \leq \PP_0 +C(\eps)\sup_{0\leq t\leq T}E_\kk(t)+(\sup_{0\leq t\leq T}\PP)\int_0^T \PP,
\label{gronwall}
\end{align}
holds true independent of $\kk$, 
where $$\PP = P(E_{\kk}(t)),$$
and  $$\PP_0=P(E_{\kk}(0), \|q(0)\|_{4.5}, \|q_t(0)\|_{3.5}, \|q_{tt}(0)\|_{2.5}),$$ 
with $P$ denoting a non-decreasing continuous function in its arguments, and $C(\eps)$ is a constant that is proportional to $\eps$ (and thus $C(\eps)\ll 1$ whenever $\eps\ll1$). In Section \ref{close}, we are going to show that $\PP_0$ can in fact be controlled by $\mathcal{C}$, i.e., the RHS of \eqref{energykk}.   For the simplicity of notations, we will omit the $\kk$ in $(\eta(\kk),v(\kk),q(\kk))$ in the rest of this paper. Also, we may assume that $\sup_{0\leq t\leq T}E_{\kk}(t)=E(T)$, and this allows us to drop $\sup_{0\leq t\leq T}$ in \eqref{gronwall}. In other words, we only need to show
\begin{align}
E_\kk(T) \leq \PP_0 +C(\eps)E_\kk(T)+\PP\int_0^T \PP,
\label{Gronwall}
\end{align}

 Before going to the proof, we need the following preliminary estimates for $\ek$ and its derivatives.
\begin{lem}\label{bpek}
We have
\begin{align}
\label{eta4} \|\ek\|_{4.5}\lesssim&\|\eta\|_{4.5}\\
\label{bk4} \|\bp\ek\|_{4.5}\lesssim& P(\|b_0\|_{4.5},\|\bp\eta\|_{4.5},\|\eta\|_{4.5}).
\end{align}
\end{lem}
\begin{proof}
\eqref{eta4} follows from standard elliptic estimates and property of mollification. To prove \eqref{bk4}, we take $\bp$ in \eqref{ekk}
\begin{equation}\label{bkk}
\begin{cases}
-\Delta(\bp\ek)=-\Delta(\bp\eta)-\left[\bp,\Delta\right]\eta+[\bp,\Delta]\ek~~~&\text{ in }\Omega,\\
\ek=\lkk^2\eta~~~&\text{ on }\p\Omega,
\end{cases}
\end{equation} and standard elliptic estimates yields that
\begin{equation}
\begin{aligned}
\|\bp\ek\|_{4.5}\lesssim&\left\|-\Delta(\bp\eta)-\left[\bp,\Delta\right]\eta+[\bp,\Delta]\ek\right\|_{2.5}\\
&+\left|\lkk^2(\bp\eta)\right|_{4}+\left|\left[\bp,\lkk^2\right]\eta\right|_{4}\\
\lesssim&\|\bp\eta\|_{4.5}+\|b_0|_{4.5}\|\eta\|_{4.5}\\
&+\sum_{l=1}^2\left|\left[\lkk^2,b_0^l\right]\TP_l\eta\right|_{1/2}+\left|\left[\lkk^2,b_0^l\right]\TP_l\TP^{7/2}\eta\right|_{1/2}+\left|\left[\TP^{7/2},[\lkk^2,b_0^l]\TP_l\right]\eta\right|_{1/2}\\
\lesssim& P(\|b_0\|_{4.5},\|\bp\eta\|_{4.5},\|\eta\|_{4.5}),
\end{aligned}
\end{equation}
where commutator estimates in Lemma \ref{tgsmooth} is also used.
\end{proof}
The next lemma concerns some auxiliary results which come in handy when studying Proposition \ref{nonlinearkk}.

\begin{lem} \label{small quantities}
Assume that $\|\eta\|_{4.5},\|v\|_{4.5}\leq N_0$, where $N_0\geq 1$. If $T\leq \eps/P(N_0)$ for some fixed polynomial $P$ and $\eta, v$ is defined on $[0,T]$, then the following inequality holds for $t\in [0,T]$:
\begin{align} 
\|\Ak^{\mu\alpha}-\delta^{\mu\alpha}\|_{3.5}\lesssim \eps,\label{compars delta and A}\\
\|\ak^{\mu\alpha}-\delta^{\mu\alpha}\|_{3.5}\lesssim \eps,\quad \|a^{\mu\alpha}-\delta^{\mu\alpha}\|_{3.5}\lesssim \eps,\label{compars delta and a}\\
\forall 0\leq s\leq 1.5,~~|\cp^s(\nnk-N)|_{L^\infty(\Gamma)} \lesssim \eps,\quad |\cp^s(\nn-N)|_{L^\infty(\Gamma)} \lesssim \eps,
\label{compars N and nn}\\
|\nnk-N|_{3} \lesssim \eps,\quad |\nn-N|_{3}\lesssim \eps, 
\label{comars N and nn in H^3}\\
|\delta^{ij}-\sqrt{g}g^{ij}|_3 \leq \eps,
\label{compars delta and ggij}\\
|\cp\eta\cdot n|_3\leq \eps,\quad |\cp^2\eta|_2\leq\eps.
\label{compar cp eta normal}
\end{align}
\end{lem}
\begin{proof}
Since 
\begin{equation}
A^{1\alpha} =\epsilon^{\alpha\lambda \tau} \p_2\eta_\lambda\p_3 \eta_\tau,\quad A^{2\alpha} =-\epsilon^{\alpha\lambda \tau} \p_1\eta_\lambda\p_3 \eta_\tau,\quad A^{3\alpha}=\epsilon^{\alpha\lambda \tau} \p_1\eta_\lambda\p_2 \eta_\tau, 
\label{A's component}
\end{equation}
and thus
\begin{equation}
\Ak^{1\alpha} =\epsilon^{\alpha\lambda \tau} \p_2\ek_\lambda\p_3 \ek_\tau,\quad \Ak^{2\alpha} =-\epsilon^{\alpha\lambda \tau} \p_1\ek_\lambda\p_3 \ek_\tau,\quad \Ak^{3\alpha}=\epsilon^{\alpha\lambda \tau} \p_1\ek_\lambda\p_2 \ek_\tau, 
\label{Ak's component}
\end{equation}
where $\epsilon^{\alpha\lambda\tau}$ is the fully antisymmnetric symbol with $\epsilon^{123}=1$, we have $|\Ak-\delta| \leq \int_0^t|\p_t Q(\p \ek)|\leq \int_0^t |Q(\p \ek)\p \vk|$, where $Q$ is defined in Notation \ref{Q notation}.  Then \eqref{compars delta and A} follows from \eqref{lkk11} and the Kato-Ponce inequality (Lemma \ref{katoponce}). Invoking the a priori assumption $\|\tilde{J}-1\|_{3.5}\lesssim \eps$, both inequalities in \eqref{compars delta and a} are proved similarly. 

In addition, for \eqref{compars N and nn}, it suffices to prove the first inequality. Since $\nnk|_{t=0} = N$ and $\nnk=Q(\cp \ek)$, for each fixed $0\leq s\leq 1.5$, there holds
\begin{equation*}
|\cp^s(\nnk-N)|_{L^\infty(\Gamma)} \leq \int_0^t |\cp^s\p_t\nnk|_{L^\infty(\Gamma)} = \int_0^t |\cp^s(Q(\TP \ek)\TP\vk)|_{L^\infty(\Gamma)}\lesssim \int_0^t |Q(\cp \ek)\cp \vk|_{3},
\end{equation*}
by the Sobolev embedding. Now, the trace lemma and the Kato-Ponce inequality yield $\int_0^t |Q(\cp \ek)\cp \vk|_{3} \lesssim \int_0^t \|Q(\cp \ek)\cp \vk\|_{3.5}\leq \int_0^t P(N_0)$ and so \eqref{compars N and nn} follows. 

Moreover, we have 
\begin{align*}
|\nnk-N|_3 \leq \int_0^t |Q(\cp\ek)\p\vk|_3 \lesssim \int_0^t \|Q(\cp \ek)\cp \vk\|_{3.5},
\end{align*}
which verifies \eqref{comars N and nn in H^3}. 

In addition, owing to the fact that $(\delta^{ij}-\sqrt{g}g^{ij})|_{t=0}=0$ and the identity \eqref{pt ggij}, there holds
$$|\delta^{ij}-\sqrt{g}g^{ij}|_3\leq \int_0^t |\p_t (\sqrt{g}g^{ij})|_3=\int_0^T \|Q(\p \eta)\TP v\|_{3.5},$$
which yields \eqref{compars delta and ggij}. Finally, a similar proof yields \eqref{compar cp eta normal} since $\cp \eta\cdot \nnk |_{t=0} = \cp\eta^3|_{t=0}=0$ and $\cp^2\eta |_{t=0}=0$.
\end{proof}

\begin{rmk} The inequalities in Lemma \ref{small quantities} can be viewed as an extended list of the a priori assumptions. Moreover, \eqref{smallak} is in fact a direct consequence of \eqref{compars delta and a} and \eqref{compars delta and A}.
\end{rmk}

We also need the following corollary of Lemma \ref{tgsmooth} that ``extends" \eqref{lkk3} and \eqref{lkk333} to the interior of $\Omega$ when applied to $\eta$ and its time derivatives. 

\begin{lem} \label{ext Lemma 3.2 to the interior}
Let $k=0,\cdots, 4$. Then
\begin{align}
\|\p \p_t^k (\ek-\eta)\|_{0}  \lesssim \|\sqrt{\kk}\p_t^k \eta\|_{1.5}.
\label{lkk 333 int}
\end{align}
Further, for $\ell=0,1,2$, there holds
\begin{align}
\|\p \p_t^\ell (\ek-\eta)\|_{L^\infty} \lesssim \|\sqrt{\kk} \p_t^\ell \eta\|_{3.5}. 
\label{lkk 3 int}
\end{align}
\end{lem}
\begin{proof}
The definition of $\ek$  in \eqref{ekk} implies that $\ek-\eta$ together its time derivatives is a harmonic function in $\Omega$. So we invoke Lemma \ref{harmonictrace} to get
\begin{align*}
\|\p \p_t^k (\ek-\eta)\|_{0} \lesssim\|\p_t^k (\ek-\eta)\|_{1} \lesssim |\p_t^k (\lkk^2 \eta -\lkk\eta)|_{0.5}+|\p_t^k (\lkk \eta-\eta)|_{0.5} \lesssim |\p_t^k(\lkk \eta-\eta)|_{0.5}
\end{align*}
where $|\p_t^k (\lkk^2 \eta -\lkk\eta)|_{0.5}\lesssim |\p_t^k (\lkk \eta-\eta)|_{0.5}$ since $\p_t$ and $\lkk$ commute, and $$|\p_t^k(\lkk\eta-\eta)|_{0.5} \lesssim |\sqrt{\kk} \p_t^k \eta|_1$$
in light of \eqref{lkk333}. This, together with the trace lemma give \eqref{lkk 333 int}.  Moreover, \eqref{lkk 3 int} follow from \eqref{lkk 333 int} and the Sobolev embedding. 
\end{proof}

\begin{rmk} It is possible to prove an improved estimate for \eqref{lkk 3 int}, i.e., 
\begin{align}
\|\p \p_t^\ell (\ek-\eta)\|_{L^\infty} \lesssim \|\sqrt{\kk} \p_t^\ell \eta\|_{2.5}. 
\label{improve lkk 3 int}
\end{align}
This can be done by adapting the following
Schauder estimate for div-curl systems: Let $X$ be a smooth vector field on $\overline{\Omega}$. For fixed $0<\delta<\frac{1}{2}$, we have
\begin{equation}
\|\p X\|_{C^{0,\delta}(\Omega)} \lesssim \|\di X\|_{C^{0,\delta}(\Omega)}+\|\curl X\|_{C^{0,\delta}(\Omega)}+\|X\|_{C^{0.5,\delta}(\partial\Omega)}+\|X\|_2. 
\label{schauder div-curl}
\end{equation}
This inequality in fact reduces to the one in \cite[Lemma 8.2]{DLMS} in the absence of the boundary term.  Thus, in view of \eqref{ekk}, we have
\begin{equation*}
\|\p \p_t^\ell (\ek-\eta)\|_{L^\infty} \leq \|\p \p_t^\ell (\ek-\eta)\|_{C^{0,\delta}(\Omega)} \lesssim |\p_t^\ell (\lkk\eta-\eta)|_{C^{0.5,\delta}(\Gamma)}+\|\p_t^\ell (\ek-\eta)\|_2, 
\end{equation*}
where the last term on the RHS is $\lesssim \|\sqrt{\kk}\p_t^\ell\eta\|_{2.5}$.
In addition to this,  \eqref{lkk3} and the Sobolev embedding suggest that $|\p_t^\ell (\lkk\eta-\eta)|_{C^{0.5,\delta}(\Gamma)} \lesssim |\sqrt{\kk}\p_t^\ell \eta|_{2+\delta}$, and so \eqref{improve lkk 3 int} follows after using the trace lemma. 

Nevertheless, we mention here that \eqref{improve lkk 3 int} will not be applied in the rest of this manuscript. Despite not being sharp,  \eqref{lkk 3 int} turns out to be sufficient. 
\end{rmk}

Finally, we state the following two lemmas that concern the boundary elliptic estimates of $\sqrt{\kk}\ek$ and $\kk \bp \ek$. These lemmas will be adapted to control the boundary error terms generated when derivatives land on the Eulerian normal $\nnk$. 

\begin{lem}[Improved Boundary Regularity] \label{super control}
Let $\MM_0 = P( \|v_0\|_{4.5},  \sqrt{\kk} \|v_0\|_{8.5}, \sqrt{\kk} \|b_0\|_{8.5}, \sqrt{\kk}|v_0|_{10})$. Then
\begin{align}
|\sqrt{\kk} \eta|_5^2 \leq& \MM_0 +C(\eps)E_\kk(T)+\PP\int_0^T\PP,
\label{super eta}\\
\int_0^T|\sqrt{\kk}  v |_5^2 \leq&   \MM_0 +C(\eps)E_\kk(T)+\PP\int_0^T\PP,\label{super v}\\
\int_0^T |\sqrt{\kk} \bp\eta|_5^2\leq& \MM_0 +C(\eps)E_\kk(T)+\PP\int_0^T\PP. 
\label{super bp eta}
\end{align}
\end{lem}
\noindent\textbf{Discussion of the proof:} By Jensen's inequality and $\eta(T)=\text{Id}+\int_0^Tv(t)\dt$, we know \[
|\sqrt{\kk} \eta|_5^2 \lesssim \MM_0+T\int_0^T|\sqrt{\kk}  v(t)|_5^2\dt,
\]and thus it suffices to prove \eqref{super v} and \eqref{super bp eta}. Indeed, one has to establish the energy estimates for $\|\sqrt{\kk}v\|_{L_t^2H_y^{5.5}}$ and $\|\sqrt{\kk}\bp\eta\|_{L_t^2H_y^{5.5}}$ and then use the trace lemma to derive  \eqref{super v} and \eqref{super bp eta}. The reason is that the boundary condition \eqref{BC1} only gives us the information in the (smoothed) Eulerian normal direction, but \eqref{super v} and \eqref{super bp eta} requires the control of all components. In view of Lemma \ref{hodge}, we need the control of 
\begin{align*}
\|\sqrt{\kk} \di v\|_{L_t^2H^{4.5}},\quad \|\sqrt{\kk}\di \bp\eta\|_{L_t^2H^{4.5}},\\
\|\sqrt{\kk} \curl v\|_{L_t^2H^{4.5}},\quad \|\sqrt{\kk}\curl \bp\eta\|_{L_t^2H^{4.5}},\\
|\sqrt{\kk} v^3|_{L_t^2H^{5}}, \quad |\sqrt{\kk}\bp\eta^3|_{L_t^2H^{5}},
\end{align*}
corresponding to the divergence, curl and normal trace, respectively. The proof of the div-curl control is parallel to Section \ref{divcurlkk} and we postpone the details to Section \ref{E3kkk}. Here, we shall sketch the control the normal traces.  

To control $\int_0^T|\sqrt{\kk} v^3|_{5}^2$, it suffices to control its corresponding Eulerian normal trace $\int_0^T|\sqrt{\kk} v\cdot\nnk|_{5}^2$ by adopting the following perturbation arguments: Because the difference between $N=(0,0,1)$ and $\nnk$ is sufficiently small in $L^\infty$ thanks to \eqref{compars N and nn} in Lemma \ref{small quantities}, we have
\begin{equation}
\int_0^T|\sqrt{\kk}(\TP^5 v^3-\TP^5v\cdot\nnk)|_{0}^2\lesssim|N-\nnk|_{L^{\infty}}^2\int_0^T|\sqrt{\kk}  v |_5^2\lesssim\eps^2\int_0^T|\sqrt{\kk}  v |_5^2.
\end{equation}
To control the Eulerian normal trace, we need to prove
\begin{align}\label{normal trace v}
\int_0^T|\sqrt{\kk}  \TP^5 v\cdot\nnk |_0^2 \leq   \MM_0 +C(\eps)E_\kk(T)+\PP\int_0^T\PP.
\end{align}
This is studied \cite[Lemma 12.6]{coutand2007LWP}. The proof is extremely technical so we will not go into the details. But we remark here that, the conclusions of this lemma essentially come from the energy estimates of the viscous surface tension equation \eqref{BC1}. In fact, one can differentiate \eqref{BC1} with $\TP^4\p_t$ and test the equation by $\TP^4 v\cdot\nnk$. Standard energy estimates give us the bounds for $\sigma|\TP^5 v\cdot\nnk|_{L_t^{\infty}L_y^2}^2$ and $|\sqrt{\kk}\TP^5 v\cdot\nnk|_{L_t^2L_y^2}^2$. The proof of \eqref{super bp eta} can be done similarly thanks to the fact that $\p_t \eta$ and $\bp\eta$ are of the same regularity. Because of this, 
we study the $\TP^4\bp$-differentiated \eqref{BC1} tested with $(\TP^4 \bp \eta)\cdot\nn$ to obtain
\begin{align}\label{normal trace bp eta}
\int_0^T|\sqrt{\kk}  \TP^5 \bp\eta\cdot\nnk |_0^2 \leq   \MM_0 +C(\eps)E_\kk(T)+\PP\int_0^T\PP.
\end{align}

\begin{rmk}In the control of the normal trace of $v$, we need the control of $\|q\|_{4.5}^2$ (which is given by \eqref{ellq}) and $\|\eta\|_{4.5}^2$. Also, we mention here that the following highest order term will be generated during the testing process
$$
\kk\int_0^T\ig (v^j \TP^5\TP_j \ek\cdot \nnk)(\TP^5 \eta\cdot \nnk),
$$
which cannot be controlled directly. Instead, we need to commute one tangential smooth operator $\Lambda_\kk$ from $\ek$ to $\eta$ and hence create a positive term after pulling $\TP_j$ out. In fact, this is the \textit{only} place that this operation is required. For the normal trace of magnetic field, we need the control of $\|\sqrt{\kk}\eta\|_{5.5}, \|\bp \eta\|_{4.5}^2$ and $\|\bp q\|_{3.5}^2$, where $\|\bp q\|_{3.5}^2 \lesssim \|b_0\|_{3.5}^2 \|q\|_{4.5}^2$ in light of the Kato-Ponce inequality \eqref{product}. 
\end{rmk}



\subsection{Elliptic estimates of pressure}\label{ellkk}

We prove the following proposition in this section.
\begin{prop}\label{qelliptic}
The pressure $q$ in \eqref{MHDLkk} and its time derivatives satisfy the following estimates
\begin{equation}\label{ellq}
\|q\|_{4.5}+\|\p_tq\|_{3.5}+\|\p_t^2 q\|_{2.5}+\|\p_t^3q\|_1\lesssim\PP.
\end{equation}
\end{prop}

First, we give control of the pressure $q$. Taking $\divA$ in the second equation of \eqref{MHDLkk} 
 we get the following elliptic system for $q$:
\begin{align*}
-\lapAk q:=&-\di_{\Ak} (\nab_{\Ak} q)=\left[\divA,\p_t\right]v+\left[\divA,\bp\right]\bp\eta+\bp\divA\left(\bp\eta\right)\\
=&-\p_t\Ak^{\mu\alpha}\p_{\mu}v_{\alpha}-(\bp \Ak^{\mu\alpha})\p_\mu \bp\eta_\alpha+\Ak^{\mu\alpha}(\p_\mu b_0\cdot \p)\bp\eta_\alpha\\
&+\bp\underbrace{\dive_a\left(\bp\eta\right)}_{=\dive b_0=0}+\bp\left((\Ak^{\mu\alpha}-a^{\mu\alpha})\p_{\mu}\bp\eta_{\alpha}\right),
\end{align*}
and thus
\begin{equation}\label{qeq}
\begin{aligned}
-\Delta q=:&-\di (\p q)=-\p_{\nu}\left((\delta^{\mu\nu}-\Ak^{\mu\alpha}\Ak^{\nu\alpha})\p_{\mu}q\right)-\p_t\Ak^{\mu\alpha}\p_{\mu}v_{\alpha}-(\bp \Ak^{\mu\alpha})\p_\mu \bp\eta_\alpha\\
&+\Ak^{\mu\alpha}(\p_\mu b_0\cdot \p)\bp\eta_\alpha+\bp\left((\Ak^{\mu\alpha}-a^{\mu\alpha})\p_{\mu}\bp\eta_{\alpha}\right).
\end{aligned}
\end{equation}
We impose Neumann boundary condition to \eqref{qeq} by contracting $\Ak^{\mu\alpha}N_{\mu}=\Ak^{3\alpha}$ with the second equation of \eqref{MHDLkk}
\begin{equation}\label{qbdry}
\frac{\p q}{\p N}=(\delta^{\mu 3}-\Ak^{\mu\alpha}\Ak^{3\alpha})\p_{\mu}q-\Ak^{3\alpha}\p_tv_{\alpha}+\Ak^{3\alpha}\bp^2\eta_{\alpha},\quad \text{on}\,\,\Gamma.
\end{equation}
Also, since $\Ak^{31}=\Ak^{32}=0$, $\Ak^{33}=1$, $v_3=0$, and $b_0^3=0$ implies $\bp \eta_3=b_0^j \TP_j \eta_3=0$  on $\Gamma_0$, \eqref{qbdry} yields
\begin{equation}
\frac{\p q}{\p N}=0,\quad \text{on}\,\,\,\Gamma_0. \label{qbdryGamma0}
\end{equation}
By the standard elliptic estimates, we have
\[
\|q\|_{4.5}\lesssim\|\text{RHS of }\eqref{qeq}\|_{2.5}+|\text{RHS of }\eqref{qbdry}|_{3}+|q|_0.
\]
Here, $|q|_0$ can be directly bounded by invoking the boundary condition of $q$, i.e., 
\begin{equation}
q = -\sigma \frac{\sqrt{g}}{\sqrt{\gk}}(\Delta_g \eta\cdot \nnk ) +\kk\frac{1}{\sqrt{\gk}}(1-\TL) (v\cdot \nnk), \label{BC2}
\end{equation}
and thus
\begin{equation}\label{q1}
|q|_0  \lesssim\PP.
\end{equation}
Since for smooth functions $f$ and $g$, \eqref{product} implies that
\begin{align*}
\|fg\|_{2.5} \lesssim \|f\|_{2.5}\|g\|_{2.5},
\end{align*}
invoking the a priori assumption \eqref{smallak}, we have
\begin{equation}\label{q2}
\begin{aligned}
\|\text{RHS of }\eqref{qeq}\|_{2.5}\lesssim&\eps\|q\|_{4.5}+P(\|\eta\|_{4.5},\|b_0\|_{3.5})(\|\p v\|_{2.5}^2+\|\p\bp\eta\|_{2.5})\\
&+\|\Ak-a\|_{3.5}\|b_0\|_{2.5}\|\bp\eta\|_{3.5}+\|\Ak-a\|_{2.5}\|b_0\|_{2.5}\|\bp\eta\|_{4.5}\\
\lesssim&\eps\|q\|_{4.5}+P(\|b_0\|_{4.5},\|\bp\eta\|_{4.5},\|\eta\|_{3.5},\|v\|_{3.5})\\
\lesssim&\eps\|q\|_{4.5}+\PP,
\end{aligned}
\end{equation}
and
\begin{equation}\label{q3}
|\text{RHS of }\eqref{qbdry}|_{3}\lesssim\eps\|q\|_{4.5}+P(\|\eta\|_{4.5})\left(\|\p_tv\|_{3.5}+\|b_0\|_{3.5}\|\bp\eta\|_{4.5}\right)\lesssim\eps\|q\|_{4.5}+\PP.
\end{equation}
Summing up \eqref{q1}-\eqref{q3} and choosing $\eps>0$ sufficiently small, we get the estimates of $q$
\begin{equation}\label{q4}
\|q\|_{4.5}\lesssim\PP.
\end{equation}

Next we take $\p_t$ in \eqref{qeq}-\eqref{qbdry} to get the equations of $\p_t q$:
\begin{equation}\label{qteq}
\begin{aligned}
-\Delta\p_tq=&-\p_{\nu}\left((\delta^{\mu\nu}-\Ak^{\mu\alpha}\Ak^{\nu\alpha})\p_{\mu}\p_tq\right)-\p_{\nu}\left((\delta^{\mu\nu}-\p_t(\Ak^{\mu\alpha}\Ak^{\nu\alpha}))\p_{\mu}q\right)\\
&-\p_t^2\Ak^{\mu\alpha}\p_{\mu}v_{\alpha}-\p_t\Ak^{\mu\alpha}\p_t\p_{\mu}v_{\alpha}+\p_t\left(\Ak^{\mu\alpha}(\p_\mu b_0\cdot \p)\bp\eta_\alpha-(\bp \Ak^{\mu\alpha})\p_\mu\bp\eta_\alpha\right)\\
&+\bp\left((\p_t\Ak-\p_ta)\p(\bp\eta)+(\Ak-a)\p(\bp v)\right),
\end{aligned}
\end{equation}with Neumann boundary condition
\begin{equation}\label{qtbdry}
\begin{aligned}
\frac{\p\p_tq}{\p N}=&~(\delta^{\mu3}-\Ak^{\mu\alpha}\Ak^{3\alpha})\p_{\mu}\p_t q-\p_t(\Ak^{\mu\alpha}\Ak^{3\alpha})\p_{\mu}Q\\
&-\Ak^{3\alpha}(\p_t^2v_{\alpha}-\bp^2v_{\alpha})-\p_t\Ak^{3\alpha}(\p_tv-\bp^2\eta)_{\alpha},\quad \text{on}\,\,\Gamma,
\end{aligned}
\end{equation}
and 
\begin{equation}
\frac{\p\p_tq}{\p N} = 0,\quad \text{on}\,\,\Gamma_0. 
\end{equation}
Invoking the standard elliptic equation again, we have 
\[
\|\p_t q\|_{3.5}\lesssim\|\text{RHS of }\eqref{qteq}\|_{1.5}+|\text{RHS of }\eqref{qtbdry}|_{2}+|\p_tq|_0.
\] 
The control of the first two terms follows similarly as above
\begin{equation}\label{qt1}
\|\text{RHS of }\eqref{qteq}\|_{1.5}+|\text{RHS of }\eqref{qtbdry}|_{2}\lesssim\PP,
\end{equation}
where we have used the inequality 
$$
\|fg\|_{1.5} \lesssim \|f\|_{1.5+\delta} \|g\|_{1.5}+\|f\|_{1.5}\|g\|_{1.5+\delta},
$$
which is a direct consequence of \eqref{product}. 

As for the boundary term, we take $\p_t$ in the surface tension equation to get
$$
\p_tq=-\sigma \frac{\sqrt{g}}{\sqrt{\gk}}(\Delta_g v\cdot \nnk ) +\kk\frac{1}{\sqrt{\gk}}(1-\TL) (\p_tv\cdot \nnk)+\text{lower-order temrs}
$$
and thus
\begin{equation}\label{qt2}
\|\p_tq\|_0\lesssim \PP.
\end{equation} Summing up \eqref{qt1}-\eqref{qt2} and choosing $\eps>0$ to be sufficiently small, we get
\begin{equation}\label{qt3}
\|\p_tq\|_{3.5}\lesssim\PP.
\end{equation}

Invoking the inequality 
\begin{align*}
\|fg\|_{0.5}\lesssim \|f\|_{0.5}\|g\|_{1.5+\delta},
\end{align*}
which follows from \eqref{product},
and time differentiating \eqref{qteq}-\eqref{qtbdry} and \eqref{BC2} again, we can silimarly get the estimates of $\|\p_t^2 q\|_{2.5}$: 
\begin{equation}\label{qtt2}
\|\p_t^2q\|_{2.5}\lesssim\PP.
\end{equation}
The treatment is similar to what has been done before and so we omit the details. 

However, we cannot use a similar method to control $\|\p_t^3 q\|_1$ because the standard elliptic estimates require at least $H^2$-regularity. Instead, we invoke Lemma \ref{H1elliptic} which allows us to perform the low regularity $H^1$-estimate for $\p_t^3$-differentiated elliptic system \eqref{qeq}-\eqref{qbdry}. To use this Lemma, we need first to rewrite the elliptic equations into the divergence form. Recall that the elliptic equation \eqref{qeq} is derived by taking smoothed Eulerian divergence $\divA$. This, together with Piola's identity $\p_\nu\Ak^{\nu\alpha}=0$ give that
\[
-\p_{\nu}(\Ak^{\nu\alpha}\Ak^{\mu\alpha}\p_{\mu}q)=\p_{\nu}\left(\Ak^{\nu\alpha}(\p_t v-\bp^2\eta)_{\alpha}\right),
\]
with the boundary condition
\[
\Ak^{3\alpha}\Ak^{\mu\alpha}\p_{\mu}q=\Ak^{3\alpha}(\p_t v-\bp^2\eta)_{\alpha} ,\quad \text{on}\,\,\Gamma,
\]
and $\frac{\p q}{\p N} =0$ on $\Gamma_0$. 
Taking $\p_t^3$ derivatives, we get
\begin{equation}\label{qttteq}
\begin{aligned}
\p_{\nu}(\Ak^{\nu\alpha}\Ak^{\mu\alpha}\p_t^3\p_{\mu}q)=&\p_{\nu}\left(\left[\Ak^{\nu\alpha}\Ak^{\mu\alpha},\p_t^3\right]\p_{\mu}q\right)+\p_{\nu}\p_t^3\left(\Ak^{\nu\alpha}(\p_tv-\bp^2\eta)_{\alpha}\right),
\end{aligned}
\end{equation}
with the boundary condition
\begin{equation}\label{qtttbdry}
\Ak^{3\alpha}\Ak^{\mu\alpha}\p_{\mu}\p_t^3 q=\left[\Ak^{3\alpha}\Ak^{\mu\alpha},\p_t^3\right]\p_{\mu}q+\p_t^3\left(\Ak^{3\alpha}(\p_t v-\bp^2\eta)_{\alpha}\right),\quad \text{on}\,\,\Gamma.
\end{equation}

Now if we set 
\[
\BB^{\nu\mu}:=\Ak^{\nu\alpha}\Ak^{\mu\alpha},\quad h:= \text{RHS of}\,\,\, \eqref{qtttbdry}
\]
and
\[
\pi^{\nu}:=\left[\Ak^{\nu\alpha}\Ak^{\mu\alpha},\p_t^3\right]\p_{\mu}q+\p_t^3\left(\Ak^{\nu\alpha}(\p_tv-\bp^2\eta)_{\alpha}\right)
\]
then the elliptic system \eqref{qttteq}-\eqref{qtttbdry} is exactly of the form \eqref{lowelliptic}. The a priori assumption \eqref{smallak} shows that $\|\BB-Id\|_{L^{\infty}}$ is sufficiently small. Now it is straightforward to see that $\pi, \dive \pi \in L^2$, i.e., 
\begin{equation}\label{piell}
\|\pi\|_0+\|\dive \pi\|_0\lesssim\PP.
\end{equation}Also, since 
\begin{align}
h-\pi\cdot N =0,
\end{align}then by Lemma \ref{H1elliptic} and invoking \eqref{q4}, \eqref{qt3}, \eqref{qtt2}, we have
\begin{align}\label{qttt11}
\left\|\p_t^3q-\overline{\p_t^3 q}\right\|_1\lesssim\|\pi\|_0\lesssim \PP.
\end{align}
 Lastly, we need to control the $H^1$-norm of $\overline{\p_t^3 q}$ by $\PP$. 
\begin{equation}
\begin{aligned}
&\overline{\p_t^3 q} = \frac{1}{\vol(\Omega)}\io \p_t^3 q\dy= \frac{1}{\vol(\Omega)}\io\p_t^3 q\TP_1 y_1\dy=- \frac{1}{\vol(\Omega)}\io y_1\TP_1\p_t^3q\\
\leq&~C(\vol(\Omega)) \|\TP\p_t^3 q\|_0\|y_1\|_0=C(\vol(\Omega))\left\|\TP(\p_t^3q-\overline{\p_t^3 q})\right\|_0\|y_1\|_0\\
\leq&~ C(\text{vol}(\Omega)) \left\|\p_t^3q-\overline{\p_t^3 q}\right\|_1.
\end{aligned}
\end{equation}
This concludes the control of $\|\p_t^3 q\|_1$, and we have
\begin{equation}\label{qttt1}
\|\p_t^3q\|_{1}\lesssim\PP.
\end{equation}

\subsection{The div-curl estimates}\label{divcurlkk}

Invoking Lemma \ref{hodge}, we have the following inequalities for $0\leq k\leq 3$
\begin{align}
\label{v40}\|v\|_{4.5}^2\lesssim&\|v\|_0^2+\|\dive v\|_{3.5}^2+\|\curl v\|_{3.5}^2+|\TP v^3|_{3}^2,\\
\label{b40}\|\bp\eta\|_{4.5}^2\lesssim&\|\bp\eta\|_0^2+\|\dive \bp\eta\|_{3.5}^2+\|\curl \bp\eta\|_{3.5}^2+|\TP \bp\eta^3|_{3}^2,\\
\label{vt0}\|\p_t^k v\|_{4.5-k}^2\lesssim&\|\p_t^kv\|_0^2+\|\dive \p_t^kv\|_{3.5-k}^2+\|\curl \p_t^kv\|_{3.5-k}^2+|\TP \p_t^kv^3|_{3-k}^2,\\
\label{bt0}\|\p_t^k\bp\eta\|_{4.5-k}^2\lesssim&\|\p_t^k\bp\eta\|_0^2+\|\dive\p_t^k \bp\eta\|_{3.5-k}^2\nonumber\\
&+\|\curl \p_t^k\bp\eta\|_{3.5-k}^2+|\TP \p_t^k\bp\eta^3|_{3-k}^2.
\end{align}
Here, notice that we do not pick up the terms on $\Gamma_0$ since $v^3=0$ and $\bp\eta^3= b_0^i \TP_i \eta^3=0$ there. 
Also,  the $L^2$-norms in \eqref{v40} and \eqref{b40} are controlled by energy conservation law. We will omit the control of $L^2$-norms appearing in the div-curl estimates in the rest of this manuscript.

\subsubsection*{Divergence estimates} 
For the velocity vector field, one has
\begin{equation}\label{divv0}
\dive v=\underbrace{\diva v}_{=0}+(\delta^{\mu\alpha}-\ak^{\mu\alpha})\p_\mu v_{\alpha}=\dive_{Id-\ak}v,
\end{equation}and thus
\begin{equation}\label{divv3}
\|\dive v\|_{3.5}\leq\|\diva v\|_{3.5}+\|(\delta^{\mu\alpha}-\ak^{\mu\alpha})\p_\mu v_{\alpha}\|_{3.5}\leq 0+\eps\|v\|_{4.5}.
\end{equation}

Time differentiating \eqref{divv0}, one has
\begin{equation}\label{divvt2}
\begin{aligned}
\|\dive \p_t v\|_{2.5}\lesssim&\|\dive_{Id-\ak}\p_t v\|_{2.5}+\|\dive_{\p_t \ak}v\|_{2.5}\\
\lesssim&\eps\|\p_t^2 v\|_{2.5}+\|\p_t\ak\|_{2.5}\|v\|_{3.5}\lesssim\eps\|\p_t^2 v\|_{2.5}+\|\p\ek\|_{2.5}\|v\|_{3.5}^2\\
\lesssim&\eps\|\p_t^2 v\|_{2.5}+P(\|v_0\|_{3.5})+\|\eta\|_{3.5}\int_0^TP(\|v\|_{4.5}),
\end{aligned}
\end{equation}where in the last step we write $\|v\|_{3.5}$ in terms of initial data plus time integral and use Young's inequality. The divergence estimates of $\|\p_t^k v\|_{3.5-k}$, $k=2,3$ are parallel and so we omit the details. 
\begin{equation}\label{divvtt1}
\|\dive\p_t^2 v\|_{1.5}+\|\dive\p_t^3 v\|_{0.5}\lesssim\eps(\|\p_t^2 v\|_{2.5}+\|\p_t^3 v\|_{1.5})+\PP_0+\PP\int_0^T\PP.
\end{equation}

As for $\bp\eta$, \textit{one no longer has $\diva(\bp\eta)=0$ due to the tangential mollification.} Instead, one can compute the evolution equation verified by $\diva(\bp\eta)$. Invoking $\diva v=0$ and $\p_t \eta=v$, we have
\begin{equation}\label{divbeq}
\p_t(\diva(\bp\eta))=[\diva,\bp]v+\dive_{\p_t\ak}\bp\eta.
\end{equation}
The commutator $[\diva,\bp]v$ only contains first order derivative of $v$ and $\bp\eta$. Using the identity
\begin{align}
\p \ak^{\mu\alpha}= - \ak^{\mu\gamma} \p_\beta \p \ek_\gamma\ak^{\beta\alpha}, 
\label{derivative a}
\end{align}
which follows from differentiating $\ak^{\mu\alpha}\p_\mu\ek_\beta=\delta^{\alpha}_\beta$, one has
\begin{align}
[\diva,\bp]v=&\ak^{\mu\alpha}\p_\mu b_0^{\nu}\p_{\nu}v_{\alpha}-b_0^{\nu}\p_{\nu}\ak^{\mu\alpha}\p_\mu v_{\alpha}\nonumber\\
=&\ak^{\mu\alpha}\p_\mu b_0^{\nu}\p_{\nu}v_{\alpha}+\p_\beta(\bp\ek_{\gamma}) \ak^{\mu\gamma}\ak^{\beta\alpha}\p_\mu v_{\alpha}-\p_{\beta}b_0^{\nu}\underbrace{\p_{\nu}\ek_{\gamma}\ak^{\mu\gamma}}_{\delta_{\nu}^{\mu}}\ak^{\beta\alpha}\p_{\mu}v_{\alpha}
\nonumber\\
=&\p_\beta(\bp\ek_{\gamma})\ak^{\mu\gamma}\ak^{\beta\alpha}\p_\mu v_{\alpha}.
\label{div1}
\end{align}
Moreover, 
\begin{align}
\dive_{\p_t\ak}\bp\eta = \p_t\ak^{\mu\alpha}\p_\mu \bp\eta_\alpha=-a^{\mu\gamma}\p_\beta \vk_\gamma \ak^{\beta\alpha}\p_\mu \bp\eta_\alpha. 
\label{div2}
\end{align}
Taking $\p^{3.5}$ in \eqref{divbeq} and testing it with $\p^{3.5}\di_{\ak}\bp\eta$, we get
\begin{equation}
\|\diva \bp\eta\|_{3.5}^2\leq\|\dive b_0\|_{3.5}^2+\int_0^T\|\diva \bp\eta\|_{3.5}\left(\|[\diva,\bp]v\|_{3.5}+\|\dive_{\p_t\ak}\bp\eta\|_{3.5}\right).
\end{equation}
This suggests that we need to control 
$
\int_0^T\|[\diva,\bp]v\|_{3.5}
$
and
$
\int_0^T\|\dive_{\p_t\ak}\bp\eta\|_{3.5}
$ on the right hand side. 
In light of \eqref{div1} and \eqref{div2}, we have
\[
\int_0^T\|[\diva,\bp]v\|_{3.5}+\|\dive_{\p_t\ak}\bp\eta\|_{3.5} \leq\int_0^T \PP. 
\]
Therefore, 
\begin{align}
\|\diva \bp\eta\|_{3.5}^2 \leq \|\di b_0\|_{3.5}^2 +\int_0^T\PP \leq \PP_0+ \int_0^T\PP,
\end{align}
which implies, after invoking \eqref{compars delta and a}, that
\begin{equation}\label{divb3}
\|\dive \bp\eta\|_{3.5}^2\lesssim\eps^2\|\bp\eta\|_{4.5}^2+\PP_0+\int_0^T\PP.
\end{equation}

Similarly, one can take $\p^{3.5-k}\p_t^k$ for $1\leq k\leq 3$ in \eqref{divbeq}, then compute the $L^2$ estimates to get
\begin{equation}\label{divbt}
\|\diva\p_t^k \bp\eta\|_{3.5-k}^2\lesssim\eps^2\|\p_t^k\bp\eta\|_{4.5-k}^2+\PP_0+\int_0^T\PP.
\end{equation}

\subsubsection*{Curl estimates}

The curl estimates can be derived by the evolution equation of $\curla v$. Taking $\curla$ in the second equation of \eqref{MHDLkk}, we get
\begin{equation}\label{curlbeq}
\p_t(\curlA v)-\bp\curlA(\bp\eta)=\curl_{\p_t\Ak}v+[\curlA,\bp]\bp\eta.
\end{equation}
Then we take $\p^{3.5}$, test it with $\p^{3.5} (\curlA v)$ and integrate $\bp$ by parts (recall that $b_0\cdot N|_{\p\Omega}=0$ and $\dive b_0=0$) to get
\begin{equation}
\begin{aligned}
&\frac12\frac{d}{dt}\io|\p^{3.5}\curlA v|^2+|\p^{3.5}\curlA \bp\eta|^2\dy\\
=&\io\left(\left[\p^{3.5}, \bp\right]\curlA\bp\eta+\p^{3.5}\left(\curl_{\p_t\Ak}v+[\curlA,\bp]\bp\eta\right)\right)(\p^{3.5}\curlA v)\dy\\
&+\io\p^{3.5}(\curlA\bp\eta)\cdot\left(\left[\p^{3.5}\curlA,\bp\right]v+\p^{3.5}(\curl_{\p_t\Ak}\bp\eta)\right)\dy\\
\lesssim& P(\|b_0\|_{4.5},\|\bp\eta\|_{4.5},\|v\|_{4.5},\|\eta\|_{4.5})\lesssim\PP,
\end{aligned}
\end{equation} and thus by the a priori assumption \eqref{smallak}, we have
\begin{equation}\label{curlvb3}
\|\curl v\|_{3.5}^2+\|\curl \bp\eta\|_{3.5}^2\lesssim\eps^2(\|v\|_{4.5}^2+\|\bp\eta\|_{4.5}^2)+\PP_0+\int_0^T\PP\dt.
\end{equation}

Similarly, replacing $\p^{3.5}$ by $\p^{3.5-k}\p_t^k$ for $1\leq k\leq 3$, we can similarly get the following curl estimates
\begin{equation}\label{curlvbt}
\|\curlA\p_t^k \bp\eta\|_{3.5-k}^2\lesssim\eps^2\|\p_t^k\bp\eta\|_{4.5-k}^2+\PP_0+\int_0^T\PP\dt.
\end{equation}

\subsection{Boundary estimates}
\label{sect bdy estimate 3.3}
We need to control the boundary term $|\TP\p_t^kv\cdot N|_{3-k}$ and $|\TP\p_t^k\bp\eta\cdot N|_{3-k}.$ In the case of zero surface tension, one can use the normal trace theorem to reduce $|\TP X\cdot N|_{s-1.5}$ to the interior tangential estimates $\|\TP^s X\|_0$. But the interior tangential estimates, especially in the full spatial derivative case, cannot be controlled due to the appearance of surface tension. 
\subsubsection{Control of $|\TP\p_t^kv\cdot N|_{3-k}$}
\begin{thm}
For $k=0,1,2,3$, one has
\begin{align} \label{boundary v}
|\TP\p_t^k v^3|_{3-k}^2\lesssim |\TP (\Pi \TP^{3-k}\p_t^k v)|_0^2+\PP\int_0^T\PP.
\end{align}
\end{thm}
First we study the case when $k=3$. Let us consider the projection of $\p_t^3 v$ to the Eulerian normal direction, i.e., $(\Pi \p_t^3 v)^3$ instead of Lagrangian normal direction. The reason is twofold.
\begin{enumerate}
\item Recall that \eqref{lapg} in Lemma \ref{geometric} gives that $$\sqrt{g}g^{ij}\Delta_g\eta^{\alpha}=\sigma\sqrt{g}g^{ij}\Pi^{\alpha}_{\lambda}\TP_{ij}^2\eta^{\lambda}.$$ So if we test $\p_t^4$-differentiated version of \eqref{lapg} with $\p_t^4 v$ and integrate by parts, then the term $|\TP(\Pi\p_t^3 v)|_0^2$ is  produced as part of energy term,i.e., 
\begin{equation}\label{lapge}
\ig\sigma\sqrt{g}g^{ij}\Pi^{\alpha}_{\lambda}\p_t^4\TP_{ij}^2\eta^{\lambda}\cdot\p_t^4 v_{\alpha}=-\frac{1}{2}\frac{d}{dt}\ig\left|\TP(\Pi\p_t^3 v)\right|^2\dS+\cdots
\end{equation}
\item The difference between $X^3$ and $(\Pi X)^3$ is expected to be small within a short period of time. 
\end{enumerate}

We will make the above assertions precise. For any vector field $X$,  the following identity holds:
\begin{equation}\label{pi3}
\begin{aligned}
X^3=\delta^3_{\lambda}X^{\lambda}=&(\delta^3_{\lambda}-g^{kl}\TP_k\eta^3\TP_l\eta_{\lambda})X^{\lambda}+g^{kl}\TP_k\eta^3\TP_l\eta_{\lambda}X^{\lambda}\\
=&\Pi^{3}_{\lambda}X^{\lambda}+g^{kl}\TP_k\eta^3\TP_l\eta_{\lambda}X^{\lambda}=(\Pi X)^3+g^{kl}\TP_k\eta^3\TP_l\eta_{\lambda}X^{\lambda}.
\end{aligned}
\end{equation} 
Using $\TP\eta^3=\int_0^T \TP v^3\dt$ (this is true since $\TP \eta^3=0$ initially), we can control the difference between $(\Pi X)^3$ and $X^3$ as
\begin{equation}\label{pigap1}
\begin{aligned}
\left|\TP\left((\Pi X)^3-X^3\right)\right|_0^2\lesssim&\left|g^{kl}\TP_k\eta^3\TP_l\eta_{\lambda}\TP X^{\lambda}\right|_0^2+\left|\TP(g^{kl}\TP_k\eta^3\TP_l\eta_{\lambda}) X^{\lambda}\right|_0^2\\
\lesssim& P(|\TP\eta|_{L^{\infty}})|\TP X|_0^2\int_0^T\left|\TP v^3\right|_{1.5}^2\dt+|X|_{L^4}^2|\TP(g^{kl}\TP_k\eta^3\TP_l\eta_{\lambda})|_{L^4}^2\\
\lesssim& \|X\|_{1.5}^2 P(|\TP\eta|_{L^{\infty}})\int_0^T\PP.
\end{aligned}
\end{equation}
Let $X=\p_t^3 v$. Since $\|\p_t^3 v\|_{1.5}^2$ is included in the energy $E_{\kk}^{(1)}$,  then \eqref{pigap1} implies
\begin{equation}
\left|\TP\left((\Pi \p_t^3 v)^3-\p_t^3 v^3\right)\right|_0^2\lesssim\PP\int_0^T\PP,
\end{equation}
and thus
\begin{align}\label{vttt3bdry}
\left| \TP\p_t^3 v^3 \right|_0^2 \lesssim \left| \TP (\Pi \p_t^3 v)\right|_0^2+ \PP\int_0^T\PP.
\end{align}
Finally, \eqref{boundary v} follows from a parallel argument by assigning $X= \TP\p_t^2 v, \TP^2\p_t v, \TP^3 v$, respectively.

\subsubsection{Control of $|\TP\p_t^k\bp\eta\cdot N|_{3-k}$}
First, when $k\geq 1$, the control of $|\TP\p_t^k\bp\eta\cdot N|_{3-k}$ requires to that of $|\TP \p_t^l v\cdot N|_{3-l}$ (modulo lower order terms generated when derivatives land on $b_0$) for $l=0,1,2$, which has been done in the previous subsection. 

Thus it suffices to study the control of $|\bp\eta^3|_4$. In Luo-Zhang \cite{luozhangMHDST3.5}, the boundary condition forms an elliptic equation $-\sigma\sqrt{g}\Delta_g\eta^{\alpha}=a^{3\alpha}q$ and thus one can take $\bp$ and then use elliptic estimates. However, the boundary condition now takes the form \eqref{BC1} in the smoothed approximate equations and it appears that there is no appropriate boundary $H^2$-control for $\kk \bp\TL (v\cdot \nnk)$. Specifically, it does not seem to be possible to control $|\kk \bp\TL (v\cdot \nnk)|_2$ by $\PP_0+C(\eps)E_\kk(T)+\int_0^T\PP$ due to the lack of time integrals. 

Our strategy here is to adapt the inequality \eqref{pigap1} with $X=\TP^3 \bp\eta$. In particular, we have
\begin{align}
\left|\TP\left((\Pi \TP^3 \bp\eta)^3-\TP^3 \bp\eta^3\right)\right|_0^2\lesssim
\|\TP^3 \bp \eta\|_{1.5} ^2P(|\TP\eta|_{L^{\infty}})\int_0^T\PP
\lesssim \PP \int_0^T\PP,
\end{align}
where the last inequality holds since $\|\bp\eta\|_{4.5}^2$ is included in $E_{\kk}^{(1)}$. Therefore, 
\begin{align}\label{b3bdry}
\left| \TP^4 \bp \eta^3 \right|_0^2 \lesssim \left| \TP (\Pi \TP^3 \bp \eta)\right|_0^2+ \PP\int_0^T\PP.
\end{align}
\begin{rmk} The term $\left| \TP (\Pi \TP^3 \bp \eta)\right|_0^2$ is part of the energy $E_\kk^{(1)}$ defined in \eqref{Ekk}, which is a positive term generated by the $\TP^3\bp$ tangential energy estimate (See Section \ref{tgkk}). There is no problem to study the $\TP^3\bp$-differentiated equations \eqref{MHDLkk} since it is analogous to the $\TP^3\p_t$-differentiated equations. Indeed, as mentioned in the remark right after \eqref{vhigh}  that $\bp \eta$ and $\p_t \eta$ (which is $v$) have the same space-time regularity. 
\end{rmk}
\section{Tangential energy estimates}\label{tgkk}
The purpose of this section is to investigate the a priori energy estimate for the tangentially differentiated approximate $\kk$-problem \eqref{MHDLkk}. In particular, we will study the energy estimate for
$$
\p_t^4, \TP\p_t^3, \cp^2\p_t^2, \cp^3 \p_t, \cp^3\bp
$$
differentiated $\kk$-problem, respectively. 
\subsection{Control of full time derivatives}\label{sect fulltime}
Now we compute the $L^2$-estimate of $\p_t^4 v$ and $\p_t^4\bp\eta$. This turns out to be the most difficult case compared to the cases with at least one tangential spatial derivative that will be treated in Section \ref{sect at least one TP}. This is because $\p_t^4 v$ can only be controlled in $L^2(\Omega)$ and so one has to control some higher-order interior terms instead. These interior terms will be treated by adapting the geometric cancellation scheme introduced in \cite{DK2017} together with an error term which can be controlled by terms in $E_{\kk}^{(3)}(t)$. 

For the sake of simplicity and clean arguments, we shall focus on treating the leading order terms. We henceforth adopt:

\begin{nota} \label{LL nota} We use $\LL$ to denote equality modulo error terms that are effective of lower order. For instance, $X\LL Y$ means that $X=Y+\RR$, where $\RR$ consists of lower order terms with respect to $Y$. 
\end{nota}

Invoking \eqref{MHDLkk} and integrating $\bp$ by parts, we get
\begin{equation}\label{tgt40}
\begin{aligned}
&\frac12\int_0^T \frac{d}{dt}\io|\p_t^4 v|^2+\left|\p_t^4\bp\eta\right|^2\dy\\
=&\int_0^T\io\p_t^4 v_{\alpha}\p_t^5 v^{\alpha}\dy\dt+\int_0^T\io\p_t^4\bp\eta_{\alpha}\p_t^4\bp v^{\alpha}\dy\dt\\
=&\int_0^T\io\p_t^4v_{\alpha}\p_t^4\bp^2\eta_{\alpha}\dy\dt-\int_0^T\io\p_t^4v_{\alpha}\p_t^4(\Ak^{\mu\alpha}\p_{\mu} q)\dy\dt\\
&+\int_0^T\io\p_t^4\bp\eta_{\alpha}\p_t^4\bp v^{\alpha}\dy\dt\\
=&-\int_0^T\io\p_t^4\bp v_{\alpha}\p_t^4\bp\eta_{\alpha}\dy\dt-\int_0^T\io\p_t^4v_{\alpha}\p_t^4(\Ak^{\mu\alpha}\p_{\mu} q)\dy\dt\\
&+\int_0^T\io\p_t^4\bp\eta_{\alpha}\p_t^4\bp v^{\alpha}\dy\dt\\
=&-\int_0^T\io\p_t^4v_{\alpha}\p_t^4(\Ak^{\mu\alpha}\p_{\mu} q)\dy\dt=:I.
\end{aligned}
\end{equation}
Then we integrate $\p_\mu$ by parts, $I$ becomes
\begin{align}
&\int_0^T \io \p_t^4 \p_\mu v_\alpha \p_t^4(\Ak^{\mu\alpha} q)-\underbrace{\int_0^T \ig \p_t^4 v_\alpha \p_t^4(\Ak^{3\alpha} q)}_{I_0}+\underbrace{\int_0^T \int_{\Gamma_0} \p_t^4 v_\alpha \p_t^4(\Ak^{3\alpha} q)}_{I_0'}\nonumber\\
=&\int_0^T \io \Ak^{\mu\alpha} \p_t^4 \p_\mu v_\alpha \p_t^4 q+\underbrace{\int_0^T \io \p_t^4 \p_\mu v_\alpha [\p_t^4,\Ak^{\mu\alpha}] q}_{I_1}+I_0\nonumber\\
=&\int_0^T \int_\Omega \underbrace{\p_t^4 \di_{\Ak} v}_{=0} \p_t^4 q-\underbrace{\int_0^T\int_\Omega [\p_t^4, \Ak^{\mu\alpha}] \p_\mu v_\alpha \p_t^4 q}_{L}+I_1+I_0+I_0'. 
\label{II} 
\end{align}
$I_0'=0$ since  on $\Gamma_0$, we have $\Ak^{31}=\Ak^{32}=0$, $\Ak^{33}=1$, and $v_3=0$. \\
$I_I$ yields a top order interior term when all $4$ time derivatives land on $\Ak^{\mu\alpha}$, i.e., 
\begin{equation}
I_{11}=\int_0^T \io \p_t^4 \p_\mu v_\alpha (\p_t^4\Ak^{\mu\alpha}) q.
\label{I11}
\end{equation}
If $\Ak^{\mu\alpha}$ were $A^{\mu\alpha}$ then this term could be controlled by adapting the cancellation scheme developed in \cite{DK2017}. This motivate us to consider 
\begin{equation}
\int_0^T \int_\Omega \p_t^4\p_\mu v_\alpha (\p_t^4 A^{\mu\alpha}) q+\int_0^T \int_\Omega \p_t^4\p_\mu v_\alpha \Big(\p_t^4 (\Ak^{\mu\alpha}-A^{\mu\alpha})\Big) q=I_{111}+I_{112}.
\end{equation}
Invoking \eqref{A's component}, we have $\p_t^4 \Ak = \sum_{i+j=3} b_{ij}(\p_t^i \p\ek) (\p \p_t^j \tilde{v})$ and $\p_t^4 A = \sum_{i+j=3} b_{ij}'(\p_t^i \p\eta) \p \p_t^j v$, where we denoted $A^{\mu\alpha}$ by $A$ and $\Ak^{\mu\alpha}$ by $\Ak$ by a slight abuse of notations. These imply that
\begin{align*}
\p_t^4 (\Ak - A) =\sum_{i+j=3} b_{ij}\p_t^i \p \ek\p\p_t^j(\tilde{v}-v)+\sum_{i+j=3} b_{ij}'\p_t^i (\p\ek-\p\eta)\p\p_t^j v,
\end{align*}
and so $
\|\p_t^4 (\Ak - A) \|_0 
$
consists the sum of $\|i_\ell\|_0$, $\ell=1,\cdots, 8$, where
\begin{align*}
i_1=(\p\p_t^2\tilde{v}) \p (\tilde{v} -v),\q i_2=(\p\p_t\tilde{v}) \p\p_t (\tilde{v} -v),\q i_3=(\p\tilde{v}) \p\p_t^2 (\tilde{v} -v),\q i_4=(\p \ek) \p\p_t^3 (\vk-v),\\
i_5=\p\p_t^2 (\vk-v)\p v,\q i_6=\p\p_t (\vk-v)\p\p_t v,\q i_7=\p (\vk-v)\p\p_t^2 v,\q i_8=\p(\ek-\eta)\p\p_t^3 v. 
\end{align*}
The $L^2$-norm of these quantities can be controlled by invoking Lemma \ref{ext Lemma 3.2 to the interior}. Specifically, 
\begin{align*}
\|i_1\|_0 \leq& \|\p (\vk-v)\|_{L^\infty}\|\p \p_t^2\vk\|_0 \lesssim \sqrt{\kk}\|v\|_{3.5} \|\p_t^2v\|_1,\\
\|i_2\|_0 \leq& \|\p\p_t (\vk-v)\|_{L^\infty}\|\p \p_t \vk\|_0 \lesssim \sqrt{\kk}\|\p_t v\|_{3.5} \|\p_t v\|_1,\\
\|i_3\|_0 \leq &\|\p\p_t^2 (\vk-v)\|_0\|\p \vk\|_{L^\infty} \lesssim \sqrt{\kk}\|\p_t^2 v\|_{1.5} \|v\|_3,\\
\|i_4\|_0 \leq& \|\p \ek\|_{L^\infty}\|\p\p_t^3 (\vk-v)\|_0\lesssim \sqrt{\kk}\|\eta\|_3 \|\p_t^3 v\|_{1.5},
\end{align*}
and 
\begin{align*}
\|i_5\|_0 \leq& \|\p \p_t^2 (\vk-v)\|_0 \|\p v\|_{L^\infty}\lesssim \sqrt{\kk} \|\p_t^2 v\|_{1.5}\|v\|_{3},\\
\|i_6\|_0 \leq&  \|\p \p_t (\vk-v)\|_{L^\infty} \|\p\p_t v\|_{0}\lesssim \sqrt{\kk} \|\p_t v\|_{3.5}\|\p_t v\|_1,\\
\|i_7\|_0\leq& \|\p  (\vk-v)\|_{L^\infty} \|\p \p_t^2 v\|_{0}\lesssim \sqrt{\kk} \|v\|_{3.5}\|\p_t^2v\|_1,\\
\|i_8\|_0 \leq &\|\p (\ek-\eta)\|_{L^\infty} \|\p \p_t^3 v\|_{L^2} \lesssim \sqrt{\kk} \|\eta\|_{3.5} \|\p_t^3 v\|_1. 
\end{align*}
Summing these up and moving $\sqrt{\kk}$ to $\|\p_t^4\p v\|_0$, we obtain
\begin{align}
I_{112}
\leq \int_0^T \|\p_t^4 \p v\|_0\|\p_t^4 (\Ak^{\mu\alpha}-A^{\mu\alpha})\|_0 \|q\|_{L^\infty}
 \leq \frac{\eps}{2}\int_0^T \|\sqrt{\kk} \p_t^4 \p v\|_0^2+ \frac{1}{2\eps}\int_0^T \PP,
\end{align}
where the first term on the RHS contributes to $\eps\PP$, and we bound $\|q\|_{L^\infty}$ by $\|q\|_2 \leq \PP$ through \eqref{ellq}. 

We next control $I_{111}$. \textit{The argument is largely similar to that used in Section 3.1.3 of \cite{DK2017} which relies on exploiting the geometric structure to create a remarkable cancellation scheme among the leading order terms.} Invoking \eqref{A's component} and then expanding the index $\mu$ in $I_{111}$, we have
\begin{align}
I_{111} = &\int_0^T \io q \epsilon^{\alpha\lambda\tau}  \TP_2\p_t^3v_\lambda \p_3 \eta_\tau \TP_1 \p_t^4 v_\alpha-\int_0^T \io q \epsilon^{\alpha\lambda\tau}  \TP_1\p_t^3v_\lambda \p_3 \eta_\tau \TP_2 \p_t^4 v_\alpha\nonumber\\
&+\int_0^T \io q \epsilon^{\alpha\lambda\tau}  \p_3\p_t^3v_\tau \TP_2 \eta_\lambda \TP_1 \p_t^4 v_\alpha-\int_0^T \io q \epsilon^{\alpha\lambda\tau}  \TP_1\p_t^3v_\tau \TP_2 \eta_\lambda \p_3 \p_t^4 v_\alpha\nonumber\\
&+\int_0^T\io q\epsilon^{\alpha\lambda\tau}\TP_2 \p_t^3 v_\tau \TP_1 \eta_\la \p_3 \p_t^4 v_\alpha-\int_0^T\io q\epsilon^{\alpha\lambda\tau} \p_3 \p_t^3 v_\tau \TP_1 \eta_\la \TP_2\p_t^4 v_\alpha + I_{low} \nonumber\\
=:&I_{1111}+I_{1112}+\cdots+I_{1116} + I_{low},
\label{cancellation 1}
\end{align}
where $I_{low}$ consists terms of the form $\int_0^T\io q \p\p_t^2 v \p v \p\p_t^3 v$. This term can be treated by integrating $\p_t$ by parts, 
\begin{align*}
\int_0^T\io q \p\p_t^2 v \p v \p\p_t^4 v = \io q \p\p_t^2 v \p v \p\p_t^3 v\Big|_{0}^T - \int_0^T\io\p_t (q \p\p_t^2 v \p v)\p \p_t^3 v,
\end{align*}
where the second term is controlled by $\int_0^T\PP$, whereas 
\begin{align*}
\left|\io q \p\p_t^2 v \p v \p\p_t^3 v\Big|_{0}^T\right| \lesssim \PP_0 + \eps\|\p_t^3 v\|_1^2 + \|q\|_{L^\infty}^2\|\p v\|_{L^\infty}^2\|\p\p_t^2 v\|_0^2 \leq \PP_0 +\eps ||\p_t^3 v||_1^2 + \int_0^T \PP.
\end{align*}
To control the leading terms in \eqref{cancellation 1}, we consider $I_{1111}+I_{1112}$, $I_{1113}+I_{1114}$, and $I_{1115}+I_{1116}$. For $I_{1111}+I_{1112}$, integrating $\p_t$ by parts in $I_{1112}$, we have
\begin{align}
I_{1111}+I_{1112} \leq& \underbrace{\int_0^T \io q \epsilon^{\alpha\lambda\tau}  \TP_2\p_t^3v_\lambda \p_3 \eta_\tau \TP_1 \p_t^4 v_\alpha-\int_0^T \io q \epsilon^{\alpha\lambda\tau}  \TP_1\p_t^4v_\lambda \p_3 \eta_\tau \TP_2 \p_t^3 v_\alpha}_{=0}\nonumber\\
&-\io q\epsilon^{\alpha\la\tau} \TP_1 \p_t^3 v_\la \p_3\eta_\tau \TP_2 \p_t^3 v_\alpha\Big|_0^T+I_{low}',
\label{I1111+I1112}
\end{align}
where $I_{low}'$ consists terms of the form $\int_0^T\io q\epsilon^{\alpha\la\tau} \p_t (q\p \eta) (\p \p_t^3 v)^2 $ which can be controlled by $\int_0^T\PP$. Next we treat the first term on the RHS of \eqref{I1111+I1112}. It suffices to consider $-\io q\epsilon^{\alpha\la\tau} \TP_1 \p_t^3 v_\la \p_3\eta_\tau \TP_2 \p_t^3 v_\alpha\Big|_{t=T}:=\TT$ as $$\io q\epsilon^{\alpha\la\tau} \TP_1 \p_t^3 v_\la \p_3\eta_\tau \TP_2 \p_t^3 v_\alpha\Big|_{t=0}\leq \PP_0.$$ 

We shall drop $\Big|_{t=T}$ in $\TT$ for the sake of clean notations. Expanding in $\tau$, we find
\begin{align}
\TT = -\io q\epsilon^{\alpha\la i} \TP_1 \p_t^3 v_\la \p_3\eta_i \TP_2 \p_t^3 v_\alpha -\io q\epsilon^{\alpha\la 3} \TP_1 \p_t^3 v_\la \p_3\eta_3 \TP_2 \p_t^3 v_\alpha.
\end{align}
Since $\p_3 \eta_i|_{t=0}=0$, we can write $\p_3 \eta_i = \int_0^T \p_3 v_i$, and so 
\begin{align}
-\io q\epsilon^{\alpha\la i} \TP_1 \p_t^3 v_\la \p_3\eta_i \TP_2 \p_t^3 v_\alpha \leq \PP\int_0^T\PP.
\end{align}
In addition to this, we have $\p_3 \eta_3 = 1 +\int_0^T \p_3 v_3$, and so
\begin{align}
-\io q\epsilon^{\alpha\la 3} \TP_1 \p_t^3 v_\la \p_3\eta_3 \TP_2 \p_t^3 v_\alpha \leq -\io q\epsilon^{\alpha\la 3} \TP_1 \p_t^3 v_\la \TP_2 \p_t^3 v_\alpha+ \PP\int_0^T\PP.
\end{align}
To treat the first term on the RHS, we expand $\epsilon^{\alpha \la 3}$ and get
\begin{align}
-\io q\epsilon^{\alpha\la 3} \TP_1 \p_t^3 v_\la  \TP_2 \p_t^3 v_\alpha=-\io q(\TP_1 \p_t^3 v_2\TP_2 \p_t^3 v_1-\TP_1\p_t^3 v_1\TP_2 \p_t^3 v_2).
\label{different step}
\end{align}
Integrating by parts $\TP_2$ in the first term and $\TP_1$ in the second term, we have
\begin{align*}
&-\io q(\TP_1 \p_t^3 v_2\TP_2 \p_t^3 v_1-\TP_1\p_t^3 v_1\TP_2 \p_t^3 v_2)\\
=&\underbrace{\io q\TP_1\TP_2 \p_t^3 v_2 \p_t^3 v_1-\io q\p_t^3 v_1\TP_1\TP_2 \p_t^3 v_2}_{=0}+ \io \TP_2 q \TP_1\p_t^3 v_2 \p_t^3 v_1-\io \TP_1 q \p_t^3 v_1 \TP_1\p_t^3 v_2.
\end{align*}
Here, 
\begin{align*}
 \Big|\io \TP_2 q \TP_1\p_t^3 v_2 \p_t^3 v_1-\io \TP_1 q \p_t^3 v_1 \TP_1\p_t^3 v_2\Big|\lesssim \epsilon ||\p_t^3 v||_1^2 + ||\p_t^3 v||_0^2||\p q||_{L^\infty}^2 \leq \eps ||\p_t^3 v||_1^2+\PP_0+\int_0^T\PP.
\end{align*}
Therefore, 
\begin{align}
I_{1111}+I_{1112} \leq \eps E(T) +\PP_0 +\PP\int_0^T\PP. 
\end{align}

On the other hand, $I_{1113}+I_{1114}$ and $I_{1115}+I_{1116}$ are treated similarly with only one exception.  Previously, we integrated $\TP_1$ and $\TP_2$ by parts in \eqref{different step} and so there were no boundary terms. However, when controlling $I_{1113}+I_{1114}$, we need to integrate $\TP_1$ and $\p_3$ by parts when treating \eqref{different step}, and thus the following boundary term will appear:
\begin{equation}
\ig q\p_t^3 v_1 \TP_1 \p_t^3 v_3.
\label{cancellation b}
\end{equation}
To control this term, we invoke the identity
\begin{equation}
\TP_1 \p_t^3 v^3 = \Pi_\la^3 \TP_1 \p_t^3 v^\la+ g^{kl} \TP_k\eta^3 \TP_l \eta_\la \TP_1\p_t^3v^\la = \Pi_\la^3 \TP_1 \p_t^3 v^\la+ g^{kl} \left(\int_0^T\TP_k v^3\right) \TP_l \eta_\la \TP_1\p_t^3v^\la,
\end{equation}
and thus \eqref{cancellation b} becomes
\begin{align*}
&\ig q \p_t^3 v_1 \Pi_\la^3 \TP_1 \p_t^3 v^\la+\ig q \p_t^3 v_1 g^{kl} \left(\int_0^T\TP_k v^3\right) \TP_l \eta_\la \TP_1\p_t^3v^\la\\
\lesssim &\eps |\Pi \TP \p_t^3 v|_0^2+|q|_{L^\infty}^2|\p_t^3 v|_0^2+\left|q \p_t^3 v_1 g^{kl} \left(\int_0^T\TP_k v^3\right) \TP_l \eta_\la\right|_{0.5}|\TP\p_t^3 v^\la|_{-0.5}\\
\lesssim& \eps |\Pi \TP \p_t^3 v|_0^2 + \PP_0+\PP\int_0^T\PP.
\end{align*}
The extra term generated when analyzing $I_{1115}+I_{1116}$ is of the same type integral and thus can be treated by the same method.
Therefore, 
\begin{align}
I_{111} \leq \eps E(T) +\PP_0+\PP\int_0^T\PP.
\end{align}

Next, we study
\begin{equation}
\begin{aligned}
I_1-I_{11} =&~ 4\int_0^T\io \p_t^4\p_\mu v_\alpha \p_t^3 \Ak^{\mu\alpha} \p_t q + 6 \int_0^T\io \p_t^4\p_\mu v_\alpha \p_t^2 \Ak^{\mu\alpha} \p_t^2 q \\
& + 4\int_0^T\io \p_t^4\p_\mu v_\alpha \p_t \Ak^{\mu\alpha} \p_t^3 q =I_{12}+I_{13}+I_{14}.
\end{aligned}
\end{equation}
For $I_{12}$, we integrating $\p_t$ by parts and obtain
\begin{equation*}
4\io \p_t^3\p_\mu v_\alpha \p_t^3 \Ak^{\mu\alpha} \p_t q-4\int_0^T\io \p_t^3\p_\mu v_\alpha \p_t(\p_t^3 \Ak^{\mu\alpha} \p_t q).
\end{equation*}
Here, the second term is $\leq \int_0^T\PP$, and since 
$$
\p_t^3 \Ak = Q(\p \ek) \p \p_t^2 \vk + \text{lower order terms}
$$
 then the first term is bounded by
\begin{align*}
\eps ||\p_t^3 v||_1^2 + \PP_0+\int_0^T\PP.
\end{align*}
$I_{13}$ is treated by adapting a similar method and so we omit the details. However, we can't integrate $\p_t$ by parts in order to control $I_{14}$ as we do not have a bound for $\p_t^4 q$. We integrate $\p_\mu$ by parts instead. 
\begin{align*}
I_{14} =4\int_0^T\ig \p_t^4 v_\alpha \p_t\Ak^{3\alpha} \p_t^3 q-4\int_0^T\io \p_t^4 v_\alpha \p_\mu (\p_t \Ak^{3\alpha} \p_t^3 q).
\end{align*}
There is no problem to control the second integral by $\int_0^T\PP$. For the first integral, invoking the boundary condition \eqref{BC1}, we obtain
\begin{align}
-4\sigma\int_0^T\ig \p_t^4 v_\alpha \p_t\Ak^{3\alpha} \p_t^3\Big( \frac{\sqrt{g}}{\sqrt{\gk}}\Delta_g \eta\cdot \nnk\Big) +4\int_0^T\ig\kk\p_t^4 v_\alpha \p_t\Ak^{3\alpha}\p_t^3\Big(\frac{1}{\sqrt{\gk}}(1-\TL) (v\cdot \nnk)\Big)=I_{141}+I_{142}.
\end{align}
Invoking \eqref{lapg}, $I_{141}$ becomes
\begin{align*}
I_{141}=&-4\sigma\int_0^T\ig \p_t^4 v_\alpha \p_t\Ak^{3\alpha}\p_t^3\Big(\frac{\sqrt{g}}{\sqrt{\gk}} g^{ij}\TP_i\TP_j\eta\cdot \nnk\Big)\\
&-4\sigma\int_0^T\ig \p_t^4 v_\alpha \p_t\Ak^{3\alpha}\p_t^3\Big(\frac{\sqrt{g}}{\sqrt{\gk}}g^{ij}g^{kl}\TP_l\eta^{\mu}\TP_i\TP_j\eta_\mu\TP_k\eta\cdot \nnk\Big).
\end{align*}
It suffices for us to consider the first integral only since the second integral is of the same type. Integrating by parts $\TP_j$ first and then $\p_t$, the first integral becomes
\begin{align*}
-4\sigma\int_0^T\ig \p_t^3\TP_i v_\alpha \p_t\Ak^{3\alpha}\Big(\frac{\sqrt{g}}{\sqrt{\gk}} g^{ij}\TP_j\p_t^3 v\cdot \nnk\Big)-4\sigma\ig \p_t^3\TP_i v_\alpha \p_t\Ak^{3\alpha}\Big(\frac{\sqrt{g}}{\sqrt{\gk}} g^{ij}\TP_j\p_t^2 v\cdot \nnk\Big)+\RR.
\end{align*}
Since $\|\p_t^3v\|_{3.5}$ is part of $E_{\kk}^{(1)}(t)$, the trace lemma implies that the first integral is bounded straightforwardly by $\int_0^T\PP$. Moreover, for the second integral, we have
\begin{equation}
\begin{aligned}
&4\sigma\ig \p_t^3\TP_i v_\alpha \p_t\Ak^{3\alpha}\Big(\frac{\sqrt{g}}{\sqrt{\gk}} g^{ij}\TP_j\p_t^2 v\cdot \nnk\Big)\\ 
\lesssim &\eps |\p_t^3 v|_1^2 + P(\|\p \eta\|_{L^\infty},\|\p v\|_{L^\infty}) |\p_t^2 v|_1^2 \leq \eps \|\p_t^3 v\|_{1.5}^2 +\PP_0+\int_0^T\PP,
\end{aligned}
\end{equation}
where we used the trace lemma in the last inequality.
In addition, 
$$
I_{142}\LL -4\int_0^T\ig(\sqrt{\kk}\p_t^4 v_\alpha) \p_t\Ak^{3\alpha}\Big(\frac{1}{\sqrt{\gk}}\TL (\sqrt{\kk}\p_t^3 v\cdot \nnk)\Big).
$$
Integrating $\TP$ by parts,then
\begin{align*}
I_{142}\LL & 4\int_0^T\ig(\sqrt{\kk}\TP\p_t^4 v_\alpha) \p_t\Ak^{3\alpha}\frac{1}{\sqrt{\gk}}(\sqrt{\kk}\TP\p_t^3 v\cdot \nnk)\\
\lesssim& 4\int_0^T \eps |\sqrt{\kk}\TP\p_t^4 v|_0^2 + P(\|\p \eta\|_{L^\infty},\|\p v\|_{L^\infty})|\sqrt{\kk} \TP \p_t^3 v|_0^2\dt\\
\lesssim& \eps \int_0^T \|\sqrt{\kk}\p_t^4 v\|_{1.5}^2 + \PP_0 +\int_0^T\PP.
\end{align*}

Now, we start to analyze the boundary integral $I_0$ in \eqref{II}. This is essentially identical to the case of the incompressible Euler equations, which have been treated in \cite{coutand2007LWP}, Sect.12. Indeed, as what appears in the previous paper \cite{luozhangMHDST3.5} concerning the a priori estimate, we found that the magnetic field plays no role in the estimate of $I_0$. But we shall provide the control of the top order terms for the sake of the completeness of our proof.

By plugging the boundary condition $$\Ak^{3\alpha}q = -\sigma \sqrt{g}(\Delta_g \eta\cdot \nnk )\nnk^\alpha +\kk\left((1-\TL) (v\cdot \nnk)\right)\nnk^\alpha$$ in $I_0$ we obtain
\begin{align} \label{boundary term p_t4}
\frac{1}{\sigma}I_0=\int_0^T\ig\p_t^4v_{\alpha}\p_t^4(\sqrt{g}\Delta_g\eta\cdot \nnk\nnk^\alpha)\dS\dt-\frac{\kk}{\sigma}\int_0^T\ig\p_t^4 v_\alpha \p_t^4 [(1-\TL)(v\cdot \nnk)\nnk^\alpha]\dS\dt,
\end{align} 
where, after integrating one tangential derivative by parts, the second term becomes
\begin{align} \label{I00 pre}
-\frac{\kk}{\sigma}\sum_{\ell=0,1} \bigg(\int_0^T\ig \cp^\ell \p_t^4 v_\alpha \p_t^4 [\cp^\ell (v\cdot\nnk)\nnk^\alpha]\dS\dt+\int_0^T\ig \p_t^4 v_\alpha \p_t^4 [\cp^\ell(v\cdot\nnk)\cp^\ell\nnk^\alpha]\dS\dt\bigg).  
\end{align}
The first term on the RHS contributes to the positive energy term (after moving to the LHS)
$$\frac{\kk}{\sigma}\int_0^T\ig\left|\p_t^4 v\cdot \nnk\right|_1^2\dS\dt$$ together with errors terms. The most difficult error term is
\begin{equation}\label{boundary error}
\kk \int_0^T\int_{\Gamma} (\TP \p_t^4 v\cdot \nnk)(v\cdot \p_t^4 \TP \nnk)\dS\dt,
\end{equation}
where the other errors are either with the same type of integrand or are effectively of lower order by one derivative with the case above. Since $\TP \nnk = Q(\TP \ek) \TP^2 \ek \cdot \nnk$, we have
\begin{align*}
&\frac{\kk}{\sigma} \int_0^T\int_{\Gamma} (\TP \p_t^4 v\cdot \nnk)(v\cdot \p_t^4 \TP \nnk)\dS\dt\\
 \LL & \frac{\kk}{\sigma} \int_0^T\int_{\Gamma} (\TP \p_t^4 v\cdot \nnk)(v\cdot \TP^2\p_t^3 \vk\cdot \nnk)\dS\dt\\
\leq &\int_0^T P(|\TP\ek|_{L^\infty{(\Gamma)}}, |v|_{L^\infty{(\Gamma)}})|\sqrt{\kk} \TP \p_t^4 v|_0|\sqrt{\kk}\TP^2\p_t^3 v\cdot \nnk|_0 \\
\lesssim &\int_0^T |\sqrt{\kk}\TP\p_t^4 v|_0^2 + \sup_t P(|\TP\ek|_{L^\infty{(\Gamma)}}, |v|_{L^\infty{(\Gamma)}}) +\Big(\int_0^T |\sqrt{\kk} \TP^2\p_t^3 v\cdot \nnk|_0^2\Big)^2\\
\lesssim& \int_0^T \|\sqrt{\kk}\p_t^4 v\|_{1.5}^2 + \Big(\int_0^T \|\sqrt{\kk} \p_t^3 v\|_{2.5}^2\Big)^2 + \sup_t P(|\TP\ek|_{L^\infty{(\Gamma)}}, |v|_{L^\infty{(\Gamma)}}) \\
\leq & E_\kk^{(3)}+ (E_{\kk}^{(3)})^2 + \sup_t P(|\TP\ek|_{L^\infty{(\Gamma)}}, |v|_{L^\infty{(\Gamma)}}) .
\end{align*}
Here, the last term can be controlled appropriately because
\begin{align*}
|\TP\ek|_{L^\infty(\Gamma)} \lesssim \|\eta\|_3 \leq \|\eta_0\|_3+\int_0^T \|v\|_3,\\
|v|_{L^\infty{(\Gamma)}} \lesssim \|v\|_2 \leq \|v_0\|_2+\int_0^T \|v_t\|_2,
\end{align*}
and so $\sup_t P(|\TP\ek|_{L^\infty{(\Gamma)}}, |v|_{L^\infty{(\Gamma)}})\leq \PP_0 + \PP\int_0^T\PP$. In addition, the second term on the RHS of \eqref{I00 pre} can be treated by the same argument.

Next we analyze the first term on the RHS of \eqref{boundary term p_t4}.
Since $\nn\cdot \nn=1$, invoking \eqref{lapg0} in Lemma \ref{geometric} and we obtain
\begin{equation}
\Delta_g \eta\cdot \nn\nn^\alpha= -\mathcal{H}\circ \eta\nn^\alpha = \Delta_g \eta^\alpha,
\end{equation}
and so we are able to rewrite
\begin{align}
\sqrt{g}\Delta_g \eta\cdot \nnk\nnk^\alpha=&\sqrt{g}\Delta_g \eta\cdot \nn\nn^\alpha+\sqrt{g}\Delta_g \eta\cdot \nnk (\nnk^\alpha-\nn^\alpha)+\sqrt{g}\Delta_g \eta\cdot (\nnk-\nn)\nn^\alpha\nonumber\\
=& \sqrt{g}\Delta_g \eta^\alpha +\sqrt{g}\Delta_g \eta\cdot \nnk (\nnk^\alpha-\nn^\alpha)+\sqrt{g}\Delta_g \eta\cdot (\nnk-\nn)\nn^\alpha.
\end{align}
In light of this, the first term on the RHS of \eqref{boundary term p_t4} becomes
\begin{align}
&\int_0^T\ig\p_t^4v_{\alpha}\p_t^4(\sqrt{g}\Delta_g \eta^\alpha)\dS\dt+\int_0^T\ig\p_t^4v_{\alpha}\p_t^4[\sqrt{g}\Delta_g \eta\cdot \nnk (\nnk^\alpha-\nn^\alpha)]\dS\dt\nonumber\\
&+\int_0^T \ig\p_t^4v_{\alpha}\p_t^4[\sqrt{g}\Delta_g \eta\cdot (\nnk-\nn)\nn^\alpha]\dS\dt.
\label{boundary 3}
\end{align}
We shall study the main term $I_{00}=\int_0^T\ig\p_t^4v_{\alpha}\p_t^4(\sqrt{g}\Delta_g \eta^\alpha)\dS\dt$. The error terms involving $\nnk-\nn$ are treated using \eqref{lkk33} and they are identical to the Euler case. We refer \cite[(12.16)-(12.19)] {coutand2007LWP} for the details. Invoking \eqref{lapg}-\eqref{tplapg}, we have
\begin{equation}\label{I000}
\begin{aligned}
&I_{00}
=\int_0^T\ig\p_t^4 v_{\alpha}\p_t^3\TP_i\left(\sqrt{g}g^{ij}\Pi^{\alpha}_{\lambda}\TP_j v^{\lambda}\right)\dS\dt\\
&+\int_0^T\ig\p_t^4 v_{\alpha}\p_t^3\TP_i\left(\sqrt{g}(g^{ij}g^{kl}-g^{lj}g^{ik})\TP_j\eta^{\alpha}\TP_k\eta_{\lambda}\TP_l v^{\lambda}\right).
\end{aligned}
\end{equation}
 Integrating $\TP_i$ by parts and expanding the parenthesis, we get
\begin{equation}\label{I0e}
\begin{aligned}
\eqref{I000}=&-\int_0^T\ig\sqrt{g}g^{ij}\Pi^{\alpha}_{\lambda}\p_t^3\TP_jv^{\lambda}\p_t^4\TP_iv_{\alpha}\\
&-\int_0^T\ig\sqrt{g}(g^{ij}g^{kl}-g^{lj}g^{ik})\TP_j\eta^{\alpha}\TP_k\eta_{\lambda}\p_t^3\TP_l v^{\lambda}\p_t^4\TP_i v_{\alpha}\\
&-3\int_0^T\ig\p_t(\sqrt{g}g^{ij}\Pi^{\alpha}_{\lambda})\p_t^2\TP_j v^{\lambda}\p_t^4\TP_iv_{\alpha}\dS\dt\\
&-3\int_0^T\ig\p_t\left(\sqrt{g}(g^{ij}g^{kl}-g^{lj}g^{ik})\TP_j\eta^{\alpha}\TP_k\eta_{\lambda}\right)\p_t^2\TP_l v^{\lambda}\p_t^4\TP_iv_{\alpha}\\
&-3\int_0^T\ig\p_t^2(\sqrt{g}g^{ij}\Pi^{\alpha}_{\lambda})\p_t\TP_j v^{\lambda}\p_t^4\TP_iv_{\alpha}\dS\dt\\
&-3\int_0^T\ig\p_t^2\left(\sqrt{g}(g^{ij}g^{kl}-g^{lj}g^{ik})\TP_j\eta^{\alpha}\TP_k\eta_{\lambda}\right)\p_t\TP_l v^{\lambda}\p_t^4\TP_iv_{\alpha}\\
&-\int_0^T\ig\p_t^3(\sqrt{g}g^{ij}\Pi^{\alpha}_{\lambda})\TP_j v^{\lambda}\p_t^4\TP_iv_{\alpha}\dS\dt\\
&-\int_0^T\ig\p_t^3\left(\sqrt{g}(g^{ij}g^{kl}-g^{lj}g^{ik})\TP_j\eta^{\alpha}\TP_k\eta_{\lambda}\right)\TP_l v^{\lambda}\p_t^4\TP_iv_{\alpha}\\
=:&I_{01}+\cdots+I_{08}.
\end{aligned}
\end{equation}

The main terms are $I_{01}$ and $I_{02}$ which produces the term $|\TP(\Pi\p_t^3 v)|_0^2$ as a part of energy functional, and the others can be controlled by estimating $I_{03}+I_{04},I_{05}+I_{06},I_{07}+I_{08}$ and integrating $\p_t$ by parts. In $I_{01}$, we integrate $\p_t$ by parts and use \eqref{Pipro} in Lemma \ref{geometric} to get
\begin{equation}\label{I010}
\begin{aligned}
I_{01}=&-\frac12\ig\sqrt{g}g^{ij}\Pi^{\alpha}_{\lambda}\p_t^3\TP_jv^{\lambda}\p_t^3\TP_iv_{\alpha}\bigg|^T_0+\frac12\int_0^T\ig\p_t(\sqrt{g}g^{ij}\Pi^{\alpha}_{\lambda})\p_t^3\TP_jv^{\lambda}\p_t^3\TP_iv_{\alpha}\dS\dt\\
=&\frac12\ig\sqrt{g}g^{ij}\TP_i(\Pi^{\alpha}_{\mu}\p_t^3v_{\alpha})\TP_j(\Pi^{\mu}_{\lambda}\p_t^3v^{\lambda})+\ig\sqrt{g}g^{ij}\TP\Pi^{\alpha}_{\mu}\p_t^3v_{\alpha}\TP_j(\Pi^{\mu}_{\lambda}\p_t^3v^{\lambda})\\
&-\frac12\ig\TP_i\Pi^{\alpha}_{\mu}\TP_j\Pi^{\mu}_{\lambda}\p_t^3v_{\alpha}\p_t^3v^{\lambda}+\frac12\int_0^T\ig\p_t(\sqrt{g}g^{ij}\Pi^{\alpha}_{\lambda})\p_t^3\TP_jv^{\lambda}\p_t^3\TP_iv_{\alpha}\dS\dt+I_{01}|_{t=0}\\
=:&I_{011}+I_{012}+I_{013}+I_{014}+I_{01}|_{t=0}.
\end{aligned}
\end{equation}
The term $I_{011}$ produces the energy term
\begin{equation}\label{I011}
\begin{aligned}
I_{011}=&-\frac12\ig\left|\TP(\Pi\p_t^3v)\right|^2\dS-\frac12\ig(\sqrt{g}g^{ij}-\delta^{ij})\TP_i(\Pi^{\alpha}_{\mu}\p_t^3v_{\alpha})\TP_j(\Pi^{\mu}_{\lambda}\p_t^3v^{\lambda})\dS\\
\lesssim&-\frac12\left|\TP(\Pi\p_t^3v)\right|_0^2+\left|\TP(\Pi\p_t^3v)\right|_0^2\left|\sqrt{g}g^{ij}-\delta^{ij}\right|_{1.5}\\
\lesssim&-\frac12\left|\TP(\Pi\p_t^3v)\right|_0^2+\left|\TP(\Pi\p_t^3v)\right|_0^2\int_0^T P(\|\p_t\TP\eta\|_{2},\|\TP\eta\|_2)\dt.
\end{aligned}
\end{equation}

The terms $I_{012},I_{013},I_{014}$ can all be directly controlled. Because $\TP^2\eta|_{t=0}=0$, then
\begin{equation}\label{I012}
\begin{aligned}
I_{012}\lesssim&\left|\sqrt{g}g^{-1}\right|_{L^{\infty}}|\TP\Pi|_{L^{\infty}}|\p_t^3 v|_0|\TP(\Pi\p_t^3 v)|_0\\
\lesssim& P(|\TP\eta|_{L^{\infty}},|\TP^2\eta|_{L^{\infty}})\|\p_t^3 v\|_{0.5}|\TP(\Pi\p_t^3 v)|_0\\
\lesssim&\eps\left|\TP(\Pi\p_t^3 v)\right|_0^2+P(\|\eta\|_4)\|\p_t^3 v\|_0\|\p_t^3 v\|_1\\
\lesssim&\eps\left(\left|\TP(\Pi\p_t^3 v)\right|_0^2+\|\p_t^3 v\|_{1.5}^2\right)+\PP_0+\int_0^T P(\|\eta\|_4,\|v\|_4,\|\p_t^4v\|_0)\dt,
\end{aligned}
\end{equation}
and
\begin{equation}\label{I013}
\begin{aligned}
I_{013}\lesssim&|\TP\Pi|_{L^4}^2|\p_t^3 v|_{L^4}^2\lesssim P(\|\TP\eta\|_2)\|\TP^2\eta\|_2\|\p_t^3 v\|_{0}\|\p_t^3 v\|_1\\
\lesssim&\eps\|\p_t^3 v\|_{1.5}^2+P(\|\TP\eta\|_2)\|\p_t^3 v\|_{0}\int_0^T\|\TP^2 v\|_2\dt,
\end{aligned}
\end{equation}
and
\begin{equation}\label{I014}
I_{014}\lesssim\int_0^T\left|\p_t^3\TP v\right|_0^2\left|\p_t(\sqrt{g}g^{ij}\Pi)\right|_{L^{\infty}}\lesssim \int_0^T P(\|\p_t^3 v\|_{1.5},\|v\|_3,\|\eta\|_3)\dt.
\end{equation}
Combining \eqref{I010} with \eqref{I011}-\eqref{I014}, we get the estimates of $I_{01}$ as follows
\begin{equation}\label{I01}
I_{01}\lesssim\eps\left(\left|\TP(\Pi\p_t^3 v)\right|_0^2+\|\p_t^3 v\|_{1.5}^2\right)+\PP_0+\PP\int_0^T\PP\dt.
\end{equation}

Next we control $I_{02}:=-\int_0^T\ig\sqrt{g}(g^{ij}g^{kl}-g^{lj}g^{ik})\TP_j\eta^{\alpha}\TP_k\eta_{\lambda}\p_t^3\TP_l v^{\lambda}\p_t^4\TP_i v_{\alpha}$. We expand the summation on $l,i$ and find that:
\begin{itemize}
\item When $l=i$, this integral is zero thanks to the symmetry.
\item When $l=1,i=2$, the integrand becomes $\sqrt{g}^{-1}(\TP_1\eta_{\lambda}\TP_2\eta_{\alpha}-\TP_1\eta_{\alpha}\TP_2\eta_{\lambda})\p_t^3\TP_1v^{\lambda}\p_t^4\TP_2v^{\alpha}$.
\item When $l=2,i=1$, the integrand becomes $-\sqrt{g}^{-1}(\TP_1\eta_{\lambda}\TP_2\eta_{\alpha}-\TP_1\eta_{\alpha}\TP_2\eta_{\lambda})\p_t^3\TP_2v^{\lambda}\p_t^4\TP_1v^{\alpha}$.
\end{itemize}
Here,  we use $g^{-1}$ to denote $\det[g^{-1}]=g^{11}g^{22}-g^{12}g^{21}$. Therefore, we have
\begin{equation}\label{I020}
\begin{aligned}
I_{02}=&-\int_0^T\ig\frac{1}{\sqrt{g}}\left(\TP_1\eta_{\lambda}\TP_2\eta_{\alpha}-\TP_1\eta_{\alpha}\TP_2\eta_{\lambda}\right)\left(\p_t^3\TP_1v^{\lambda}\p_t^4\TP_2v^{\alpha}+\p_t^3\TP_2v^{\lambda}\p_t^4\TP_1v^{\alpha}\right)\dS\dt\\
=&\int_0^T\ig\frac{1}{\sqrt{g}}\frac{d}{dt}\left(\det\underbrace{\begin{bmatrix}
\TP_1\eta_{\mu}\p_t^3\TP_1v^{\mu}& \TP_1\eta_{\mu}\p_t^3\TP_2v^{\mu}\\
\TP_2\eta_{\mu}\p_t^3\TP_1v^{\mu}& \TP_2\eta_{\mu}\p_t^3\TP_2v^{\mu}
\end{bmatrix}}_{=:\mathbf{A}}\right)+\text{ lower order terms}\\
\overset{\p_t}{=}&\ig\frac{1}{\sqrt{g}}\det\mathbf{A}\bigg|^T_0-\int_0^T\ig\p_t\left(\frac{1}{\sqrt{g}}\right)\det\mathbf{A}
\end{aligned}
\end{equation} 
The first term in the last line of \eqref{I020} can be expanded into two terms
\begin{equation}\label{I021}
\ig\frac{1}{\sqrt{g}}\det\mathbf{A}
=\ig\frac{1}{\sqrt{g}}\left(\TP_1\eta_{\mu}\TP_2\eta_{\lambda}\TP_1\p_t^3v^{\mu}\TP_2\p_t^3v^{\lambda}-\TP_1\eta_{\mu}\TP_2\eta_{\lambda}\TP_2\p_t^3v^{\mu}\TP_1\p_t^3v^{\lambda}\right).
\end{equation}
It can be seen that the top order terms cancel with each other if one integrates $\TP_1$ by parts in the first term and $\TP_2$ by parts in the second. The remaining terms are all of the form $-\ig Q_{\mu\lambda}(\TP\eta,\TP^2\eta)\p_t^3 v^{\mu}\TP\p_t^3v^{\lambda}$, which can be controlled as
\begin{align}
&-\ig Q_{\mu\lambda}(\TP\eta,\TP^2\eta)\p_t^3 v^{\mu}\TP\p_t^3v^{\lambda}\nonumber\\
\lesssim& P(|\TP^2\eta|_{L^{\infty}},|\TP\eta|_{L^{\infty}})|\p_t^3 v|_0|\TP\p_t^3 v|_0\nonumber\\
\lesssim&\eps\|\p_t^3 v\|_{1.5}^2+\frac{1}{4\eps}\|\p_t^3 v\|_{0.5}\int_0^TP(\|\TP^2 v\|_2)\dt.
\end{align}
The second term of \eqref{I020} can be directly controlled, i.e., 
\begin{equation}\label{I022}
\int_0^T\ig\p_t\left(\frac{1}{\sqrt{g}}\right)\det\mathbf{A}\lesssim|\p_t\TP\eta|_{L^{\infty}}|\TP\eta|_{L^{\infty}}^2|\TP\p_t^3 v|_0^2\dt\lesssim\int_0^T\PP.
\end{equation} 
Therefore, we get the estimates of $I_{02}$:
\begin{equation}\label{I02}
I_{02}\lesssim\eps\|\p_t^3 v\|_{1.5}^2+\PP_0+\PP\int_0^T\PP\dt,
\end{equation}

Next we control the remaining terms in $I_0$, i.e., $I_{03},\cdots, I_{08}$. The strategy here is to study $I_{03}+I_{04},I_{05}+I_{06},I_{07}+I_{08}$, where

\begin{equation}\label{I034}
\begin{aligned}
I_{03}+I_{04}=&-3\int_0^T\ig\p_t(Q(\TP\eta))\p_t^2\TP v\p_t^4\TP v\dS\dt\\
\overset{\p_t}{=}&3\int_0^T\ig \p_t^2(Q(\TP\eta))\p_t^2\TP v\p_t^3\TP v+3\int_0^T\ig \p_t(Q(\TP\eta))\p_t^3\TP v\p_t^3\TP v+3\ig \p_t(Q(\TP\eta))\p_t^2\TP v\p_t^3\TP v\bigg|_{0}^T\\
=&3\int_0^T\ig \left(\underbrace{Q(\TP\eta)\TP v}_{\p_t(Q(\TP\eta))}\TP v+Q(\TP\eta)\TP\p_t v\right)\p_t^2\TP v\p_t^3\TP v\\
&+3\int_0^T\ig Q(\TP\eta)\TP v\p_t^3\TP v\p_t^3\TP v+\ig Q(\TP\eta)\TP v\p_t^2\TP v\p_t^3\TP v\dS\bigg|_{0}^{T}\\
\lesssim&\PP_0+\int_0^T \PP+|Q(\TP\eta)\TP v|_{L^{\infty}}|\p_t^2 v|_1|\p_t^3 v|_1\\
\lesssim&\eps\|\p_t^3 v\|_{1.5}^2+\PP_0+\int_0^T\PP.
\end{aligned}
\end{equation}
Similarly, by plugging $\p_t^3(Q(\TP\eta))=Q(\TP\eta)(\TP\p_tv\TP v\TP v+\TP\p_t v\TP v+\TP\p_t^2 v)$ into $I_{05}+I_{06}$, we get
\begin{equation}\label{I056}
\begin{aligned}
I_{05}+I_{06}=&\int_0^T\ig\p_t^2(Q(\TP\eta))\p_t\TP v\p_t^4\TP v\dS\dt\\
\overset{\p_t}{=}&-\int_0^T\ig\p_t^3(Q(\TP\eta))\p_t\TP v\p_t^3\TP v\dS\dt-\int_0^T\ig \p_t^2(Q(\TP\eta))\p_t^2\TP v\p_t^3\TP v+\ig\p_t^2(Q(\TP\eta))\p_t\TP v\p_t^3\TP v\bigg|^T_0\\
\lesssim&\PP_0+\int_0^T \PP+|\TP v|_{L^{\infty}}^2|\p_t^3\TP v|_0|\TP \p_tv|_0\\
\lesssim&\PP_0+\int_0^T \PP+\eps\|\p_t^3 v\|_{1.5}^2+\|\p_t v\|_{1.5}^4+\|v\|_3^8\\
\lesssim&\PP_0+\int_0^T \PP+\eps\|\p_t^3 v\|_{1.5}^2.
\end{aligned}
\end{equation}

Following the same way as above, we can control $I_{07}+I_{08}$ by $\PP_0+\int_0^T \PP+\eps\|\p_t^3 v\|_{1.5}^2$ so we omit the details. Combining this with \eqref{I0e}, \eqref{I01}, \eqref{I02}-\eqref{I056}, we get the estimates of $I_0$ by 
\begin{equation}\label{I0}
I_0+\left|\TP\left(\Pi\p_t^3 v\right)\right|_0^2\lesssim \eps\|\p_t^3 v\|_{1.5}^2+\PP_0+\PP\int_0^T\PP.
\end{equation}

Now the only term left to control in \eqref{II} is $L$.  Expanding $[\p_t^4, \Ak^{\mu\alpha}]$,  we have
\begin{equation}\label{L20}
\begin{aligned}
L=&\int_0^T\io\p_t^4\Ak^{\mu\alpha}\p_{\mu}v_{\alpha}\p_t^4 q\dy\dt+4\int_0^T\io\p_t^3\Ak^{\mu\alpha}\p_t\p_{\mu}v_{\alpha}\p_t^4 q\dy\dt\\
&+6\int_0^T\io\p_t^2\Ak^{\mu\alpha}\p_t^2\p_{\mu}v_{\alpha}\p_t^4 q\dy\dt+4\int_0^T\io\p_t\Ak^{\mu\alpha}\p_t^3\p_{\mu}v_{\alpha}\p_t^4 q\dy\dt\\
=:&L_{21}+L_{22}+L_{23}+L_{24}.
\end{aligned}
\end{equation}
Despite having the right amount of derivatives, there is no direct control of $\|\p_t^4 q\|_0$ and so we have to make some extra efforts to control $L_{21},\cdots, L_{24}$. 

The hardest term to treat here is $L_{21}$. Since
\begin{equation}
\begin{aligned}
\p_t^4\Ak^{\mu\alpha}\p_\mu v_\alpha=\p_t^4 (\tilde{J} \ak^{\mu\alpha})\p_\mu v_\alpha=&\tilde{J}(\p_t^4 \ak^{\mu\alpha})\p_\mu v_\alpha + (\p_t^4 \tilde{J})\underbrace{\ak^{\mu\alpha}\p_\mu v_\alpha}_{=0}+\text{lower-order terms}\\
=&-\ak^{\mu\nu}\p_{\beta}\p_t^3\vk_{\nu}\Ak^{\beta\alpha}\p_\mu v_\alpha+ \text{lower-order terms}
\end{aligned}
\end{equation}
we have
\begin{align} 
L_{21}\LL\int_0^T\io\ak^{\mu\nu}\p_{\beta}\p_t^3\vk_{\nu}\Ak^{\beta\alpha}\p_{\mu}v_{\alpha}\p_t^4 q.
\label{L210'}
\end{align}
Since 
$$
\Ak^{\beta \alpha} \p_t^4 q = \p_t^4 (\Ak^{\beta\alpha} q)-(\p_t^4\Ak^{\beta\alpha}) q-4(\p_t^3 \Ak^{\beta\alpha}) \p_t q-6(\p_t^2 \Ak^{\beta\alpha}) \p_t^2 q - 4(\p_t \Ak^{\beta\alpha}) \p_t^3 q,
$$
and thus one can write the RHS of \eqref{L210'} as
\begin{align*}
&\int_0^T\io\ak^{\mu\nu}\p_{\beta}\p_t^3\vk_{\nu}\p_{\mu}v_{\alpha}\p_t^4(\Ak^{\beta\alpha} q) - 4\int_0^T\io\ak^{\mu\nu}\p_{\beta}\p_t^3\vk_{\nu}\p_{\mu}v_{\alpha}\p_t^3\Ak^{\beta\alpha} \p_t q\\
&-6\int_0^T\io\ak^{\mu\nu}\p_{\beta}\p_t^3\vk_{\nu}\p_{\mu}v_{\alpha}\p_t^2\Ak^{\beta\alpha} \p_t^2q-4\int_0^T\io\ak^{\mu\nu}\p_{\beta}\p_t^3\vk_{\nu}\p_{\mu}v_{\alpha}\p_t\Ak^{\beta\alpha} \p_t^3q\\
=:&L_{211}+L_{212}+L_{213}+L_{214}.
\end{align*}
It is not hard to see that $L_{212},L_{213},L_{214}$ can all be controlled directly by $\int_0^T\PP$ thanks to \eqref{ellq}. To treat $L_{211}$, we integrate $\p_\beta$ by parts and get
\begin{align*}
\int_0^T\ig\ak^{\mu\nu}\p_t^3\vk_{\nu}\p_{\mu}v_{\alpha}\p_t^4(\Ak^{3\alpha} q)- \int_0^T\io\p_t^3\vk_{\nu}\p_\beta\Big(\ak^{\mu\nu}\p_{\mu} v_{\alpha}\p_t^4(\Ak^{\beta\alpha} q)\Big)=L_{2111}+L_{2112}. 
\end{align*}
Since $L_{2112}\LL -\int_0^T\io\p_t^3\vk_{\nu}\ak^{\mu\nu}\p_{\mu} v_{\alpha}\p_t^4\p_\beta(\Ak^{\beta\alpha} q)$,
we integrate $\p_t$ by parts in the last term and get
\begin{align}
-\io\p_t^3\vk_{\nu}\ak^{\mu\nu}\p_{\mu} v_{\alpha}\p_t^3\p_\beta(\Ak^{\beta\alpha} q) \bigg|_0^T+\int_0^T\io \p_t\Big(\p_t^3\vk_{\nu}\ak^{\mu\nu}\p_{\mu} v_{\alpha}\Big)\p_t^3\p_\beta(\Ak^{\beta\alpha} q)=L_{21121}+L_{21122}.
\end{align}
Now, since $\p_\beta \Ak^{\beta\alpha}=0$, we can write 
\begin{equation}
\p_t^3 \p_\beta (\Ak^{\beta\alpha} q) = -\p_t^4 v^\alpha+\p_t^3\bp^2\eta^\alpha.
\end{equation}
 In light of this, we have
\begin{align}
L_{21122} \leq \int_0^T\PP. 
\end{align}
Also, 
\begin{align*}
L_{21121} =& -\io \p_t^3\vk_{\nu}\ak^{\mu\nu}\p_{\mu} v_{\alpha}(-\p_t^4 v^\alpha+\p_t^3 \bp^2\eta^\alpha)\bigg|^T_{0}\\
\lesssim& \PP_0+ \|\ak^{\mu\nu}\p_\mu v_\alpha\|_{L^\infty}\|\p_t^3 v\|_0(\|\p_t^4 v\|_0+\|b_0\|_{L^\infty}\|\p_t^3 \bp\eta\|_1) \\
\lesssim& \PP_0 + \eps(\|\p_t^4 v\|_{0}^2+\|\p_t^3 \bp\eta\|_1^2)+P(\|\p_t^3 v\|_0, \|\ak^{\mu\nu}\p_\mu v_\alpha\|_{2}) \\
\leq &\PP_0 +\eps(\|\p_t^4 v\|_{1}^2+\|\p_t^3 \bp\eta\|_1^2)+ \PP\int_0^T\PP.
\end{align*}

Moreover, by plugging the boundary condition \eqref{BC1} to $L_{2111}$ we obtain
\begin{align*}
-\sigma\int_0^T\ig\ak^{\mu\nu}\p_t^3\vk_{\nu}\p_{\mu}v_{\alpha}\p_t^4(\sqrt{g}\Delta_g \eta\cdot \nnk \nnk^\alpha)+\kk\int_0^T\ig\ak^{\mu\nu}\p_t^3\vk_{\nu}\p_{\mu}v_{\alpha}\p_t^4\Big(\left((1-\TL) (v\cdot \nnk)\right)\nnk^\alpha\Big)=L_{21111}+L_{21112}.
\end{align*}
Invoking \eqref{lapg}, we have
\begin{align*}
L_{21111}= &-\sigma\int_0^T\ig\ak^{\mu\nu}\p_t^3\vk_{\nu}\p_{\mu}v_{\alpha}\p_t^4(\sqrt{g}g^{ij}\TP_i\TP_j\eta\cdot \nnk \nnk^\alpha)\\
&+\sigma\int_0^T\ig\ak^{\mu\nu}\p_t^3\vk_{\nu}\p_{\mu}v_{\alpha}\p_t^4(\sqrt{g}g^{ij}g^{kl}\TP_l\eta^{\mu}\TP_i\TP_j\eta_{\mu}\TP_k\eta\cdot \nnk \nnk^\alpha).
\end{align*}
It suffices to control the first term only since the second term has the highest order contribution with the same type of integrand. Also, 
\begin{align*}
-\sigma\int_0^T\ig\ak^{\mu\nu}\p_t^3\vk_{\nu}\p_{\mu}v_{\alpha}\p_t^4(\sqrt{g}g^{ij}\TP_i\TP_j\eta\cdot \nnk \nnk^\alpha)\LL& -\sigma\int_0^T\ig\ak^{\mu\nu}\p_t^3\vk_{\nu}\p_{\mu}v_{\alpha}\sqrt{g}g^{ij}\TP_i\TP_j \p_t^3v\cdot \nnk \nnk^\alpha\\
&-\sigma\int_0^T\ig\ak^{\mu\nu}\p_t^3\vk_{\nu}\p_{\mu}v_{\alpha}\sqrt{g}g^{ij}\TP_i\TP_j\eta\cdot \nnk (\p_t^4 \nnk^\alpha).
\end{align*}
Now, since 
\begin{align}
\p_t^4 \nnk = Q(\TP\ek)\TP \p_t^3 v\cdot \nnk + \text{lower-order terms}, 
\label{4 time derivatives on nn}
\end{align}
and so we have, after using the Sobolev embedding and trace lemma, that
\begin{align}
\sigma\int_0^T\ig\left|\ak^{\mu\nu}\p_t^3\vk_{\nu}\p_{\mu}v_{\alpha}\sqrt{g}g^{ij}\TP_i\TP_j\eta\cdot \nnk (\p_t^4 \nnk^\alpha)\right| \leq \int_0^T\PP.
\end{align}
In addition, by integrating $\TP_i$ by parts and then using the trace lemma, we have
\begin{align}
\sigma\int_0^T\ig\left|\ak^{\mu\nu}\p_t^3\vk_{\nu}\p_{\mu}v_{\alpha}\sqrt{g}g^{ij}\TP_i\TP_j \p_t^3v\cdot \nnk \nnk^\alpha\right|\leq \int_0^T\PP. 
\end{align}
Moreover, we still need to control $L_{21112}$. In light of \eqref{4 time derivatives on nn}, we only need to study the case when all four time derivatives land on $\TL v$, i.e., 
\begin{align*}
-\kk\int_0^T\ig\ak^{\mu\nu}\p_t^3\vk_{\nu}\p_{\mu}v_{\alpha}\TL (\p_t^4v\cdot \nnk)\nnk^\alpha.
\end{align*}
Integrating $\TP$ by parts, this term has the contributes to 
\begin{align*}
\kk\int_0^T\ig\ak^{\mu\nu}\p_t^3\TP\vk_{\nu}\p_{\mu}v_{\alpha}\TP\p_t^4v\cdot \nnk\nnk^\alpha,
\end{align*}
up to terms with the same type integrand, whose analysis (and bound) is identical. To control the main term, one has
\begin{align*}
&\kk\int_0^T\ig\ak^{\mu\nu}\p_t^3\TP\vk_{\nu}\p_{\mu}v_{\alpha}\TP\p_t^4v\cdot \nnk\nnk^\alpha =\sqrt{\kk} \int_0^T\ig Q(\TP\ek, \p v)\TP\p_t^3 v \sqrt{\kk}\TP \p_t^4 v\\
\leq& \sqrt{\kk}\int_0^T Q(\|\TP\ek\|_{L^\infty}, \|\p v\|_{L^\infty}) \|\p_t^3 v\|_{1.5} \|\sqrt{\kk} \p_t^4 v\|_{1.5}\\
\leq& \frac{1}{2}\bigg(\sqrt{\kk}\int_0^TQ(\|\TP\ek\|_{L^\infty}, \|\p v\|_{L^\infty})\|\p_t^3v\|_{1.5}^2+\int_0^T \|\sqrt{\kk}\p_t^4v\|_{1.5}^2\bigg)\\
\leq& \sqrt{\kk} E_{\kk}^{(3)}+\int_0^T\PP .
\end{align*}

Finally, combining \eqref{tgt40} with the computations above, we finally get the control of full time derivatives
\begin{equation}\label{tgt4}
\left\|\p_t^4 v\right\|_0^2+\left\|\p_t^4\bp\eta\right\|_0^2+\left|\TP\left(\Pi\p_t^3 v\right)\right|_0^2\lesssim E_\kk^{(3)}+(E_{\kk}^{(3)})^2+ \PP_0+C(\eps)E_\kk(T)+\PP\int_0^T\PP.
\end{equation}

\subsection{Control of mixed space-time tangential derivatives}\label{sect at least one TP}

To finish the control of $E_\kk(T)$, it remains to study the tangential energies generated by the $\TP\p_t^3$, $\TP^2\p_t^2$, $\TP^3\p_t$ and $\TP^3 \bp$-differentiated $\kk$-problem. Generally speaking, the energy estimate becomes much simpler when the tangential spatial derivative(s) $\TP$ is taken into account. This is due to that we can avoid the higher-order terms in the interior, i.e., terms associated with $I_{11}$ in \eqref{I11}. This can be done by having all top-order terms on the boundary, and those terms can be controlled thanks to the extra $0.5$ interior regularity.

\noindent\textbf{The $\TP\p_t^3$-tangential energy:} Similar to \eqref{tgt40}, we have
\begin{equation}\label{tgt30}
\begin{aligned}
&\frac{1}{2}\int_0^T \frac{d}{dt}\io\left|\TP\p_t^3 v\right|_0^2+\left|\TP\p_t^3\bp\eta\right|^2\dy\dt\\
=&\underbrace{-\int_0^T\io\TP\p_t^3(\Ak^{\mu\alpha}\p_{\mu}q)\TP\p_t^3v_{\alpha}\dy\dt}_{I^*}\\
&+\int_0^T\io\TP\p_t^3\bp^2\eta_{\alpha}\TP\p_t^3v_{\alpha}\dy\dt+\int_0^T\io\TP\p_t^3\bp\eta\alpha\TP\p_t^3\bp v_{\alpha}\dy\dt.
\end{aligned}
\end{equation}
By integrating $\bp$ by parts in the second term, we can get the cancellation with the third term at the top order
\begin{equation}\label{tgb3}
\begin{aligned}
&\int_0^T\io\TP\p_t^3\bp^2\eta_{\alpha}\TP\p_t^3v_{\alpha}\dy\dt+\int_0^T\io\TP\p_t^3\bp\eta\alpha\TP\p_t^3\bp v_{\alpha}\dy\dt\\
=&-\int_0^T\io\TP\p_t^3\bp\eta_{\alpha}\TP\p_t^3\bp v_{\alpha}\dy\dt+\int_0^T\io\TP\p_t^3\bp\eta_\alpha\TP\p_t^3\bp v_{\alpha}\dy\dt\\
&+\int_0^T\io\left[\TP,\bp\right]\p_t^3\bp\eta^{\alpha}\cdot\TP\p_t^3v_{\alpha}-\TP\p_t^3\bp\eta^{\alpha}\cdot\left[\bp,\TP\right]\p_t^3v_{\alpha}\dy\dt\\
\lesssim&\int_0^TP(\|b_0\|_{3},\|\p_t^3v\|_1,\|\p_t^2 v\|_2)\dt
\end{aligned}
\end{equation}

The main term $I^*$ is treated a bit different compare to $I$ in \eqref{II}. Specifically, one commutes $\Ak^{\mu\alpha}$ with $\TP\p_t^3$ first and then integrates by parts. This allows us to avoid the appearance of the higher-order interior terms. 

\begin{equation}\label{IIs}
\begin{aligned}
I^*=&-\int_0^T\io\TP\p_t^3 v_{\alpha}\Ak^{\mu\alpha}\TP\p_t^3\p_{\mu}q\underbrace{-\int_0^T\io\TP\p_t^3 v_{\alpha}~\left[\TP\p_t^3,\Ak^{\mu\alpha}\right]\p_{\mu}q}_{L_1^*}\\
\overset{\p_\mu}{=}&\int_0^T\io\Ak^{\mu\alpha}\TP\p_t^3\p_{\mu}v_{\alpha}\TP\p_t^3q\underbrace{-\int_0^T\ig\TP\p_t^3v_{\alpha}\Ak^{3\alpha}\TP\p_t^3 q}_{I_B^*}\underbrace{+\int_0^T\int_{\Gamma_0}\TP\p_t^3v_{\alpha}\Ak^{3\alpha}\TP\p_t^3 q}_{I_B^{**}}+L_1^*\\
=&\int_0^T\io\underbrace{\TP\p_t^3(\divA v)}_{=0}\TP\p_t^3 q\underbrace{+\int_0^T\io\left[\Ak^{\mu\alpha},\TP\p_t^3\right]\p_{\mu}v_{\alpha}~\TP\p_t^3 q}_{L_2^*}+I_B^*+I_B^{**}+L_1^*.
\end{aligned}
\end{equation}
Here, $I_B^{**}=0$ because $\Ak^{13}=\Ak^{23}=0$, $\Ak^{33}=1$, and $v_3=0$ on $\Gamma_0$. Also,  $L_1^*$ and $L_2^*$ can be directly controlled. For simplicity we only list the computation of the highest order terms
\begin{equation}\label{L1s}
\begin{aligned}
L_1^*=&-\int_0^T\io\TP\p_t^3 v_{\alpha}~\left[\TP\p_t^3,\Ak^{\mu\alpha}\right]\p_{\mu}q\dy\dt\\
\overset{L}{=}&-\int_0^T\io\TP\p_t^3 v_{\alpha}~\TP\p_t^3\Ak^{\mu\alpha}\p_{\mu}q\dy\dt\lesssim \int_0^T\PP\dt.
\end{aligned}
\end{equation}
and
\begin{equation}\label{L2s}
\begin{aligned}
L_2^*=&\int_0^T\io\left[\Ak^{\mu\alpha},\TP\p_t^3\right]\p_{\mu}v_{\alpha}~\TP\p_t^3 q\dy\dt\\
\overset{L}{=}&\int_0^T\io\TP\p_t^3\Ak^{\mu\alpha}\p_{\mu}v_{\alpha}~\TP\p_t^3 q\dy\dt\lesssim \int_0^T\PP\dt.
\end{aligned}
\end{equation}

Next we analyze the boundary integral $I_B^*$.
\begin{equation}\label{IBs0}
\begin{aligned}
I_B^*=&-\int_0^T\ig\TP\p_t^3v_{\alpha}\TP\p_t^3(\Ak^{3\alpha}q)\dS\dt\\
&+\int_0^T\ig\TP\p_t^3v_{\alpha}\TP\p_t^3\Ak^{3\alpha}q\dS\dt+\int_0^T\ig\TP\p_t^3v_{\alpha}\p_t^3\Ak^{3\alpha}\TP q\dS\dt\\
&+3\int_0^T\ig\TP\p_t^3v_{\alpha}\TP\p_t^2\Ak^{3\alpha}\p_tq\dS\dt+3\int_0^T\ig\TP\p_t^3v_{\alpha}\p_t^2\Ak^{3\alpha}\TP\p_tq\dS\dt\\
&+3\int_0^T\ig\TP\p_t^3v_{\alpha}\TP\p_t\Ak^{3\alpha}\p_t^2q\dS\dt+3\int_0^T\ig\TP\p_t^3v_{\alpha}\p_t\Ak^{3\alpha}\TP\p_t^2q\dS\dt\\
&+\int_0^T\ig\TP\p_t^3v_{\alpha}\TP\Ak^{3\alpha}\p_t^3q\dS\dt\\
=:&J_0+J_1+\cdots+J_7.
\end{aligned}
\end{equation}
Since we have $H^{1.5}(\Omega)$ regularity for $\p_t^3 v$ and $H^1(\Omega)$ regularity for $\p_t^3 q$, the top order terms contributed by $J_1$ to $ J_7$ can all be directly controlled by the trace lemma. In the end, we have
\begin{equation}\label{J1}
J_1+\cdots+J_7\lesssim\int_0^T\PP.
\end{equation}
By plugging the boundary condition $$\Ak^{3\alpha}q = -\sigma \sqrt{g}(\Delta_g \eta\cdot \nnk )\nnk^\alpha +\kk\left((1-\TL) (v\cdot \nnk)\right)\nnk^\alpha$$ in $J_0$, we obtain
\begin{equation}\label{J0s}
\begin{aligned}
\frac{1}{\sigma} J_{0}=&\int_0^T\ig\TP\p_t^3(\sqrt{g}\Delta_g\eta\cdot \nnk\nnk^\alpha)\TP\p_t^3 v_{\alpha}\dS\dt-\frac{\kk}{\sigma}\int_0^T\ig\TP\p_t^3[\left((1-\TL) (v\cdot \nnk)\right)\nnk^\alpha]\TP\p_t^3v_{\alpha}\dS\dt
\end{aligned}
\end{equation}

For the second term, after integrating one $\TP$ by parts, it contributes to the positive energy term (after moving to the LHS)
\begin{equation}
\frac{\kk}{\sigma} \int_0^T \ig |\p_t^3 v\cdot \nnk|_2^2\dS\dt,
\end{equation}
and some error terms. Here, the most difficult error term reads
\begin{align}
\frac{\kk}{\sigma} \int_0^T\int_{\Gamma} (\TP^2 \p_t^3 v\cdot \nnk)(v\cdot \p_t^3 \TP^2 \nnk)\dS\dt 
\label{boundary error 2}
\end{align}
which can be treated as follows:
\begin{align*}
&\frac{\kk}{\sigma} \int_0^T\int_{\Gamma} (\TP^2 \p_t^3 v\cdot \nnk)(v\cdot \p_t^3 \TP^2 \nnk)\dS\dt \LL \frac{\kk}{\sigma} \int_0^T\int_{\Gamma} (\TP^2 \p_t^3 v\cdot \nnk)(v\cdot \TP^3\p_t^2 \vk\cdot \nnk)\dS\dt\\
\leq& \int_0^T P(|\TP\ek|_{L^\infty{(\Gamma)}}, |v|_{L^\infty{(\Gamma)}})|\sqrt{\kk} \TP^2 \p_t^3 v|_0|\sqrt{\kk}\TP^3\p_t^2 v\cdot \nnk|_0 \\
\lesssim& \int_0^T |\sqrt{\kk}\TP^2\p_t^3 v|_0^2 + \sup_t P(|\TP\ek|_{L^\infty{(\Gamma)}}, |v|_{L^\infty{(\Gamma)}}) +\Big(\int_0^T |\sqrt{\kk} \TP^3\p_t^2 v\cdot \nnk|_0^2\Big)^2\\
\lesssim& \int_0^T \|\sqrt{\kk}\p_t^3 v\|_{2.5}^2 + \Big(\int_0^T \|\sqrt{\kk} \p_t^2 v\|_{3.5}^2\Big)^2 + \sup_t P(|\TP\ek|_{L^\infty{(\Gamma)}}, |v|_{L^\infty{(\Gamma)}}) \\
\leq& E_\kk^{(3)}+ (E_{\kk}^{(3)})^2 + \PP_0+\PP\int_0^T\PP .
\end{align*}

The first term in \eqref{J0s} is treated analogous to the first term in \eqref{boundary term p_t4}. The main term we need to study in this case reads
\begin{align*}
&\int_0^T\ig (\TP\p_t^3 (\sqrt{g}\Delta_{g} \eta^\alpha)) (\TP\p_t^3 v)\dS\dt\\
=&\int_0^T\ig\TP\p_t^2\TP_i\left(\sqrt{g}g^{ij}\Pi^{\alpha}_{\lambda}\TP_jv^{\lambda}+\sqrt{g}(g^{ij}g^{kl}-g^{lj}g^{ik})\TP_j\eta^{\alpha}\TP_k\eta^{\lambda}\TP_lv_{\lambda}\right)\TP\p_t^3 v^{\lambda}\dS\dt
\end{align*}
Integrating $\TP_i$ by parts, we get
\begin{equation}\label{J00s}
\begin{aligned}
J_{00}\overset{\TP_i}{=}&-\int_0^T\ig\sqrt{g}g^{ij}\Pi^{\alpha}_{\lambda}\TP\p_t^2\TP_jv^{\lambda}\TP\p_t^3\TP_iv_{\alpha}\dS\dt\\
&-\int_0^T\ig \sqrt{g}(g^{ij}g^{kl}-g^{lj}g^{ik})\TP_j\eta^{\alpha}\TP_k\eta^{\lambda}\TP\TP_l\p_t^2v_{\lambda}\p_t^3\TP_i\TP v_{\alpha}\dS\dt+\RR_0\\
&=:J_{01}+J_{02}+\RR_0,
\end{aligned}
\end{equation}
where $\RR_0$ consists terms that can be treated in the same way as in $I_{03},\cdots, I_{08}$ in \eqref{I0e}. 

In $J_{01}$, we can integrate $\p_t$ by parts and mimic the proof of \eqref{I010} to get
\begin{equation}\label{J01}
J_{01}+\left|\TP^2(\Pi\p_t^2 v)\right|_0^2\lesssim\eps\left(\left|\TP(\Pi\TP\p_t^2 v)\right|_0^2+\|\p_t^2 v\|_{2.5}^2\right)+\PP_0+\PP\int_0^T\PP\dt
\end{equation}
$J_{02}$ can also be controlled similarly as $I_{02}$. We find that the integrand is zero if $l=i$. So it suffices to compute the case $(l,i)=(1,2)$ and $(2,1)$. Similarly we get
\begin{equation}\label{J021}
J_{02}=\ig\frac{1}{\sqrt{g}}\det
\begin{bmatrix}
\TP_1\eta_{\mu}\TP_1\p_t^2\TP v^{\mu}&\TP_1\eta_{\mu}\TP_2\p_t^2\TP v^{\mu}\\
\TP_2\eta_{\mu}\TP_1\p_t^2\TP v^{\mu}&\TP_2\eta_{\mu}\TP_2\p_t^2\TP v^{\mu}
\end{bmatrix}\dS\bigg|^T_0+\int_0^T\PP+\RR.
\end{equation}
The main term can be computed as follows
\begin{equation}\label{J022}
\begin{aligned}
&\ig\frac{1}{\sqrt{g}}\det
\begin{bmatrix}
\TP_1\eta_{\mu}\TP_1\p_t^2\TP v^{\mu}&\TP_1\eta_{\mu}\TP_2\p_t^2\TP v^{\mu}\\
\TP_2\eta_{\mu}\TP_1\p_t^2\TP v^{\mu}&\TP_2\eta_{\mu}\TP_2\p_t^2\TP v^{\mu}
\end{bmatrix}\dS\\
=&\ig\frac{1}{\sqrt{g}}\left(\TP_1\eta_{\mu}\TP_1\p_t^2\TP v^{\mu}\TP_2\eta_{\mu}\TP_2\p_t^2\TP v^{\mu}-\TP_1\eta_{\mu}\TP_2\p_t^2\TP v^{\mu}\TP_2\eta_{\mu}\TP_1\p_t^2\TP v^{\mu}\right)\\
\overset{\TP_1,\TP_2}{=}&\ig Q_{\mu\lambda}^i(\TP\eta,\TP^2\eta)\TP\p_t^2 v^{\mu}\TP_i\TP\p_t^2v^{\lambda}\dS\\
\lesssim&P\left(|\TP\eta|_{L^{\infty}},|\TP^2\eta|_{L^{\infty}}\right)|\TP\p_t^2 v|_0\TP^2\p_t^2v|_0\\
\lesssim&\eps\|\p_t^2 v\|_{2.5}^2+\PP_0+\int_0^T\PP,
\end{aligned}
\end{equation}and thus we get the control of $J_{02}$
\begin{equation}\label{J02}
J_{02}\lesssim\eps\left(\|\p_t^2 v\|_{2.5}^2\right)+\PP_0+\PP\int_0^T\PP.
\end{equation}

Combining \eqref{tgt30}-\eqref{J01} and \eqref{J02}, we get the $\TP\p_t^3$-tangential estimates as follows
\begin{equation}\label{tgt3}
\left\|\TP\p_t^3 v\right\|_0^2+\left\|\TP\p_t^3\bp\eta\right\|_0^2+\left|\TP\left(\Pi\TP\p_t^2 v\right)\right|_0^2\lesssim E_\kk^{(3)}+ (E_{\kk}^{(3)})^2+\eps\|\p_t^2 v\|_{2.5}^2+\PP_0+\PP\int_0^T\PP.
\end{equation}
\noindent\textbf{The $\TP^2\p_t^2$, $\TP^3\p_t$ and $\TP^3 \bp$-tangential energies:} The control of the other tangential energies that involving at least one $\TP$ is follows from the arguments above by replacing $\TP\p_t^3$ to the corresponding derivatives. Hence, we shall omit the details and only illustrate the major differences.

First, we mention that the derivatives $\TP^3\p_t$ and $\TP^3\bp$ behave the same since both $v$ and $\bp\eta$ are of the same interior regularity.

Second, one needs to pay attention to the terms that analogous to the error term generated by \eqref{J0s} during the construction of the energy term. In particular, we need to study the top order error term analogous to \eqref{boundary error 2}. Setting $\dd=\p_t, \TP$ or $\bp$, and so $\TP^2\p_t^2, \TP^3\p_t, \TP^3\bp$ can be denoted systematically by $\TP^2\dd^2$.  Now we consider 
\begin{equation}\label{boundary error 3}
\frac{\kk}{\sigma} \int_0^T\int_{\Gamma} (\TP^3 \dd^2 v\cdot \nnk)(v\cdot  \TP^3 \dd^2 \nnk)\dS\dt.
\end{equation}
When $\dd^2=\p_t^2$ then \eqref{boundary error 3} is treated similar to \eqref{boundary error 2}. This is due to that 
$$
\TP^3 \p_t^2\nnk = Q(\TP\ek)\TP^4 \p_t \vk\cdot \nnk+\text{lower-order terms},
$$
and $\int_0^T|\p_t^4 v|_{1.5}^2$ is included in $E_{\kk}^{(3)}$.
In the end, we obtain
\begin{align*}
\frac{\kk}{\sigma} \int_0^T\int_{\Gamma} (\TP^3 \p_t^2 v\cdot \nnk)(v\cdot  \TP^3 \p_t^2 \nnk)\dS\dt\leq  E_\kk^{(3)}+ (E_{\kk}^{(3)})^2+\PP_0+\PP\int_0^T\PP.
\end{align*}

On the other hand, when $\dd^2=\TP\p_t, \TP\bp$, then using the fact that $\dd \nnk = Q(\TP\ek) \dd\TP \ek\cdot \nnk$, we have
\begin{align}
\frac{\kk}{\sigma} \int_0^T\int_{\Gamma} (\TP^4\p_t v\cdot \nnk)(v\cdot  \TP^4\p_t \nnk)\dS\dt \LL\frac{\kk}{\sigma} \int_0^T\int_{\Gamma} (\TP^4\p_t v\cdot \nnk)(v\cdot  \TP^5 v\cdot \nnk)\dS\dt,\\
\frac{\kk}{\sigma} \int_0^T\int_{\Gamma} (\TP^4\bp v\cdot \nnk)(v\cdot  \TP^4\bp \nnk)\dS\dt \LL\frac{\kk}{\sigma} \int_0^T\int_{\Gamma} (\TP^4\bp v\cdot \nnk)(v\cdot  \TP^5\bp \ek\cdot \nnk)\dS\dt.
\end{align}
The terms on the RHS requires $\int_0^T |\sqrt{\kk} v|_5^2$ and $\int_0^T |\sqrt{\kk} \bp\eta|_5^2$, respectively, to control. However, owing to \eqref{super eta} and \eqref{super bp eta}, both of them can be controlled by
$
\MM_0+C(\eps)E_\kk(T)+\PP\int_0^T\PP. 
$
Hence, 
\begin{align}
\left\|\TP^2\p_t^2 v\right\|_0^2+\left\|\TP^2\p_t^2\bp\eta\right\|_0^2+\left|\TP\left(\Pi\TP^2\p_t v\right)\right|_0^2\lesssim E_\kk^{(3)}+ (E_{\kk}^{(3)})^2+\eps E_{\kk}(T)+\PP_0+\PP\int_0^T\PP,\label{tgt5}\\
\left\|\TP^3\p_t v\right\|_0^2+\left\|\TP^3\p_t\bp\eta\right\|_0^2+\left|\TP\left(\Pi\TP^3 v\right)\right|_0^2\lesssim E_\kk^{(3)}+ (E_{\kk}^{(3)})^2+\eps E_{\kk}(T)+\MM_0+\PP\int_0^T\PP,\label{tgt6}\\
\left\|\TP^3\bp v\right\|_0^2+\left\|\TP^3\bp^2\eta\right\|_0^2+\left|\TP\left(\Pi\TP^3 \bp \eta\right)\right|_0^2\lesssim E_\kk^{(3)}+ (E_{\kk}^{(3)})^2+\eps E_{\kk}(T)+\MM_0+\PP\int_0^T\PP.\label{tgt7}
\end{align}
Notice that the RHS of \eqref{tgt6} and \eqref{tgt7} rely on $\MM_0$, which is given in Lemma \ref{super control}. In Section \ref{close}, in fact, we are able to control $\MM_0$ by $\mathcal{C}(\|v_0\|_{4.5},\|b_0\|_{4.5},|v_0|_5)$. 

\section{Estimates for the higher order weighted interior norms}
\label{sect E3kk}
It remains to control $E_\kk^{(3)}(T)$ in order to complete the proof of Proposition \ref{nonlinearkk}. 
\subsection{Full time derivatives}\label{sect E3kk1}
We shall first study the first two terms, i.e., 
$$
\int_0^T \Big (\big\|\sqrt{\kk} \p_t^4 v\big\|_{1.5}^2+\big\|\sqrt{\kk} \p_t^4 \bp \eta\big\|_{1.5}^2\Big)\dt=K_1+K_2.
$$
These terms appear to be the most difficult ones to control. In particular, they yield error terms that contribute to the top order and can only be controlled in $L^2(0, T)$ instead of $L^{\infty}(0, T)$. In other words, we cannot use the time integral to create terms that can be controlled by $\PP\int_0^T\PP$. 

The goal is to show:
\begin{align}
K_1+K_2 \leq \PP_0 +C(\eps)E_\kk(T)+\PP\int_0^T \PP.
\label{Gronwall''}
\end{align}
The control of $K_1, K_2$ relies on the div-curl estimate and so the $H^1$-norms of $\p_t^4 v$ and $\p_t^4 \bp \eta$ have to be studied together owing to the strong coupling structure of the MHD equations. In particular, 
\begin{align}
K_1\leq &\int_0^T \left(\big\|\sqrt{\kk}\di \p_t^4 v\big\|_{0.5}^2+\big\|\sqrt{\kk}\curl \p_t^4 v\big\|_{0.5}^2+ \big|\sqrt{\kk}\p_t^4 v^3\big|_{1}^2\right)\dt=:K_{11}+K_{12}+K_{13},\\
K_2 \leq &\int_0^T\left(\big\|\sqrt{\kk}\di \p_t^4 \bp\eta\big\|_{0.5}^2+\big\|\sqrt{\kk}\curl \p_t^4 \bp\eta\big\|_{0.5}^2+\big|\sqrt{\kk}\p_t^4 \bp \eta^3\big|_{1}^2\right)\dt=:K_{21}+K_{22}+K_{23}. 
\end{align} 

\noindent\textbf{Bound for $K_{13}$ and $K_{23}$:} 

For $K_{13}$, there holds
\begin{align*}
K_{13} \leq \underbrace{\int_0^T |\sqrt{\kk} \p_t^4 v\cdot \nnk|_{1}^2}_{\leq E_{\kk}^{(2)}} +\int_0^T |\sqrt{\kk} \p_t^4 v\cdot (N-\nnk)|_{1}^2,
\end{align*}
and for the error term, we have
\begin{align*}
\int_0^T |\sqrt{\kk} \p_t^4 v\cdot (N-\nnk)|_{1}^2 \lesssim \int_0^T \|\sqrt{\kk}\p_t^4 v\|_{1.5}^2 \cdot |N-\nnk|_{1+}^2 \lesssim \eps^2\PP,
\end{align*}
where \eqref{comars N and nn in H^3} is used in the last inequality. To control $K_{23}$, since $\p_t\eta =v$ we have 
$
\int_0^T\big|\sqrt{\kk}\p_t^3 \bp v^3\big|_{1}^2
$
and so it suffices to control 
$
\int_0^T\big|\sqrt{\kk}\p_t^3 v^3\big|_{2}^2.
$
This term can then be treated similarly to $K_{13}$. 

\noindent\textbf{Bound for $K_{11}$ and $K_{21}$:} First we state the following application of the Kato-Ponce inequality which shall be used frequently. Let $f\in H^{0.5}(\Omega)$ and $g$ be a smooth function. Then
\begin{align}
\|fg\|_{0.5} \lesssim \|f\|_{0.5} \|g\|_{1.5+}.
\label{strategy normal}
\end{align}
For $K_{11}$, we have
\begin{align}
K_{11}\leq\int_0^T(\|\sqrt{\kk}\di_{\ak} \p_t^4 v\|_{0.5}^2+\|\sqrt{\kk}\di_{a-\ak}\p_t^4 v\|_{0.5}^2).
\end{align} 
Since $\|a-\ak\|_{1.5+}^2\leq \kk P(\|\eta\|_{3.5})$ thanks to \eqref{lkk 333 int}, the error term can be controlled as
\begin{align}
\int_0^T \|\sqrt{\kk}\di_{a-\ak}\p_t^4 v\|_{0.5}^2\lesssim \int_0^T \|a-\ak\|_{1.5+}^2\|\sqrt{\kk}\p \p_t^4 v\|_{0.5}^2,
\label{id a-ak}
\end{align}
which can be controlled by the RHS of \eqref{Gronwall''} when $\kk$ is small. 
 For the first term, since $\diva v=0$ 
 we have
\begin{align}
\int_0^T\|\sqrt{\kk}\di_{\ak} \p_t^4 v\|_{0.5}^2 =\int_0^T \|\sqrt{\kk}[\p_t^4, \ak] \p v\|_{0.5}^2 \LL \int_0^T \|\sqrt{\kk}(\p_t\ak^{\mu\alpha}) \p_\mu\p_t^3v_\alpha\|_{0.5}^2+\|\sqrt{\kk}(\p_t^4 \ak^{\mu\alpha}) \p_\mu v_\alpha\|_{0.5}^2.
\label{div v high order}
\end{align}
It is not hard to see that that $\int_0^T \|\sqrt{\kk}\p_t\ak \p\p_t^3v\|_{0.5}^2 \leq \int_0^T\PP$ as $\p_t^3 v\in H^{1.5}(\Omega)$ a priori. In addition, since
$
\p_t \ak^{\mu\alpha}=Q(\p \ek),
$ 
we obtain
\begin{equation}
\int_0^T \|\sqrt{\kk}\p_t^4 \ak \p v\|_{0.5}^2 \LL \int_0^T \|\sqrt{\kk}Q(\p \ek)\p\p_t^3 v\p v\|_{0.5}^2 \leq \int_0^T \PP,
\end{equation}

The control of $K_{21}$ is a bit more involved. We cannot commute $\p_t^4$ to \eqref{divbeq} as this would yield $\dive_{\p_t^5\ak}\bp\eta$ on the RHS which cannot be controlled. However, by writing $\di \p_t^4\bp\eta= \di \p_t^3 \bp v$ and then we have
\begin{align}
\int_0^T\|\sqrt{\kk}\di \p_t^3 \bp v\|_{0.5}^2 \leq\int_0^T\|\sqrt{\kk}\diva \p_t^3 \bp v\|_{0.5}^2+\int_0^T\|\sqrt{\kk}\di_{\ak-a} \p_t^3 \bp v\|_{0.5}^2.  
\end{align}
The second term on the RHS is again easy to control similar to \eqref{id a-ak}. For the first term, because
\begin{align*}
\diva \p_t^3 \bp v =& \p_t^3 \diva (\bp v)-[\p_t^3, \diva] \bp v\\
=&\p_t^3 ([\diva, \bp] v)-[\p_t^3, \diva] \bp v\\
= &\sum_{0\leq i\leq 3}(\p_t^i \ak^{\mu\alpha})(\p_\mu b_0^\nu)(\p_t^{3-i}\p_\nu v_\alpha)-\sum_{0\leq i\leq 3}\p_t^i \left(b_0^\nu \p_\nu\ak^{\mu\alpha}\right)(\p_t^{3-i}\p_\mu v_\alpha)\\
&-\sum_{1\leq j\leq 3}(\p_t^j \ak^{\mu\alpha})\p_\mu(\bp \p_t^{3-j}v_\alpha) ,
\end{align*}
then it can be seen, after counting the derivatives that 
\begin{align*}
\sum_{0\leq i\leq 3}\int_0^T\|\sqrt{\kk}(\p_t^i \ak^{\mu\alpha})(\p_\mu b_0^\nu)(\p_t^{3-i}\p_\nu v_\alpha)\|_{0.5}^2,\quad \sum_{0\leq i\leq 3}\|\sqrt{\kk}\p_t^i \left(b_0^\nu \p_\nu\ak^{\mu\alpha}\right)(\p_t^{3-i}\p_\mu v_\alpha)\|_{0.5}^2,\\
 \sum_{1\leq j\leq 3}\int_0^T\|\sqrt{\kk}(\p_t^j \ak^{\mu\alpha})\p_\mu(\bp \p_t^{3-j}v_\alpha)\|_{0.5}^2
\end{align*}
 can be controlled by $\int_0^T\PP$ owing to the fact that $\p_t^k v\in H^{4.5-k}(\Omega)$, $k=2,3$. 
 
 \noindent\textbf{Bound for $K_{12}$ and $K_{22}$:}  We would like to state the following strategy that will come in handy when dealing with the leading order terms in $K_{12}$ and $K_{22}$. Let $X$ be the term such that $\int_0^T \|\sqrt{\kk}X\|_{0.5}^2$ is part of $E_\kk^{(3)}$ and $Y$ be a lower order term such that $\|Y\|_{1.5+}^2$ is controlled by $E_{\kk}^{(1)}$. Then
\begin{align}
\int_0^T \int_0^t \|\sqrt{\kk}XY\|_{0.5}^2 \dt \leq& T\int_0^T  \|\sqrt{\kk}XY\|_{0.5}^2   
\lesssim T\sup_t \|Y\|_{1.5+}^2 \int_0^T \|\sqrt{\kk}X\|_{0.5}^2\nonumber\\
\leq& \frac{\eps}{2}\Big(\int_0^T \|\sqrt{\kk}X\|_{0.5}^2\Big)^2+\frac{T^2}{2\eps} \sup_t \|Y\|_{1.5+}^4,\label{strategy}
\end{align}
which is bounded by the RHS of \eqref{Gronwall} if $T$ is sufficiently small. 

$K_{12}$ and $K_{22}$ will be considered together via studying the evolution equation verified by $\curl \p_t^4 v$ and $\curl \p_t^4 \bp \eta$. But this evolution equation cannot be derived by taking $\p_t^4$ to \eqref{curlbeq} as this would yield $\curl_{\p_t^5 \Ak} v$ in the source term which cannot be controlled. Instead, we commute $\p_t^4\curlA$ to the equation 
$
\p_t v +\bp^2\eta = \nab_{\Ak} q 
$
and get
\begin{align*}
\p_t^4 \curlA \p_t v + \p_t^4 \curlA (\bp^2\eta) = 0.
\end{align*}
This yields the following evolution equation by commuting three time derivatives through $\curlA$ in the first term on the LHS:
\begin{align}
\p_t \curlA \p_t^4 v +\curlA (\bp^2 \p_t^4 \eta) = -\p_t([\p_t^3,\curlA] \p_t v)- [\p_t^4,\curlA] \bp^2\eta :=f,
\label{time diff curl}
\end{align}
and, after expansion, the source term $f$ becomes:
\begin{align}
f= \p_t\bigg( \sum_{1\leq j\leq 3} \epsilon_{\alpha\beta\gamma}(\p_t^j \ak^{\mu\beta})\p_\mu \p_t^{4-j} v^\gamma\bigg) + \sum_{1\leq j\leq 4} \epsilon_{\alpha\beta\gamma} (\p_t^j \ak^{\mu\beta})\p_\mu \bp^2 \p_t^{4-j} \eta^\gamma.
\label{time diff curl RHS}
\end{align}
By multiplying $\kk \TP(\curl_{\ak} \p_t^4 v)$ to the evolution equation \eqref{time diff curl} and then integrating in space, we have
\begin{align}
\io \kk \p_t (\curlA \p_t^4 v)\TP (\curlA \p_t^4 v)+\io\kk  \Big(\curlA(\bp^2 \p_t^4\eta\Big)\TP(\curlA \p_t^4 v)= \io \kk f \TP(\curlA \p_t^4 v), \label{3131}
\end{align}
where the first term contributes to $\frac{1}{2}\frac{d}{dt} \|\curlA \p_t^4 v\|_{0.5}^2$ after integrating $\TP^{\frac{1}{2}}$ by parts. 

Next, if we integrate $\bp$ by parts in $\io\kk  \Big(\curlA(\bp^2 \p_t^4\eta\Big)\TP(\curlA \p_t^4 v)$ and then integrate $\TP^{\frac{1}{2}}$ by parts,  we obtain $\frac{1}{2}\frac{d}{dt} \|\curlA \p_t^4\bp\eta\|_{0.5}^2$ up to terms involving commutators (which will be recorded below). In particular,  the following energy inequality is achieved:
\begin{align}
&\frac{1}{2}\|\sqrt{\kk}\curl_{\Ak} \p_t^4 v\|_{0.5}^2+\frac{1}{2}\|\sqrt{\kk} \curl_{\Ak} \p_t^4\bp\eta\|_{0.5}^2 \nonumber\\
 \lesssim& \PP_0+ \int_0^T \|\sqrt{\kk}f\|_{0.5}\|\sqrt{\kk}\curl_{\Ak}\p_t^4 v\|_{0.5}\dt\\
&+\int_0^T\|\sqrt{\kk}[\curl_{\Ak}, \bp]\bp\p_t^4 \eta\|_{0.5}\|\sqrt{\kk}\curl_{\Ak}\p_t^4 v\|_{0.5}\dt\nonumber\\
 &+ \int_0^T \|\sqrt{\kk}[\curl_{\Ak}, \bp]\p_t^4 v\|_{0.5}\|\sqrt{\kk}\curl_{\Ak} \bp \p_t^4 \eta\|_{0.5}\dt\nonumber\\
 &+\int_0^T\|\sqrt{\kk}\curl_{\p_t\Ak}\p_t^4\bp \eta\|_{0.5}\|\sqrt{\kk}\curl_{\Ak} \bp \p_t^4 \eta\|_{0.5}\dt.
\end{align}

Hence, by integrating in time one more time, we get
\begin{align}
&\frac{1}{2}\int_0^T \|\sqrt{\kk}\curl_{\Ak} \p_t^4 v\|_{0.5}^2+\frac{1}{2}\int_0^T\|\sqrt{\kk} \curl_{\Ak} \p_t^4\bp\eta\|_{0.5}^2 \nonumber\\
 \lesssim& \int_0^T \PP_0+ \int_0^T\int_0^t \|\sqrt{\kk}f\|_{0.5}\|\sqrt{\kk}\curl_{\Ak}\p_t^4 v\|_{0.5}\dt\nonumber\\
 &+\int_0^T\int_0^t \|\sqrt{\kk}[\curl_{\Ak}, \bp]\bp\p_t^4 \eta\|_{0.5}\|\sqrt{\kk}\curl_{\Ak}\p_t^4 v\|_{0.5}\dt\nonumber\\
 &+ \int_0^T\int_0^t \|\sqrt{\kk}[\curl_{\Ak}, \bp]\p_t^4 v\|_{0.5}\|\sqrt{\kk}\curl_{\Ak} \bp \p_t^4 \eta\|_{0.5}\dt\nonumber\\
&+\int_0^T\int_0^t \|\sqrt{\kk}\curl_{\p_t\Ak}\p_t^4\bp \eta\|_{0.5}\|\sqrt{\kk}\curl_{\Ak} \bp \p_t^4 \eta\|_{0.5}\dt,
 \label{time diff curl energy id}
\end{align}
where we have dropped one $\dt$ for the sake of concise notations. 
This suggests that we should control 
\begin{align*}
\int_0^T\int_0^t \|\sqrt{\kk} f\|_{0.5}^2\dt,\quad \int_0^T\int_0^t \|\sqrt{\kk}[\curl_{\Ak}, \bp]\bp\p_t^4 \eta\|_{0.5}^2\dt, \\
  \int_0^T\int_0^t \|\sqrt{\kk}[\curl_{\Ak}, \bp]\p_t^4 v\|_{0.5}^2\dt,\quad \int_0^T\int_0^t \|\sqrt{\kk}\curl_{\p_t\Ak}\p_t^4\bp \eta\|_{0.5}^2\dt.
\end{align*}

For the second term, we have
\begin{align*}
\int_0^T\int_0^t\|\sqrt{\kk}[\curl_{\Ak}, \bp]\bp\p_t^4 \eta\|_{0.5}^2\dt \lesssim \int_0^T\int_0^t \|\sqrt{\kk} \epsilon_{\alpha\beta\gamma}\Ak^{\mu\beta} (\p_\mu b_0^\nu)(\p_\nu\bp \p_t^4 \eta^\gamma\|_{0.5}^2\dt\\
+\int_0^T\int_0^t \|\sqrt{\kk} \epsilon_{\alpha\beta\gamma}(b_0^\nu\p_\nu\Ak^{\mu\beta} )(\p_\mu\bp \p_t^4 \eta^\gamma\|_{0.5}^2\dt, 
\end{align*}
which can be controlled by the RHS of \eqref{Gronwall''} by adapting \eqref{strategy}. 
The third and forth term are treated analogously. For $\int_0^T\int_0^t  \|\sqrt{\kk} f\|_{0.5}^2\dt$, invoking \eqref{time diff curl RHS}, we need to consider
\begin{align}
i=&\sum_{1\leq j\leq 3}\itt\|\sqrt{\kk}\p_t\big(  \epsilon_{\alpha\beta\gamma}(\p_t^j \Ak^{\mu\beta})\p_\mu \p_t^{4-j} v^\gamma\big)\|_{0.5}^2\dt,\\
 ii=&\sum_{1\leq j\leq 4}\itt \| \sqrt{\kk}\epsilon_{\alpha\beta\gamma} (\p_t^j \Ak^{\mu\beta})\p_\mu \bp^2 \p_t^{4-j} \eta^\gamma\|_{0.5}^2\dt. \label{ii}
\end{align}
Here, $i\LL \itt \|\sqrt{\kk}(\p_t \Ak)(\p\p_t^4 v)\|_{0.5}^2$, which controlled appropriately by adapting \eqref{strategy}.  Moreover, 
\begin{align}
ii\LL \itt \| \sqrt{\kk}\epsilon_{\alpha\beta\gamma} (\p_t \Ak^{\mu\beta})\p_\mu \bp^2 \p_t^{3} \eta^\gamma\|_{0.5}^2
\lesssim \int_0^T\int_0^t \|\sqrt{\kk} (\p_t \Ak)\p\p_t^3[\bp^2\eta]\|_{0.5}^2\dt\leq \int_0^T\PP,
\label{alternativ}
\end{align}
since $\int_0^T \|\sqrt{\kk} \p_t^3\bp \eta\|_{2.5}^2$ is included in $E_{\kk}^{(3)}$, and this concludes the control of $K_1+K_2$. 

\begin{rmk} There is an alternative way to control the last integral in \eqref{alternativ}. We may use the equation to replace $\bp^2 \eta$ by $\p_t v+\nab_{\Ak}q$, and this allow us to control this integral without using $\int_0^T \|\sqrt{\kk} \p_t^3\bp \eta\|_{2.5}^2$. In fact, one can show
$$
\itt\|\sqrt{\kk}\p_t^3 q\|_{2.5}^2\dt\leq \PP
$$
by employing the elliptic estimate we used in Section \ref{ellkk} (similar to the control of \eqref{II q}),
and so 
$$
\itt\|\sqrt{\kk}(\p_t\Ak) \p\p_t^3 [\bp^2 \eta]\|_{0.5}^2\dt\leq \itt \|\sqrt{\kk}(\p_t \Ak) \p \p_t^4 v\|^2_{0.5}+ \|\sqrt{\kk}(\p_t \Ak) \p\p_t^3 \nab_{\Ak} q\|_{0.5}^2\dt \leq\int_0^T\PP,
$$
because $\int_0^T \|\sqrt{\kk} \p_t^4 v\|_{1.5}^2$ is part of $E_{\kk}^{(3)}$.  
\end{rmk}

\subsection{Mixed space-time derivatives}\label{sect E3kk2}
The treatment for the remaining terms of $E_\kk^{(3)}$ is parallel and so we shall only sketch the details. We shall consider 
$$
\int_0^T \Big (\big\|\sqrt{\kk} \p_t^k v\big\|_{5.5-k}^2+\big\|\sqrt{\kk} \p_t^k \bp \eta\big\|_{5.5-k}^2\Big)\dt,\quad k=1,2,3.
$$
First, the boundary terms contributed by the time derivative(s) of $\bp \eta$, i.e., terms analogous to $K_{23}$, reads
$$
\int_0^T |\sqrt{\kk}\p_t^k \bp \eta^3|_{5-k}^2,\quad k=1,2,3.
$$
Generally speaking, for each fixed $k=1,2,3$, the control of the above term requires that of $\int_0^T |\sqrt{\kk} \p_t^{k-1} v^3|_{6-k}^2$, and this process stops when $k=1$. In particular, 
for each fixed $k=2,3$, we write $\int_0^T|\sqrt{\kk}\p_t^k\bp\eta^3|_{5-k}^2$ as $\int_0^T|\sqrt{\kk}\p_t^{k-1} \bp v^3|_{5-k}^2$, which can then be controlled together with $\int_0^T |\sqrt{\kk}\p_t^j v^3|_{5-j}^2$ with $j=1,2$. On the other hand, when $k=1$, the control of $\int_0^T|\sqrt{\kk} \p_t\bp \eta^3|_4^2$ requires that of 
$$
\int_0^T |\sqrt{\kk}\bp v^3|_4^2\lesssim P(\|b_0\|_{4.5})\int_0^T|\sqrt{\kk}v^3|_5^2,
$$
where, in view of \eqref{super v}, we have $\int_0^T|\sqrt{\kk}v|_5^2 \leq \MM_0+C(\eps)E_\kk(T)+\PP\int_0^T\PP$.

Second, the control of the analogous terms of $ii$ (defined in \eqref{ii}) for $k=1,2,3$ requires a similar analysis as above. For each fixed $k$, we need to investigate
\begin{align}
 ii'=\sum_{1\leq j\leq k}\itt \| \sqrt{\kk}\epsilon_{\alpha\beta\gamma} (\p_t^j \Ak^{\mu\beta})\p_\mu \bp^2 \p_t^{k-j} \eta^\gamma\|_{4.5-k}^2\dt. 
\end{align}
Again, it suffices to consider the most difficult term contributed by setting $j=1$, i.e., 
\begin{align}
 ii'=&\itt \| \sqrt{\kk}\epsilon_{\alpha\beta\gamma} (\p_t \Ak^{\mu\beta})\p_\mu \bp^2 \p_t^{k-1} \eta^\gamma\|_{4.5-k}^2\dt\label{II1}\\
  \lesssim &\itt P(\|v\|_{4.5}, \|b_0\|_{4.5}, \|\eta\|_{4.5})\|\sqrt{\kk} \p^3 \p_t^{k-1} \eta\|_{4.5-k}^2\dt.  
  \label{II2}
\end{align}
In \eqref{II2}, it can be seen that when $k=2,3$, $\itt \|\sqrt{\kk} \p^3 \p_t^{k-1} \eta\|_{4.5-k}\dt$ is bounded by $\itt\|\sqrt{\kk} \p^3 v\|_{2.5}^2\dt$ and $\itt\|\sqrt{\kk} \p^3 \p_t v\|_{1.5}^2\dt$, respectively. Moreover, when $k=1$, we need to consider \eqref{II1} instead. The strategy here is to replace $\bp^2 \eta$ by $\p_t v+\nab_{\Ak} q$, and so 
\begin{align}
ii'=\itt \| \sqrt{\kk}\epsilon_{\alpha\beta\gamma} (\p_t \Ak^{\mu\beta})\p_\mu \p_t v^\gamma\|_{3.5}^2\dt+\itt \| \sqrt{\kk}\epsilon_{\alpha\beta\gamma} (\p_t \Ak^{\mu\beta})\p_\mu \nab_{\Ak}^\gamma q\|_{3.5}^2\dt, 
\label{3133}
\end{align}
where the first term is bounded by the RHS of \eqref{Gronwall''} owing to \eqref{strategy}. For the second term, since $v\in H^{4.5}(\Omega)$, so it suffices to consider the case when all derivatives land on $\nab_{\Ak} q$, whose control requires that of 
\begin{align}\itt\|\sqrt{\kk} \nab_{\Ak} q\|_{4.5}^2\dt \label{II q}
\end{align}
after adapting \eqref{strategy}. Actually, we are able to prove a slightly stronger bound by removing one time integral, i.e., we want to bound $\int_0^T \|\sqrt{\kk}\nab_{\Ak} q\|_{4.5}^2$. By the div-curl estimate, one has 
\begin{align*}
\int_0^T\|\sqrt{\kk}\nab_{\Ak} q\|_{4.5}^2\lesssim \int_0^T \Big(\|\sqrt{\kk}\di \nab_{\Ak} q\|_{3.5}^2+\|\sqrt{\kk}\curl \nab_{\Ak} q\|_{3.5}^2+|\sqrt{\kk}N\cdot \nab_{\Ak}q\big|_3^2 + |\sqrt{\kk}q|_0^2\Big).
\end{align*}
Here, 
\begin{align}
\int_0^T \|\sqrt{\kk}\di \nab_{\Ak} q\|_{3.5}^2 \lesssim \int_0^T \|\sqrt{\kk} \Delta_{\Ak} q\|_{3.5}^2+\int_0^T \|\sqrt{\kk} \di_{\Ak-\delta} \nab_{\Ak} q\|_{3.5}^2,
\label{Neumann ell}
\end{align}
and by invoking \eqref{compars delta and a}, \eqref{ellq}, \eqref{strategy}, and since $v\in H^{4.5}(\Omega)$ a priori, we have 
\begin{align*}
\int_0^T \|\sqrt{\kk} \di_{\Ak-\delta} \nab_{\Ak} q\|_{3.5}^2 \lesssim \eps \int_0^T\|\sqrt{\kk}\nab_{\Ak} q\|_{4.5}^2 +\int_0^T\PP.
\end{align*}
Similarly, because $\curl_{\Ak} \nab_{\Ak} q=0$, we have
\begin{align*}
\int_0^T \|\sqrt{\kk}\curl \nab_{\Ak} q\|_{3.5}^2 \lesssim \eps\int_0^T\|\sqrt{\kk}\nab_{\Ak} q\|_{4.5}^2+\int_0^T\PP.
\end{align*}
Moreover, invoking \eqref{comars N and nn in H^3}, \eqref{strategy} and the trace lemma, then
\begin{align*}
\int_0^T |\sqrt{\kk}N\cdot \nab_{\Ak} q|_3^2 \lesssim& \int_0^T |\sqrt{\kk}\nnk\cdot \nab_{\Ak} q|_3^2+ \int_0^T |\sqrt{\kk}(N-\nnk)\cdot \nab_{\Ak} q|_3^2\\
\lesssim &\int_0^T |\sqrt{\kk}\nnk\cdot \nab_{\Ak} q|_3^2+ \eps \int_0^T\|\sqrt{\kk}\nab_{\Ak} q\|_{4.5}^2+\int_0^T\PP.
\end{align*}
As a consequence, \eqref{Neumann ell} becomes
\begin{align}
\int_0^T\|\sqrt{\kk}\nab_{\Ak} q\|_{4.5}^2 \lesssim \int_0^T \bigg(\|\sqrt{\kk} \Delta_{\Ak} q\|_{3.5}^2+|\sqrt{\kk}\nnk\cdot \nab_{\Ak} q\big|_3^2 + |\sqrt{\kk}q|_0^2\bigg).
\label{Neumann ell'}
\end{align}

To control the RHS, we recall that $q$ verifies 
\begin{equation}\label{higher order q elliptic int}
-\Delta_{\Ak} q=-\p_t\Ak^{\mu\alpha}\p_{\mu}v_{\alpha}+\p_{\beta}(\bp\ek_{\nu})\p_{\nu}\Ak^{\mu\nu}\Ak^{\beta\alpha}\p_{\mu}\bp\eta_{\alpha}
\end{equation}
with the Dirichlet and Neumann boundary conditions
\begin{align}
\sqrt{\tilde{g}}q =& -\sigma \sqrt{g}(\Delta_g \eta\cdot \nnk ) +\kk(1-\TL) (v\cdot \nnk),\\
\nnk \cdot \nab_{\Ak} q=&-\p_tv \cdot \nnk+\bp^2\eta\cdot \nnk. \label{higher order q elliptic bdy}
\end{align}
Now, 
\begin{align}
\int_0^T \|\sqrt{\kk} \Delta_{\Ak} q\|_{3.5}^2 \leq \int_0^T \kk\|\p_t \Ak \p v\|_{3.5}^2+ \int_0^T \kk\|\p (\bp\eta) (\p_t a) (a\p \bp\eta)\|_{3.5}^2,
\end{align}
and because $v,\bp\eta \in H^{4.5}(\Omega)$ a priori, the RHS is bounded by $\int_0^T \PP$. Also, it is not hard to see, via the trace lemma and the Dirichlet boundary condition,  that 
$$
\int_0^T |\sqrt{\kk} q|_0^2 \leq \int_0^T\PP. 
$$

Next, we control $\int_0^T|\sqrt{\kk}\nnk\cdot\nab_{\Ak}q|_3^2$. In view of the Neumann boundary condition \eqref{higher order q elliptic bdy}, it contributes to 
$$
\int_0^T \kk |\p_t v \cdot \nnk|_3^2,\quad \int_0^T |\sqrt{\kk}\bp^2 \eta\cdot \nnk|_{3}^2.
$$
 For the first term, since $\p_t v\in H^{3.5}(\Omega)$ and $\eta \in H^{4.5}(\Omega)$ a priori, as well as $\cp \nnk=Q(\cp \eta) \cp^2\eta$, we have, after employing the trace theorem, that
\begin{align*}
\int_0^T \kk|\p_t v\cdot \nnk|_3^2 \leq \int_0^T\PP.
\end{align*}
 Also, for the second term, 
 $$
 \int_0^T |\sqrt{\kk}\bp^2 \eta\cdot \nnk|_{3}^2\LL \int_0^T |\sqrt{\kk}\bp^2 \p^3 \eta\cdot \nnk|_0^2 \leq \int_0^T P(\|b_0\|_{4.5}, ||\eta||_{4.5}) |\sqrt{\kk} \TP^5\eta|_0^2,
 $$ 
 which can be controlled by $\MM_0+C(\eps)E_\kk(T)+\PP\int_0^T\PP$ owing to \eqref{super eta}.   
\subsection{Full spatial derivatives}\label{E3kkk}
It remains for us to bound
\begin{equation*}
\int_0^T \Big (\big\|\sqrt{\kk}  v\big\|_{5.5}^2+\big\|\sqrt{\kk}  \bp \eta\big\|_{5.5}^2\Big)\dt
\end{equation*}
in order to conclude the control of $E_\kk^{(3)}(T)$. This will be treated via the div-curl decomposition, and
since there is no time derivative, the arguments are similar to those in Section 3.2.
First, we assume that the quantity
\begin{equation}\label{supereta}
\|\sqrt{\kk}\eta\|_{5.5}
\end{equation}
is bounded a priori. 
Second, for the divergence part of velocity, the perturbation argument yields
  \begin{align}\label{perturb 5.5}
  \int_0^T\|\sqrt{\kk}\di v\|_{4.5}^2 =& \int_0^T \|\sqrt{\kk}(\ak^{\mu\alpha}-\delta^{\mu\alpha})\p_\mu v_\alpha\|_{4.5}^2\nonumber\\
 &\leq \int_0^T \|\sqrt{\kk}(\int_0^t \p_t \ak^{\mu\alpha}\,ds)\p_\mu v_\alpha\|_{4.5}^2\leq P\left(\|\sqrt{\kk} \eta\|_{5.5}, \|\sqrt{\kk}v\|_{L_t^2H_y^{5.5}}\right)\int_0^T\PP.
  \end{align}
 Also, for the divergence part of magnetic field,  it suffices to study $\int_0^T \|\sqrt{\kk} \di_{\ak} \bp\eta\|_{4.5}^2$ by invoking a perturbation argument similar to \eqref{perturb 5.5}. We differentiate \eqref{divbeq} with $\sqrt{\kk}\p^{4.5}$ and then integrate in time to get
\[
\sqrt{\kk}\p^{4.5}(\diva\bp\eta)=\int_0^T\sqrt{\kk}\p^{4.5}(\p_{\beta}\bp\eta_{\gamma}\ak^{\mu\gamma}\ak^{\beta\alpha}\p_{\mu}v_{\alpha}-\ak^{\mu\gamma}\p_\beta\tilde{v}_{\gamma}\ak^{\beta\alpha}\p_{\mu}\bp\eta_{\alpha}),
\] and taking $L_t^2L_y^2$ norm yields that
\begin{equation}\label{superdiv}
\|\sqrt{\kk}\p^{4.5}(\di_{\ak}\bp\eta)\|_{L_t^2L_y^2}^2\lesssim T\left(\|\sqrt{\kk} v\|_{L_t^2H_y^{5.5}}^2+\|\sqrt{\kk} \bp\eta\|_{L_t^2H_y^{5.5}}^2+\|\sqrt{\kk} \eta\|_{L_t^2H_y^{5.5}}^2\right)\int_0^T \PP.
\end{equation}
  Third, for the curl part of both velocity and magnetic field, it suffices to study $\int_0^T \|\sqrt{\kk}\curl_{\Ak} v\|_{4.5}^2+ \|\sqrt{\kk} \curlA\bp\eta\|_{4.5}^2$ by differentiating \eqref{curlbeq} with $\p^{4.5}$ and test it with $\kk\p^{4.5}\curl_{\Ak}v$ in $L_t^2L_y^2$. We can show, parallel to \eqref{curlvb3}, that
\begin{equation}\label{supercurl}
\begin{aligned}
&\frac12\frac{d}{dt}\left(\|\sqrt{\kk} \p^{4.5}\curlA v\|_{L_t^2L_y^2}^2+\|\sqrt{\kk} \p^{4.5}\curlA \bp\eta\|_{L_t^2L_y^2}^2\right)\\
\lesssim &~P(\|\sqrt{\kk} b_0\|_{5.5},\|\sqrt{\kk}\bp\eta\|_{L_t^2H_y^{5.5}},\|\sqrt{\kk}v\|_{L_t^2H_y^{5.5}},\|\sqrt{\kk}\eta\|_{5.5}), 
\end{aligned}
\end{equation}
Since $\|\sqrt{\kk}\eta\|_{5.5} \leq \|\sqrt{\kk}\eta_0\|_{5.5}+\int_0^T\|\sqrt{\kk} v\|_{5.5}$, by the above estimates together with \eqref{normal trace v} and \eqref{normal trace bp eta} obtained in the discussion of Lemma \ref{super control}, we get 
\begin{equation}\label{superdivcurl}
\|\sqrt{\kk}v\|_{L_t^2H_y^{5.5}}^2+\|\sqrt{\kk}\bp\eta\|_{L_t^2H_y^{5.5}}^2+\|\sqrt{\kk}\eta\|_{5.5}^2\lesssim\mathcal{M}_0+C(\eps)E_\kk(T)+\PP\int_0^T\PP.
\end{equation}

In summary, we have
 \begin{align}
 E_{\kk}^{(3)} \leq \MM_0+C(\eps)E_\kk(T)+\PP\int_0^T\PP. 
 \label{estimate Ekk^(3)}
 \end{align}
 
 \begin{rmk}
 The estimate \eqref{superdivcurl}, together with the trace lemma concludes the proof of Lemma \ref{super control}.
 \end{rmk}

 \section{Closing the nonlinear energy estimate}\label{close}
 In this section we conclude the proof of Proposition \ref{nonlinearkk}. 
 \subsection{Regularity of initial data}\label{datakk}
Our first task is to remove the extra regularity assumptions on the initial data. These additional regularities are introduced in $\MM_0$ (defined in Lemma \ref{super control}). In addition to this, one has to control $\|q(0)\|_{4.5}, \|q_t(0)\|_{3.5}, \|q_{tt}(0)\|_{2.5}$ in terms of $v_0$ and $b_0$ by the elliptic estimate, and extra regularity on $v_0$ and $b_0$ shall appear due to the viscosity term.

Note that $q_0$ verifies the elliptic equation
\begin{equation}
\begin{cases}
-\Delta q_0=(\p v_0)(\p v_0)-(\p b_0)(\p b_0)~~~&\text{ in }\Omega\\
q_0=\kk(1-\TL)v^3~~~&\text{ on }\Gamma\\
\frac{\p q_0}{\p N}=0~~~&\text{ on }\Gamma_0\\
\end{cases}
\end{equation}
by standard elliptic estimates, we get
\[
\|q_0\|_{4.5}\lesssim\|\p v_0\|_{2.5}^2+\|\p b_0\|_{2.5}^2+\kk\|v_0^3\|_{4.5}+\kk|v_0^3|_6.
\]

Moreover, note that the energy functional contains time derivatives of $v$ and $\bp\eta$, so we need to express their initial data in terms of $v_0$ and $b_0$ as well. We invoke $\p_t v(0)-\bp b_0=-\p q_0$ to get
\[
\|\p_t v(0)\|_{3.5}\lesssim\|b_0\|_{3.5}\|b_0\|_{4.5}+\|q_0\|_{4.5},
\]and 
\[
\|\p_t \bp\eta(0)\|_{3.5}\lesssim\|b_0\|_{3.5}\|v_0\|_{4.5}.
\]
Similarly, we consider the $\p_t$-differentiated elliptic equation of $q$ to get
\[
\|\p_t q(0)\|_{3.5}\lesssim P(\|v_0\|_{4.5},\|b_0\|_{4.5})(|v_0^3|_5+\kk|\p_t v(0)|_5),
\]and further
\[
\|\p_t^2 q(0)\|_{2.5}+\|\p_t^3 q(0)\|_1\lesssim P(\|v_0\|_{4.5},\|b_0\|_{4.5},|v_0|_5)(1+\kk|\TL\p_t^2 v(0)|_2).
\]
By Sobolev trace lemma, we need to bound $\kk\|\p_t^2v(0)\|_{4.5}$ which requires the control of $\kk(\|v_0\|_{6.5}+\|b_0\|_{5.5}+\|\p_tq(0)\|_{5.5})$. We replace 3.5 by 5.5 in the estimates of $\p_t q(0)$, and thus we need to control $$\kk^2|\p_tv(0)|_7\lesssim\kk^2(\|b_0\|_{7.5}\|b_0\|_{8.5}+\|q_0\|_{8.5}).$$ Finally, replacing 4.5 by 8.5 in the estimates of $q_0$, we need to control $$\kk^2(\|v_0\|_{7.5}^2+\|b_0\|_{7.5}^2)+\kk^3(|v_0|_{8}+|v_0|_{10}).$$

In view of the above analysis and the definition of $\MM_0$, we need to control $\kk$-weighted norms of $\|v_0\|_{8.5},\|b_0\|_{8.5}$ and $|v_0|_{10}$. However, our given initial data is $v_0\in H^{4.5}(\Omega)\cap H^5(\Gamma)$ and $b_0\in H^{4.5}$ and so we have to remove the additional regularity assumptions on the initial data. This can be done by adapting a similar argument in Section 12 of Coutand-Shkoller \cite{coutand2007LWP}. We define $\Omega_{\kk}$ to be the regularized version of $\Omega$ tangentially mollified by $\zeta_{\exp^{-\kk}}$ and define $E_{\Omega_\kk}$ to be the extension operator from $\Omega$ to $\Omega_\kk$. Next we set $$\vvv:=\zeta_{\exp^{-\kk}}*E_{\Omega_\kk}(v_0),~\bbb:=\zeta_{\exp^{-\kk}}*E_{\Omega_\kk}(b_0),~\qqq:=\zeta_{\exp^{-\kk}}*E_{\Omega_\kk}(q_0).$$ Therefore, integrating by parts repeatedly to transfer derivatives to the mollifier $\zeta_{\exp^{-\kk}}$, we get
\begin{equation}
\|\kk \vvv\|_{8.5}+\|\kk\bbb\|_{8.5}+\|\kk\qqq\|_{8.5}+|\kk\vvv|_{10}\lesssim \|v_0\|_{4.5}+\|b_0\|_{4.5}+\|q_0\|_{4.5}+|v_0|_5\leq \mathcal{C}, 
\label{removing extra regularity}
\end{equation}
where $\mathcal{C}$ is the constant that appears in \eqref{energykk}.

\subsection{Nonlinear a priori estimates}\label{nkk}

Now we summarize the a priori estimates of the nonlinear $\kk$-approximation system \eqref{MHDLkk}. 
\begin{enumerate}
\item  \eqref{ellq} gives the elliptic estimates of $q$ and its time derivatives. 
\item \eqref{divv3}-\eqref{divvtt1} and \eqref{divb3}, \eqref{divbt} give the divergence estimate and \eqref{curlvb3}-\eqref{curlvbt} give the curl estimate.
\item \eqref{boundary v} and \eqref{b3bdry} control the boundary part of $v,\bp\eta$ and its time derivative. 
\item \eqref{tgt4}, \eqref{tgt3}, \eqref{tgt5}-\eqref{tgt7} provide control of the mixed tangential derivatives of $v$ and $\bp\eta$ and the Eulerian normal projections of $v$. Note that these estimate depends on $E_{\kk}^{(3)}$ on the RHS. 
\item Finally, \eqref{estimate Ekk^(3)} provides the estimate for $E_{\kk}^{(3)}$.
\end{enumerate}
Thus, by combining these estimates and then invoking \eqref{removing extra regularity},  we obtain a Gronwall-type inequality:
\begin{equation}
E_\kk(T)-E_\kk(0)\lesssim C(\eps)E_{\kk}(T)+\mathcal{C}(\|v_0\|_{4.5}, \|b_0\|_{4.5})+P(E_\kk(T))\int_0^TE_\kk(t)\dt.
\end{equation} We pick $\eps>0$ suitably small such that the $\eps$-terms can be absorbed to LHS. Therefore, by the nonlinear Gronwall inequality in Chapter 2 of Tao \cite{tao2006nonlinear}, we know there exists some time $T>0$ independent of $\kk$, such that
\begin{equation}
\sup_{0\leq t\leq T}E_{\kk}(t)\leq \mathcal{C}.
\end{equation}
This concludes the proof for Proposition \ref{nonlinearkk}.


\section{Existence and uniqueness for the linearized approximate system}\label{linearlwpkk}
Since we have obtained a uniform-in-$\kk$ a priori energy estimate for the approximate $\kk$-problem \eqref{MHDLkk}, our next goal is to construct a solution for this system for each fixed $\kk>0$ that is sufficiently small. We shall assume that $0<\kk\ll1$ \textit{is fixed} throughout the rest of this manuscript. 

Let $T>0$. We define
\begin{align}
\XX = \{ \mathbf{u} \in L^\infty(0,T;H^{4.5}{(\Omega)}): \sup_{[0,T]} \|\mathbf{u}\|_{4.5} \leq 2(\|v_0\|_{4.5}+\|b_0\|_{4.5})+1\},
\end{align}
which is a closed subset of the space $L^\infty(0,T;H^{4.5}(\Omega))$.  In order to solve the approximate $\kk$-problem \eqref{MHDLkk} for each fixed $\kk>0$, we study the following linearized problem whose fixed-point shall provide the desired solutions. Fix an arbitrary function $\er=\er(t,y)\in \XX$ whose time derivative $\er_t\in \XX$, we denote by  $\ar$, $\gr$, $\Jr$ and $\Ar$ the associated quantities in Lagrangian coordinates. Also, we define $\erk$ by
\begin{equation*}
\begin{cases}
-\Delta\erk=-\Delta\er,~~~&\text{in}\,\,\Omega,\\
\erk=\lkk^2\er~~~&\text{ on }\p\Omega,
\end{cases}
\end{equation*}
and let $\ark:=[\p\erk]^{-1},~\Jrk:=\det[\p\erk],~\Ark:=\Jrk\ark$ and $\nnr$ to be the associated smoothed quantities.

We aim to construct $\eta$ and $v$ that solve
\begin{equation}
\begin{cases}\label{linearized}
\p_t\eta=v~~~& \text{in}~[0,T]\times\Omega;\\
\p_tv-\bp^2\eta+\nab_{\Ark} q=0~~~& \text{in}~[0,T]\times\Omega;\\
\di_{\Ark} v=0,\quad \dive b_0=0 &\text{in}~[0,T]\times\Omega;\\
v^3=b_0^3=0~~~&\text{on}~\Gamma_0;\\
\Ark^{3\alpha}q = -\sigma \sqrt{\gr}(\Delta_{\gr} \er\cdot \nnr )\nnr^\alpha +\kk(1-\TL) (v\cdot \nnr)\nnr^\alpha~~~&\text{on}~\Gamma;\\
(\eta,v)=(Id,v_0)~~~&\text{on}~\{t=0\}{\times}\overline{\Omega}.
\end{cases}
\end{equation}
The rest of this section is devoted to showing the existence of $(\eta, v)$ by first establishing the existence of the weak solution and then boosting up their regularity. The construction of the solution for the nonlinear $\kk$-problem will be postponed until the next section.

We will adapt the method developed in Coutand-Shkoller \cite{coutand2007LWP} to construct the weak solution for \eqref{linearized}. In particular, we study the penalized $\la$-problem \eqref{penalized} whose solution $(\xi_\la, w_\la, q_\la)$ can be obtained by the Galerkin's approximation. Also, we show that this solution converges to that of \eqref{linearized} when $\la\to 0$. 
We mention here that in \cite{coutand2007LWP}, the authors were able to prove $w_\la \in L^2(0, T; H^1(\Omega))$, which allowed them to obtain a strong solution of the incompressible Euler equations after taking the limit. However, the coupling between the velocity and magnetic field prevents us from employing this technique, as we cannot boost the regularity of $w_\la$ alone without considering $\bp\xi_\la$. 
Because of this, it appears that we have to first construct the weak solution of \eqref{linearized} in $L^2(0,T; H^{-1}(\Omega))$ (Section \ref{section la problem}), and then prove that this solution in fact has $L^2(0,T;H^1(\Omega))$ regularity by employing a bootstrap argument (Subsection \ref{sect boundary term}). 

\begin{rmk}
There is another reason that prevents us from improving the regularity of $(w_\la, \bp\xi_\la)$. Unlike $w_\la$, the energy estimate \eqref{ODE L^2 energy eps} fails to give any control of the normal component of $\bp\xi_\la$ on the boundary. This is, however, required for estimating $\bp\xi_\la$ in $H^1(\Omega)$. 
\end{rmk}

\subsection{The penalized problem}\label{section la problem}
The goal of this subsection is to study the penalized version (of the divergence-free condition on the velocity) of the linearized $\kk$-problem \eqref{linearized}. In particular, for $0<\la \ll 1$, let $w_\la, \xi_{\la}$ be the solutions for \eqref{linearized} with
\begin{align}
\di_{\Ark} w_\la = -\la q_{\la}, 
\end{align} 
where $q_\la$ is defined to be the penalized pressure. In this case, \eqref{linearized} becomes
\begin{equation}
\begin{cases}\label{penalized}
\p_t\xi_\la=w_\la~~~& \text{in}~[0,T]\times\Omega;\\
\p_tw_\la-\bp^2\xi_\la+\nab_{\Ark} q_\la=0~~~& \text{in}~[0,T]\times\Omega;\\
\di_{\Ark}w_\la=-\la q_{\la},\quad \dive b_0=0 &\text{in}~[0,T]\times\Omega;\\
w_\la^3=b_0^3=0~~~&\text{on}~\Gamma_0;\\
\Ark^{3\alpha}q_\la = -\sigma \sqrt{\gr}(\Delta_{\gr} \er\cdot \nnr )\nnr^\alpha +\kk(1-\TL) (w_\lambda\cdot \nnr)\nnr^\alpha~~~&\text{on}~\Gamma;\\
(\xi_\la,w_\la)=(Id,v_0)~~~&\text{on}~\{t=0\}{\times}\overline{\Omega}.
\end{cases}
\end{equation}
Since each penalized problem is indexed by $\la$ (recall $\kk$ is fixed), we shall denote them by ``$\la$-problem" throughout the rest of this section. 

\subsubsection{Weak solution for the $\la$-problem via Galerkin's approximation}
First of all, for each fixed $\la$, we will solve the $\la$-problem by the Galerkin approximation and obtain a weak solution. By introducing a basis $(e_k)_{k=1}^{\infty}$ of $L^2(\Omega)\cap H^1(\Omega)$, and considering the approximation 
\begin{align}
\xi_m(t,y) =& \sum_{j=1}^m Z_j(t)e_j(y),\quad m\geq 2, \label{p_t xi}\\
w_m (t,y) =& \p_t \xi_m(t,y) =\sum_{k=1}^m Z_j' (t) e_j(y),
\end{align}
one can form a system of ODE by multiplying a test vector field $\phi\in\text{span}(e_1,\cdots,e_m)$ to the $\la$-problem. Specifically, we have
\begin{align}
\io (w_m^\alpha)_t \phi_\alpha -\io [\bp^2 \xi_m^\alpha] \phi_\alpha + \io [\Ark^{\mu\alpha} \p_\mu q_m] \phi_\alpha = 0.
\end{align}

We recall that $\bp|_{\Gamma}$ is tangential to $\Gamma$. Owing to this and the boundary condition of $q_m$, i.e.,
\begin{equation}
\Ark^{3\alpha}q_m = -\sigma \sqrt{\gr}\left(\Delta_{\gr} \er_m\cdot \nnr \right)\nnr^\alpha +\kk(1-\TL) (w_m\cdot \nnr)\nnr^\alpha,\quad \text{on}\,\,\, \Gamma,
\end{equation}
where $\er_m$ denotes the projection of $\er$ onto $\text{span}(e_1,\cdots,e_m)$, 
we obtain, after integration by parts, that
\begin{align}
&\io (w_m^\alpha)_t \phi_\alpha +\io [\bp\xi_m^\alpha] [\bp \phi_\alpha] +\kk\sum_{l=0,1}\ig \TP^l (w_m\cdot\nnr)\TP^l(\phi\cdot \nnr)\nonumber\\
 & -\io  q_m[\Ark^{\mu\alpha} \p_\mu \phi_\alpha] = \sigma \ig (\sqrt{\gr}\Delta_{\gr} \er_m \cdot \nnr)  (\phi\cdot \nnr).\label{ODE p}
\end{align}
Also, invoking the identity holds for the penalized pressure
\begin{align}
\di_{\Ark} w_m = -\la q_m, \label{penalized q}
\end{align}
we obtain 
\begin{align}\label{ODE pre}
&\io (w_m^\alpha)_t \phi_\alpha +\io [\bp\xi_m^\alpha] [\bp \phi_\alpha] +\kk\sum_{l=0,1}\ig \TP^l (w_m\cdot\nnr)\TP^l(\phi\cdot \nnr)\nonumber\\
 & +\frac{1}{\la}\io  (\divAr w_m)(\divAr \phi) = \sigma \ig (\sqrt{\gr}\Delta_{\gr} \er_m \cdot \nnr)  (\phi\cdot \nnr).
\end{align}

Let $\phi = e_k$, for each fixed $k=1,2,\cdots, m$.  Then \eqref{ODE pre} yields a system of ODE:
\begin{align}
&Z_k''(t)+\sum_{j=1}^m\left( \int_\Omega [\bp e_j^\alpha][\bp (e_k)_\alpha]\right) Z_j(t)+\kk\sum_{l=0,1}\sum_{j=1}^m\left(\int_\Gamma \TP^l (e_j\cdot \nnr)\TP^l (e_k\cdot \nnr)\right) Z_j'(t)\nonumber\\
&+\frac{1}{\la}\sum_{j=1}^m\left(\int_\Omega (\divAr e_j)(\divAr e_k)\right)Z_j'(t)=\sigma \ig (\sqrt{\gr}\Delta_{\gr} \er_m \cdot \nnr)  (e_k\cdot \nnr), \label{ODE}
\end{align}
equipped with initial data
\begin{align}
Z_k(0)=\langle Id, e_k\rangle, \quad Z_k'(0) = \langle v_0, e_k\rangle.\label{ODE data}
\end{align}
The standard ODE theory gives the the existence and uniqueness of $\xi_m$ and $w_m$ in $[0,T_\la]$ for some $T_\la>0$.  
We remark here that it is important to introduce the penalized pressure \eqref{penalized q} because we are not able to formulate a system of ODE directly from \eqref{ODE p}. This is because $\phi$ is not divergence-free in general. 

Setting $\phi = w_m$ in \eqref{ODE pre}, and since $\sigma|\sqrt{\gr}\Delta_{\gr} \er_m\cdot \nnk|_0\leq \Nf$, where $\Nf$ denotes a generic polynomial function such that $$\Nf=P(\|v_0\|_{4.5}, \|b_0\|_{4.5}),$$  then
\begin{align}
\|w_m\|_0^2 + \|\bp \xi_m\|_0^2 +\frac{1}{\la} \int_0^{T_\la} \|\divAr w_m\|_0^2+\kk \int_0^{T_\la} |w_m\cdot \nnr|_1^2 \leq  \mathcal{N}_0.
\label{ODE L^2 energy}
\end{align}
Because $\mathcal{N}_0$ is independent of $m$, 
there is a subsequence, which is still indexed by $m$, satisfying
\begin{align}
w_m\rightharpoonup  w_\la, \quad \bp \xi_m\rightharpoonup \bp \xi_\la,\quad &\text{in}\,\,L^{\infty}(0,T_\la; L^2(\Omega)),\\
\divAr w_m \rightharpoonup \divAr w_\la,\quad &\text{in}\,\,L^2(0,T_\la; L^2(\Omega)), \label{div w_m cov}\\
w_m \cdot \nnr\rightharpoonup w_\la\cdot \nnr,\quad &\text{in}\,\, L^2(0,T_\la; H^1(\Gamma)), 
\end{align}
as $m\to \infty$, 
and the following estimate holds:
\begin{align}
\|w_\la\|_0^2 + \|\bp \xi_\la\|_0^2 +\frac{1}{\la} \int_0^{T_\la} \|\divAr w_\la\|_0^2+\kk \int_0^{T_\la} |w_\la \cdot \nnr|_1^2 \leq  \mathcal{N}_0. 
\label{ODE L^2 energy eps}
\end{align}
In addition, since $q_m = -\frac{1}{\la} \divAr w_m$ and $q_\la = -\frac{1}{\la} \divAr w_\la$, \eqref{div w_m cov} implies that
\begin{equation}
q_m \rightharpoonup q_\la,\quad \text{in}\,\,L^2(0,T_\la; L^2(\Omega)).
\end{equation}

Let $Y$ be a Banach space, and we denote its dual by $Y'$. For $\Psi\in H^s(\Omega)'$ and $\Phi \in H^s(\Omega)$, the pairing between $\Psi$ and $\Phi$ is denoted by $\langle \Psi, \Phi\rangle_s$.  It follows from the ODE defining $w_m$, that $\p_t w_\la\in L^2(0,T_\la; H^{\frac{3}{2}+}(\Omega)')$,  where $H^{s+} := H^{s+\delta}$ for some $0<\delta\ll 1$. 
Let $\varphi\in L^2(0,T_\la; H^{\frac{3}{2}+}(\Omega))$ with $\int_0^{T_\la} \|\varphi\|_{\frac{3}{2}+}^2= 1$ be a test vector field. Then there holds the identity:
\begin{equation}\label{variational eq}
\begin{aligned}
\int_0^{T_\la} \langle \p_t w_\la, \varphi\rangle_{\frac{3}{2}+}
- \int_0^{T_\la} \langle q_\la, \divAr \varphi\rangle_{\frac{1}{2}+} 
&= -\int_0^{T_\la} \langle \bp^2\xi_\la, \varphi\rangle_{\frac{3}{2}+}\\
&+\sigma \int_0^{T_\la} \ig (\sqrt{\gr}\lap_{\gr} \er\cdot \nnr) (\varphi\cdot\nnr)-\kk \sum_{l=0,1} \int_0^{T_\la} \ig \cp^l (w_\la\cdot \nnr)\cp^l (\varphi\cdot \nnr).
\end{aligned}
\end{equation}
Thanks to \eqref{ODE L^2 energy eps}, we have
\begin{align}
-\int_0^{T_\la} \langle \bp^2\xi_\la, \varphi\rangle_{\frac{3}{2}+} = \int_0^{T_\la} \langle \bp\xi_\la, \bp\varphi\rangle_{\frac{1}{2}+}\leq \Nf.
\end{align}
Also, invoking $\sigma|\sqrt{\gr}\Delta_{\gr} \er\cdot \nnk|_0\leq \Nf$ and $\kk\int_0^{T_\la} |w_\la \cdot \nnr|_1^2\leq \Nf$ (inferred from \eqref{ODE L^2 energy eps}), we obtain
\begin{equation}
\text{The second line of the RHS of \eqref{variational eq}} \leq \sigma \Nf \int_0^{T_\la} \|\varphi\|_{\frac{1}{2}+}^2+\kk \Nf \int_0^{T_\la} \|\varphi\|_{\frac{3}{2}+}^2\leq (\sigma+\kk)\Nf \leq \Nf.
\end{equation}
and so 
\begin{equation}\label{RHS bound}
\text{RHS of \eqref{variational eq}} \leq \Nf. 
\end{equation}
Since $\p_t w_\la\in L^2(0,T_\la; H^{\frac{3}{2}+}(\Omega)')$, \eqref{variational eq} and \eqref{RHS bound} indicate
\begin{equation}
q_\la \in L^2(0,T_{\la}; H^{\frac{1}{2}+}(\Omega)'),\quad \text{and}\,\,\nab_{\Ark} q_\la \in L^2(0,T_{\la}; H^{-1}(\Omega)). \label{q_la space}
\end{equation}
Therefore, we have that
\begin{align}
\p_t w_{\la} - \bp^2 \xi_\la + \nab_{\Ark} q_\la = 0, \quad\text{in}\,\,\, L^2(0, T_\la; H^{-1}(\Omega)). 
\label{equation epsilon}
\end{align}


\subsubsection{The limit as $\la\rightarrow 0$}
Because $\mathcal{N}_0$ is independent of $\lambda$, the energy estimate \eqref{ODE L^2 energy eps} indicates that 
\begin{equation*}
\|w_\la\|_0^2 + \|\bp \xi_\la\|_0^2 \leq \Nf.
\end{equation*}
Thus, there exists a \textit{$\la$-independent} $T>0$, such that the sequences $\{w_\la\}$ and $\{\bp \xi_\la\}$ admit converging subsequences (still indexed by $\la$) as $\la \to 0$, i.e.,
\begin{equation}
w_\la \rightharpoonup v,\quad \bp \xi_\la \rightharpoonup \bp\eta,\quad  \text{in}\,\,L^{\infty}(0,T; L^2(\Omega)).\label{int limit}
\end{equation}
Furthermore, we must have
\begin{align*}
\frac{1}{\la} \int_0^{T} \|\divAr w_\la\|_0^2+\kk \int_0^{T} |w_\la \cdot \nnr|_1^2 \leq  \mathcal{N}_0,
\end{align*}
and thus
\begin{align}
\divAr w_\la \rightharpoonup \divAr v=0,\quad &\text{in}\,\, L^2(0,T; L^2(\Omega)),\\
w_\la\cdot\nnr \rightharpoonup v\cdot\nnr\quad  &\text{in}\,\,L^2(0,T; H^1(\Gamma)). \label{bdy limit}
\end{align}
Moreover, the following energy estimate holds:
\begin{align}
\|v\|_0^2 + \|\bp \eta\|_0^2 +\kk \int_0^T |v \cdot \nnr|_1^2 \leq  \mathcal{N}_0. \label{L^2 estimate lambda independent}
\end{align}

Our next goal is to show that $(\eta, v)$ is indeed a weak solution for \eqref{linearized}. To achieve this, we need to show $\p_t w_\la$ converges to $\p_t v$, and $\bp^2\xi_\la$ converges to $\bp^2\eta$ in $L^2(0, T; H^{\frac{3}{2}+} (\Omega)')$. In addition, we need to find $q$, in terms of the pressure function, that belongs to $L^2(0,T; H^{\frac{1}{2}+}(\Omega)')$.
First, analogous to \eqref{variational eq}, we have
\begin{equation}\label{variational eq T}
\begin{aligned}
&\int_0^{T} \langle \p_t w_\la, \varphi\rangle_{\frac{3}{2}+}
+ \int_0^{T} \frac{1}{\la}\langle \divAr w_\la, \divAr \varphi\rangle_{\frac{1}{2}+}\\
&= -\int_0^{T} \langle \bp^2\xi_\la, \varphi\rangle_{\frac{3}{2}+}\\
&+\sigma \int_0^{T} \ig (\sqrt{\gr}\lap_{\gr} \er\cdot \nnr) (\varphi\cdot\nnr)-\kk \sum_{l=0,1} \int_0^{T} \ig \cp^l (w_\la\cdot \nnr)\cp^l (\varphi\cdot \nnr).
\end{aligned}
\end{equation}
Second, let 
\begin{equation}
q:= \lim_{\la\to 0} q_\la = -\lim_{\la\to 0} \frac{1}{\la}\divAr w_\la, \label{q def}
\end{equation}
where the limit $$-\lim_{\la\to 0} \frac{1}{\la}\divAr w_\la$$ exists in $L^2(0,T; H^{\frac{1}{2}+}(\Omega)')$ owing to  $\p_t w_\la \in L^2(0,T;H^{\frac{3}{2}+}(\Omega)')$, and the RHS of \eqref{variational eq T} $\leq  \Nf$, which is independent of $\la$. We then have $q \in L^2(0,T; H^{\frac{1}{2}+}(\Omega)')$ and $\nab_{\Ark} q \in L^2(0,T; H^{-1}(\Omega))$, analogous to \eqref{q_la space}. Invoking \eqref{int limit}, this yields
\begin{equation}
\p_t w_\la \rightharpoonup \p_t v,\quad \text{in}\,\,L^2(0,T;H^{\frac{3}{2}+}(\Omega)'),
\end{equation}
as well as
\begin{equation}
\lim_{\la\to 0} \int_0^T \|\p_t w_{\la}\|_{H^{\frac{3}{2}+}(\Omega)'}^2\leq \NN_0.
\label{v_t -1/2+} 
\end{equation}
On the other hand, by employing similar arguments on
\begin{equation}\label{variational eq T'}
\begin{aligned}
\int_0^{T} \langle \bp^2\xi_\la, \varphi\rangle_{\frac{3}{2}+}=-\int_0^{T} \langle \p_t w_\la, \varphi\rangle_{\frac{3}{2}+}
 &-\int_0^{T} \frac{1}{\la}\langle \divAr w_\la, \divAr \varphi\rangle_{\frac{1}{2}+}\\
&+\sigma \int_0^{T} \ig (\sqrt{\gr}\lap_{\gr} \er\cdot \nnr) (\varphi\cdot\nnr)-\kk \sum_{l=0,1} \int_0^{T} \ig \cp^l (w_\la\cdot \nnr)\cp^l (\varphi\cdot \nnr), 
\end{aligned}
\end{equation}
we have
\begin{equation}
\bp^2\xi_\la\rightharpoonup \bp^2 \eta, \quad \text{in}\,\,L^2(0,T;H^{\frac{3}{2}+}(\Omega)'),
\end{equation}
and
\begin{equation}\label{7.36}
\lim_{\la\to 0} \int_0^T \|\bp^2 \xi_{\la}\|_{H^{\frac{3}{2}+}(\Omega)'}^2\leq \NN_0.
\end{equation}

Lastly, we infer from \eqref{variational eq T}, and the convergence of $\p_t w_\la$, $\bp^2 \xi_\la$, and $q_\la$, that
 \begin{equation}\label{var}
\begin{aligned}
\int_0^{T} \langle \p_t v, \varphi\rangle_{\frac{3}{2}+}+\int_0^{T} \langle \bp^2\eta, \varphi\rangle_{\frac{3}{2}+}
- \int_0^{T} \langle q, \divAr \varphi\rangle_{\frac{1}{2}+}\\
= \sigma \int_0^{T} \ig (\sqrt{\gr}\lap_{\gr} \er\cdot \nnr) (\varphi\cdot\nnr)-\kk \sum_{l=0,1} \int_0^{T} \ig \cp^l (v\cdot \nnr)\cp^l (\varphi\cdot \nnr).
\end{aligned}
\end{equation}
This implies that $(\eta, v, q)$ verifies
\begin{align}
\p_tv-\bp^2\eta+\pArk q=0,\quad\text{and}\,\,\, \divAr v=0,\quad \text{in}\,\,\,L^2(0,T;H^{-1}(\Omega)),
\end{align}
and so we've shown that $\eta, v$ is indeed a weak solution for \eqref{linearized}. 
\begin{rmk}
We find that it is easier to construct $q$ directly by \eqref{q def} instead of employing the representation argument used in Section 8 of \cite{coutand2007LWP}. 
\end{rmk}

Finally, we consider the difference between \eqref{var} with $v$ and $v'$, respectively, i.e., 
\begin{align}
\begin{aligned}
\int_0^T \langle \p_t (v-v'), \varphi\rangle_{\frac{3}{2}+}+\int_0^T \langle \bp^2(\eta-\eta'), \varphi\rangle_{\frac{3}{2}+}
+\kk \sum_{l=0,1} \int_0^T \ig \cp^l ((v-v')\cdot\nnr)\cp^l( \varphi\cdot \nnr) \\
- \int_0^T \langle (q-q'), \Ark^{\mu\alpha} \p_\mu \varphi\rangle_{\frac{1}{2}+} 
= 0. 
\end{aligned}
\end{align} 
where $(\eta', v',q' )$ is assumed to be another solution with the initial data. The uniqueness of the weak solution follows from setting $\varphi = v-v'$. 

\subsection{$H^1$-Regularity estimates of $v$, $\bp \eta$ and $q$ }
We shall show that $v$, $\bp \eta$ and $q$ are in fact $L^2(0,T;H^1(\Omega))$. 
Let 
\begin{equation}
e(t):= \int_0^t \|\eta\|_1^2+\|v\|_1^2+\|\bp \eta\|_1^2\dt,\quad t\in [0,T].
\end{equation}
Our goal is to show 
\begin{equation}
e(T) \leq \frac{1}{\kk}\Nf,  \label{estimate e(t)}
\end{equation}
for some $T=T(\Nf, \kk)$. 
It suffices to consider $\int_0^T \|v\|_1^2$ and $\int_0^T\|\bp\eta\|_1^2$ only 
since 
\begin{equation*}
\int_0^T \|\eta\|_1^2 \leq \int_0^T\Big( \|\eta_0\|_1^2 + \int_0^t \|v\|_1^2\dt\Big)\dt. 
\end{equation*}
Thanks to Lemma \ref{hodge}(2), it suffices for us to control 
$$\int_0^T \|\di v \|_0^2,\quad \int_0^T \|\curl v \|_0^2, \quad \int_0^T|v^3|_{0.5}^2,$$
 as well as 
 $$\int_0^T \|\di \bp \eta \|_0^2,\quad \int_0^T\|\curl \bp \eta \|_0^2,\quad \int_0^T|\bp \eta^3|_{0.5}^2,$$
 in order to control  $\int_0^T \|v\|_1^2$ and $\int_0^T\|\bp\eta\|_1^2$.
 \subsubsection{Control of the divergence and curl}
The estimates we need here are essentially the same as those in Section \ref{divcurlkk} but without considering the time differentiated quantities. First, since \eqref{compars delta and A} in Lemma \ref{small quantities} remains true with $\Ak$ replaced by $\Ark$, then
\begin{align}
\int_0^T \|\di v\|_0^2 \leq \int_0^T\|(\Ark^{\mu\alpha}-\delta^{\mu\alpha}) \p_\mu v_\alpha\|_0^2\leq \eps\int_0^T \|\p v\|_0^2 \leq \eps e(T). 
\end{align}
Second, because $\divAr\bp \eta$ verifies the evolution equation
\begin{align}
\p_t \divAr (\bp\eta) = [\divAr,\bp] v +(\p_t \Ark^{\mu\alpha})\p_\mu (\bp \eta_\alpha).
\label{div b eq}
\end{align} 
So, one needs to bound 
$$
\itt \|\text{RHS of}\,\, \eqref{div b eq}\|_0^2\dt
$$
 in order to control $\int_0^T\| \divAr (\bp\eta) \|_0^2$. 
 We have
\begin{align}
\itt \|\p_t \Ark^{\mu\alpha} \p_\mu (\bp \eta)\|_0^2\dt \leq& \itt\|\p_t \Ark \|_{L^\infty}^2 \|\p (\bp \eta)\|_0^2\dt \\
\leq& \itt \NN_0\|\p (\bp \eta)\|_0^2\dt \leq T\NN_0e(T). 
\end{align} 
Moreover, by writing $[\divAr,\bp] v  = \Ark^{\mu\alpha} ((\p_\mu b_0) \cdot \p)v_\alpha-(\bp \Ark^{\mu\alpha})\p_\mu v_\alpha$, one gets
\begin{align}
\itt \|[\divAr\bp] v\|_0^2 \leq \itt \NN_0 \|\p v\|_0^2\dt \leq T\NN_0e(T). 
\end{align}
Thus, 
\begin{align}
\int_0^T \|\divAr (\bp\eta)\|_0^2 \leq T\NN_0 e(T).
\end{align}
In addition, since 
\begin{align*}
\|\di \bp\eta\|_0^2 \leq \|\divAr \bp\eta\|_0^2 +\|\Ark-\delta\|_{L^\infty}^2\|\p\bp\eta\|_0^2,
 \end{align*}
invoking \eqref{compars delta and A}, we conclude that 
\begin{align}
 \int_0^T \|\di \bp\eta\|_0^2 \leq \eps e(T)+T\NN_0e(T).
\end{align}

Third, the evolution equation satisfied by $\curlAr v$ and $\curlAr\bp \eta$ reads
\begin{align}
\p_t (\curlAr v)_\alpha - \bp \curlAr(\bp \eta)_\alpha = [\curlAr, \bp] (\bp \eta)_\alpha+\curl_{\p_t\Ark} v_\alpha,
\end{align}
and this yields the following $L^2(0,T;L^2(\Omega))$ energy identity after testing with $\curlAr v$ and integrating in space and time:
\begin{align*}
 \|\curlAr v\|_0^2+\|\curlAr \bp \eta\|_0^2   \lesssim& \Nf+\int_0^t\|[\bp, \curlAr] \bp \eta\|_{0}^2+\int_0^t \|\curl_{\p_t\Ark}  v\|_0^2 \nonumber\\ 
 &+  \int_0^t\|[\bp, \curlAr] \bp v\|_{0}^2+\int_0^t \| \curl_{\p_t\Ark} \bp \eta\|_0^2.
\end{align*}
Integrating in time one more time, we achieve
\begin{align*}
 \int_0^T\Big(\|\curlAr v\|_0^2+\|\curlAr \bp \eta\|_0^2 \Big) \lesssim& T\Nf+ \itt \Big(\|[\bp, \curlAr] \bp \eta\|_{0}^2+ \|\curl_{\p_t\Ark}  v\|_0^2 \Big)\dt\nonumber\\ 
 &+  \itt\Big(\|[\bp, \curlAr] \bp v\|_{0}^2+\| \curl_{\p_t\Ark} \bp \eta\|_0^2\Big)\dt.
\end{align*}
It suffices to control the first two terms on the RHS since the third and fourth terms can then be controlled by an analogous method with the same bound.

For the first term on the RHS, since one can express
$$
[\bp, \curlAr] \bp \eta_\alpha =  \epsilon_{\alpha \beta \gamma}(\bp\Ark^{\nu\beta})\p_\nu \eta^\gamma- \epsilon_{\alpha \beta \gamma}\Ark^{\nu\beta} (\p_\nu b_0\cdot \p) \eta^\gamma
$$
and so 
\begin{equation}
\itt\|[\bp, \curlAr] \bp \eta_\alpha\|_0\dt \lesssim T\NN_0e(T).
\end{equation}
 Similarly, 
\begin{equation}
\|[\bp, \curlAr] \bp v_\alpha\|_0 \leq T\NN_0e(T).
\end{equation} 
 In addition, for the second term, writing $\curl_{\p_t\Ar} v=\epsilon_{\alpha \beta \gamma}(\p_t\Ar^{\nu\beta})\p_\nu v^\gamma$, one obtains
\begin{equation}
\itt \|\curl_{\p_t\Ark} v\|_0\dt \leq  T\NN_0e(T).
\end{equation}
Summing these up, we obtain
\begin{align}
\itt\Big(\|\curlAr v\|_0^2 + \|\curlAr \bp \eta\|_0^2\Big)\dt \leq  T\NN_0e(T). 
\end{align}

  \subsubsection{Control of the boundary terms}
  \label{sect boundary term}
First we state some supplementary results which will come in handy when treating the boundary estimates. The following inequality is a direct consequence of \eqref{product}. Let $f\in H^{0.5}(\p\Omega)$ and $g$ be a smooth function. Then
\begin{align}
|fg|_{0.5} \lesssim |f|_{0.5} |g|_{1+}.
\label{strategy normal'}
\end{align}  
Also, we remark here that \eqref{compars N and nn}, \eqref{comars N and nn in H^3} remain true by replacing $\nnk$ by $\nnr$. 
  
  \noindent\textbf{Control of $\int_0^T |v^3|_{0.5}^2$:}
 It suffices to control $\int_0^T |v\cdot \nnr|_{0.5}^2$ since
 \begin{equation}
 \int_0^T |v^3|_{0.5}^2 \leq \int_0^T |v\cdot \nnr|_{0.5}^2+\int_0^T |v\cdot (\nnr-N)|_{0.5}^2,
 \end{equation}
 where, after invoking \eqref{strategy normal'} and the trace lemma, we have
 \begin{align}
 \int_0^T |v\cdot (\nnr-N)|_{0.5}^2 \leq  \int_0^T |v|_{0.5}^2 |\nnr-N|_{1+}^2 \lesssim \eps e(T).  
 \end{align}
 Moreover, the control of $\int_0^T |v\cdot \nnr|_{0.5}^2$ is a direct consequence of \eqref{L^2 estimate lambda independent},
\begin{align}
\int_0^T |v\cdot \nnr|_{0.5}^2 \lesssim \frac{1}{\kk} \int_0^T|v\cdot\nnr|_1^2 \leq \frac{\NN_0}{\kk}. 
\label{v^3 0.5}
\end{align}
 \noindent\textbf{Control of $\int_0^T |\bp \eta^3|_{0.5}^2$:} 
 
 Similar to the control $\int_0^T |v^3|_{0.5}^2$, it suffices to bound $\int_0^T |\bp (\eta\cdot \nnr)|_{0.5}^2$ only. 
 
 Since $\bp|_{\Gamma}=b_0\cdot \TP$ and $\TP (\eta\cdot \nnr)|_{t=0}=\TP \eta^3|_{t=0}=0$, we have 
 \begin{equation}
\bp (\eta\cdot \nnr) = \int_0^T \p_t\bp (\eta\cdot\nnr)\dt.
 \end{equation} 
 Hence,
 \begin{align*}
 \int_0^T |\bp (\eta\cdot \nnr)|_{0.5}^2 \leq \int_0^T \left| \int_0^t \p_t\bp (\eta\cdot\nnr)\right|^2_{0.5}\dt\lesssim \itt |\p_t\bp (\eta\cdot \nnr)|_{0.5}^2\dt,
 \end{align*}
 by Jensen's inequality. Here, the term on the RHS is equal to
 \begin{align*}
 \itt |\bp (v \cdot\nnr)|_{0.5}^2\dt + \itt |\bp (\eta\cdot \p_t\nnr)|_{0.5}^2=I+II.  
 \end{align*}
 Since $\p_t \nnr = Q(\TP \erk)\TP \mathring{\vk}\cdot \nnr$, invoking \eqref{strategy normal'} and the trace lemma, we have $II\leq T\NN_0e(T)$. 
Next, invoking \eqref{strategy normal'}, we have
\begin{align*}
I \lesssim ||b_0||_{0.5} \itt |v\cdot \nnr|_2^2\dt. 
\end{align*}
By employing the boundary condition we obtain the following elliptic equation verified by $v\cdot \nnr$ on $\Gamma$:
\begin{align}
\TL (v\cdot \nnr) = \frac{1}{\kk}\Big((v\cdot \nnr) +\sqrt{\grk} q +\sigma\sqrt{\gr}\Delta_{\gr} \er\cdot \nnr\Big). 
\label{v^3 ell eq}
\end{align}
 By the virtual of the elliptic estimate, we have
\begin{align}
\itt |v\cdot \nnr |_{2+}^2\dt \leq \frac{C}{\kk^2}\itt \Big(|v\cdot \nnr|_{0+}^2+|\sqrt{\grk} q|_{0+}^2 +\sigma|\sqrt{\gr}\Delta_{\gr} \er^{3}|_{0+}^2\Big)\dt,
\label{bdy ell for v}
\end{align} 
where $C$ denotes a generic constant. 
 It is clear that the third term can be controlled by $\frac{1}{\kk^2} T\NN_0$, and the first term is bounded by $\frac{1}{\kk^3}T\NN_0$ in view of \eqref{v^3 0.5}. Therefore, 
\begin{equation}
\itt |v\cdot \nnr|_{2+}^2 \leq \frac{1}{\kk^2} T\NN_0 +\frac{1}{\kk^3} T\NN_0 + \frac{1}{\kk}\itt \NN_0 |q|_{0+}^2\dt.
\label{control of v^3 in 2+}
\end{equation}
Now, we impose the bootstrap assumption
\begin{align}
\int_0^T \|q\|_{H^{\frac{1}{2}+}(\Omega)}^2 \leq \frac{\mathcal{N}_0}{\kk}.
\label{q 1/2+}
\end{align}
In light of this and the trace lemma, the second term on the RHS of \eqref{control of v^3 in 2+} is bounded by $\frac{1}{\kk^2}T\Nf$. Hence,
\begin{align}
\itt |v\cdot \nnr|_{2+}^2 \leq \left(\frac{1}{\kk^2}+\frac{1}{\kk^3}\right)T\Nf.
\label{control of v^3 in 2+ done}
\end{align}

In summary, we have
 \begin{equation}
 e(T) \leq \frac{1}{\kk}\Nf+\left(\frac{1}{\kk^2}+\frac{1}{\kk^3}\right)T\Nf+ \eps e(T) + T\NN_0 e(T),
 \end{equation}
 and this implies \eqref{estimate e(t)} if $T$ is chosen sufficiently small, say $T=\min(\frac{\eps}{\NN_0}, \kk^3)$. 

\noindent\textbf{Closing the bootstrap argument:} We test \eqref{linearized}  with $\varphi\in L^\infty(0,T; H^{\frac{1}{2}+}(\Omega))$, $\|\varphi\|_{\frac{1}{2}+}=1$, and integrate in time, to get 
  \begin{equation}\label{var 1/2+}
\begin{aligned}
\int_0^{T} \langle \p_t v, \varphi\rangle_{\frac{1}{2}+}+\int_0^{T} \langle \bp^2\eta, \varphi\rangle_{\frac{1}{2}+}
- \int_0^{T} \langle q, \divAr \varphi\rangle_{\frac{1}{2}+}\\
= \sigma \int_0^{T} \ig (\sqrt{\gr}\lap_{\gr} \er\cdot \nnr) (\varphi\cdot\nnr)-\kk \sum_{l=0,1} \int_0^{T} \ig \cp^l (v\cdot \nnr)\cp^l (\varphi\cdot \nnr).
\end{aligned}
\end{equation}
 The first term on the RHS is $\leq \sigma \Nf\leq \Nf$. For the second term on the RHS,
 we integrate $\TP$ by parts and obtain 
 \begin{align*}
 -\kk \sum_{l=0,1} \int_0^{T} \ig \cp^l (v\cdot \nnr)\cp^l (\varphi\cdot \nnr)=\kk \sum_{l=0,1,2} \int_0^{T} \ig \cp^l (v\cdot \nnr)(\varphi\cdot \nnr)\\
 \leq \kk \Nf \int_0^T |v\cdot\nnr|_2
 \end{align*}
It suffices to control $\kk\int_0^T |v\cdot\nnr|_{2+}$. Invoking \eqref{v^3 ell eq} and we have, by the elliptic estimate, that
\begin{equation}
\kk\int_0^T |v\cdot \nnr |_{2+} \leq \int_0^T |v\cdot \nnr|_{0+}+\int_0^T |\sqrt{\grk} q|_{0+} +\int_0^T \sigma|\sqrt{\gr}\Delta_{\gr} \er^{3}|_{0+}.
\end{equation}
The third term is bounded by $T\Nf$. We square the second term and then apply Jensen's inequality to obtain
\begin{equation}
\left(\int_0^T |\sqrt{\grk} q|_{0+} \right)^2 \leq T\Nf \int_0^T \|q\|_{\frac{1}{2}+}^2\leq \Nf^2,
\end{equation}
where the last inequality follows from $T\leq \kk^3\leq \kk$, and so
\begin{equation}
\int_0^T |\sqrt{\grk} q|_{0+}  \leq \Nf. 
\end{equation} 
Similarly, 
\begin{equation}
\left(\int_0^T |v\cdot\nnr|_{0+}\right)^2 \leq T\int_0^T |v\cdot \nnr|_{0.5+}^2 \leq \frac{T}{\kk} \Nf, 
\end{equation}
after invoking \eqref{v^3 0.5}, and so
\begin{equation}
\int_0^T |v\cdot\nnr|_{0+} \leq \Nf.
\end{equation}
Summing up, the RHS of \eqref{var 1/2+} is bounded by $\Nf$. 
Therefore, by employing arguments parallel to \eqref{variational eq T}-\eqref{7.36}, we have
\begin{equation}\label{v_t H1/2 dual}
\lim_{\la\to 0} \left(\int_0^T \|\p_t v\|_{H^{\frac{1}{2}+}(\Omega)'}^2+\int_0^T \|\bp^2\eta \|_{H^{\frac{1}{2}+}(\Omega)'}^2 \right)\leq \NN_0.
\end{equation}
Moreover, by employing the Lagrange multiplier lemma (i.e., \cite[Lemma 7.4]{coutand2007LWP}), 
\eqref{var 1/2+} implies that
\begin{align}
\int_0^T \|q\|_{H^{\frac{1}{2}+}(\Omega)}^2 \leq \mathcal{N}_0,
\end{align}
which improves the bootstrap assumption \eqref{q 1/2+} whenever $\kk$ is fixed to be sufficiently small. 
\begin{rmk}
We need to modify the Lagrange multiplier lemma slightly. Since we need our $q\in H^{\frac{1}{2}+
}(\Omega)$, we need to consider the linear functional 
$
\langle \divAr \varphi, q\rangle_{\frac{1}{2}+}
$
defined on $X(t)$, 
where $X(t)=\{\varphi: \divAr \varphi \in H^{\frac{1}{2}+}(\Omega)'\}$. 
\end{rmk}
\subsubsection{The strong solution for the linearized equations}
Since $v,\bp \eta \in L^2(0,T;H^1(\Omega))$ and so $\p_t v,\bp^2\eta \in L^2(0,T; L^2(\Omega))$, we can now adapt Lemma \ref{H1elliptic} to bound $q$ in $L^2(0,T;H^1(\Omega))$. 

Therefore, we have obtained a strong solution for the linearized $\kk$-problem \eqref{linearized}. This allows us to further boost the regularity of the linearized solution to $H^{4.5}(\Omega)$ via classical methods in the upcoming section. Then we can achieve a solution for the nonlinear $\kk$-problem by approximating it by a sequence of linearized solutions.

\section{Existence for the nonlinear approximate $\kk$-problem}\label{sect exist kk-prob}

 We aim to construct a solution for the nonlinear $\kk$-problem for fixed $0< \kk\ll 1$.  Let $(\eta_{(0)}, v_{(0)})=(Id, 0)$. For each $m\geq 0$, Let $(\eta_{(m)}, v_{(m)}, q_{(m)})$ be the solution for 
\begin{equation}\label{linear kk with m}
\begin{cases}
\p_t\eta_{(m)}=v_{(m)}~~~& \text{in}~[0,T]\times\Omega;\\
\p_tv_{(m)}-\bp^2\eta_{(m)}+\nab_{\Ak_{(m-1)}} q_{(m)}=0~~~& \text{in}~[0,T]\times\Omega;\\
\di_{\Ak_{(m-1)}} v_{(m)}=0,\quad \dive b_0=0 &\text{in}~[0,T]\times\Omega;\\
v_{(m)}^3=b_0^3=0~~~&\text{on}~\Gamma_0;\\
\Ak_{(m-1)}^{3\alpha}q_{(m)} = -\sigma \sqrt{g_{(m-1)}}(\Delta_{g_{(m-1)}} \eta_{(m-1)}\cdot \nnk_{(m-1)} )\nnk_{(m-1)}^\alpha +\kk(1-\TL) (v_{(m)}\cdot \nnk_{(m-1)})\nnk_{(m-1)}^\alpha~~~&\text{on}~\Gamma;\\
(\eta_{(m)},v_{(m)})=(Id,v_0)~~~&\text{on}~\{t=0\}{\times}\overline{\Omega}.
\end{cases}
\end{equation}  
Here, the (linearized) coefficients, e.g.,  $\Ak_{(m-1)}, g_{(m-1)}, \nnk_{(m-1)}$ are determined by $(\eta_{(m-1)}, v_{(m-1)})$. 

The goal is to prove that the sequence $\{(\eta_{\mm}, v_{\mm})\}_{m\geq 0}$ strongly converges (and so does $q_{(m)}$), and the limit verifies the nonlinear approximate $\kk$-problem. This is done by applying Picard's iteration. 

\subsection{A priori estimate of the linearized approximate problem}
We first establish the $H^{4.5}$-energy estimate for $(\eta_{\mm}, v_{\mm})$, and then this estimate can be carried over to the difference between two successive systems \eqref{linear kk with m} which yields the convergence of $(\eta_{\mm}, v_{\mm})$ in $H^{3.5}$ as $m\to \infty$.  Let $m\geq 1$ be fixed and assume the solutions $(\eta_{(l)}, v_{(l)})$ are known for all $l\leq m$. For the sake of clean notations, we will denote $(\eta_{(m)}, v_{(m)})$ by $(\eta, v)$ and $(\eta_{(m-1)}, v_{(m-1)})$ by $(\er, \mathring{v})$ if no confusion is raised.  Let $\Ak_{(m-1)}:=\Ark, g_{(m-1)}:=\gr, \nnk_{(m-1)}:=\nnr$.  Then  \eqref{linear kk with m} reduces exactly to the linearized system \eqref{linearized}.

\begin{prop}\label{linearkk}
For each fixed $\kk>0$,
there exists some $T_\kk>0$ such that the solution $(\eta,v)$ for \eqref{linearized} satisfies
\begin{equation}\label{energylkk}
\sup_{0\leq t\leq T_\kk}\Er(t)\leq \mathcal{C},
\end{equation}
where $\mathcal{C}$ is a constant depends on $\|v_0\|_{4.5}$, $\|b_0\|_{4.5}, |v_0|_5$, provided that
\begin{align}
\label{smallark} \|\Jrk(t)-1\|_{3.5}+\|Id-\Ark(t)\|_{3.5}+\|Id-\Ark^T\Ark(t)\|_{3.5}\leq \eps.
\end{align} 
holds for all $t\in[0,T_\kk]$. 
Here the energy functional $\Er$ of \eqref{linearized} is defined to be
\begin{equation}\label{Ekk linear}
\Er(t)=\Er^{(1)}(t)+\Er^{(2)}(t),
\end{equation}
where
\begin{equation}
\begin{aligned}
\Er^{(1)}(t)&:=\left\|\eta\right\|_{4.5}^2+\left\|v\right\|_{4.5}^2+\left\|\p_t v\right\|_{3.5}^2+\left\|\p_t^2 v\right\|_{2.5}^2+\left\|\p_t^3 v\right\|_{1.5}^2+\left\|\p_t^4 v\right\|_{0}^2\\
&+\left\|\bp\eta\right\|_{4.5}^2+\left\|\p_t \bp\eta\right\|_{3.5}^2+\left\|\p_t^2 \bp\eta\right\|_{2.5}^2+\left\|\p_t^3 \bp\eta\right\|_{1.5}^2+\left\|\p_t^4 \bp\eta\right\|_{0}^2\\
\Er^{(2)}(t)&:=\frac{\kk}{\sigma}\int_0^T\left|\p_t^4 v\cdot\nnr\right|_1^2\dt+\kk\bigg(\int_0^T \|\p_t^4 v\|_{1.5}^2+\int_0^T \|\p_t^4 \bp \eta\|_{1.5}^2\bigg).\nonumber
\end{aligned}
\end{equation}
\end{prop}

Thanks to the Gronwall's inequality, \eqref{energylkk} is a direct consequence of 
\begin{align}
\sup_{0\leq t\leq T_\kk}\Er(t) \lesssim_{\kk^{-1}} \mathcal{C}(\|v_0\|_{4.5}, \|b_0\|_{4.5}, |v_0|_5) +C(\eps)\sup_{0\leq t\leq T_\kk}\Er(t)+(\sup_{0\leq t\leq T_\kk}\PP)\int_0^{T_{\kk}} \PP,
\label{ronwall'}
\end{align}
where $\PP = P(\Er(t), \|\vr\|_{4.5}, \|\bp \er\|_{4.5} )$ (after a slight abuse of notations). Also, since $\kk$ is fixed, we will drop the subscript $\kk$ and denote $T_\kk=T$ for the sake of clean notations. Similar to \eqref{Gronwall} we shall assume that $\sup_{0\leq t\leq T}\Er(t) =\Er(T)$, and this allows us to drop $\sup\sup_{0\leq t\leq T_\kk}$ in \eqref{ronwall'}. In other words, we only need to show
\begin{align}
\Er(T) \lesssim_{\kk^{-1}} \PP_0+C(\eps)\Er(T)+\PP\int_0^T \PP,
\label{Gronwall'}
\end{align}
where $\PP_0 = \PP(\Er(0), \|q(0)\|_{4.5}, \|q_t(0)\|_{3.5}, \|q_{tt}(0)\|_{2.5})$. Also, it suffices to put $\PP_0$ on the RHS of \eqref{Gronwall'} since \eqref{removing extra regularity} allows us to control $\PP_0$ by $\mathcal{C}$. 

\begin{rmk}It can be seen that $\Er(t)$ constructed above is significantly simpler than $E_\kk(t)$ given in \eqref{Ekk}. The main reason is that the energy estimate for \eqref{Ekk linear} does not have to be uniform in $\kk$. Instead, the right side of \eqref{Gronwall'} may depend on $\kk^{-1}$. Besides, the surface tension term $-\sigma \sqrt{\gr}(\Delta_{\gr} \er\cdot \nnr )\nnr^\alpha$ now serves as a given source term in the boundary condition. Based on these two facts, we can significantly simplify the following steps, compared with the nonlinear estimates in Section \ref{tgkk} -- Section \ref{close}. 
\begin{enumerate}
\item \textbf{Elliptic estimates for the boundary normal traces.}  Unlike section \ref{sect bdy estimate 3.3}, we now regard the boundary condition of \eqref{linearized} as an elliptic equation of $v\cdot \nnr$ 
\[
\TL (v\cdot \nnr)=v\cdot \nnr-\kk^{-1}\left(\sqrt{\grk}q+\sigma\sqrt{\gr}\Delta_{\gr}\er\cdot\nnr\right).
\] Since the surface tension term is a given fixed-term here, one can directly control $v\cdot\nn$ (and its time derivatives) via elliptic estimates. Therefore, we now only need to perform the tangential energy estimate of full-time derivatives. We refer to Subsection \ref{sect boundary est simplified} for the details.
\item \textbf{Removing extra (tangential) spatial derivatives via the mollifier properties.} In light of \eqref{lkk2}, we can absorb additional tangential spatial derivatives when necessary. This will allow us to greatly simplify most of the estimates including the surface tension terms, artificial viscosity terms, and the error terms involving $N-\nnr$ on the boundary.  Compared to $E_\kk^{(2)}, E_{\kk}^{(3)}$ in \eqref{Ekk}, much fewer $\kk$-weighed higher-order terms are needed here in $\mathcal{E}^{(2)}(t)$ to close the energy estimates. In addition to that, we do not need to establish the improved boundary regularity as in Lemma \ref{super control}.
\end{enumerate}
We will mainly focus on the simplified parts as mentioned above in the following calculations. The div-curl estimates and the elliptic estimates for $q$ are essentially identical to Section \ref{nonlinearkk}, so we will only outline them in Subsection \ref{Subsect. 8.1.3}  but without repeating the detailed proof. 
\end{rmk}

\subsubsection{Boundary estimates} \label{sect boundary est simplified}
This subsection is devoted to the control of boundary normal traces $|\p_t^k v^3|_{4-k}$ and $|\p_t^k \bp \eta^3|_{4-k}$ for $k=0,1,2,3$. 
\begin{lem}\label{lemma 8.2}
For $k=0,1,2,3$, we have
\begin{align}
|\p_t^k v^3|_{4-k}^2 \lesssim_{\kk^{-1}} \PP_0 + C(\eps)\Er(T)+\PP\int_0^T\PP,\\
|\p_t^k \bp\eta^3|_{4-k}^2 \lesssim_{\kk^{-1}} \PP_0 + C(\eps)\Er(T)+\PP\int_0^T\PP.
\end{align}
\end{lem}
Note that we no longer require the energy bound to be $\kk$-independent. Hence, \textit{we are able to use \eqref{lkk2} to absorb extra tangential spatial derivatives on the smoothed variables, i.e., variables with $\sim$ on top.} We can absorb at most two tangential spatial derivatives since $\tilde{\cdot}=\lkk^2\cdot$ on the boundary.  Recall that the boundary condition in the linearized equations reads
\begin{equation}
\sqrt{\grk}q = -\sigma \sqrt{\gr}\Delta_{\gr} \er\cdot \nnr +\kk(1-\TL) (v\cdot \nnr).
\label{BC linear}
\end{equation}
This can be converted to an elliptic equation satisfied by $v\cdot \nnr$, i.e., 
\begin{equation}\label{ellll}
\TL (v\cdot \nnr)=v\cdot \nnr-\kk^{-1}\left(\sqrt{\grk}q+\sigma\sqrt{\gr}\Delta_{\gr}\er\cdot\nnr\right).
\end{equation}
Now, invoking the standard elliptic estimate and \eqref{lapg},  we get
\begin{equation}\label{v3lbdry}
\begin{aligned}
|v\cdot \nnr|_4^2\lesssim&|v\cdot \nnr|_2^2+\kk^{-1}\left(\left|\sqrt{\grk} q\right|_2^2+\sigma P(|\TP \er|_{L^\infty}, |\TP^2 \er|_{L^\infty})|\er|_4^2\right)\\
\lesssim_{\kk^{-1}}& \PP_0+\int_0^T\PP,
\end{aligned}
\end{equation}
where the used the trace lemma and \eqref{ellq} in the second inequality. 

For the magnetic field, since $\bp=b_0^j \TP_j$ on $\Gamma$ and hence $\bp (\eta\cdot \nnr) |_{t=0}=0$. Thus, 
\begin{equation}
\bp (\eta\cdot \nnr) = \int_0^T \p_t \Big(\bp (\eta\cdot \nnr)\Big)=\int_0^T \bp (v\cdot\nnr) + \int_0^T\bp (\eta \cdot \p_t \nnr).
\end{equation}
Since $\p_t \nnr = -\grk^{kl} \TP_k \vrk\cdot \nnr \TP_l \erk=Q(\TP \erk)\TP \vrk\cdot \nnr$, and invoking \eqref{lkk2} and the Jensen's inequality,  we have
\begin{align}
\left| \int_0^T\bp (\eta \cdot \p_t \nnr) \right|_4^2 \lesssim T\int_0^T |\bp (\eta\cdot \p_t\nnr)|_4^2 
\lesssim_{\kk^{-1}} \int_0^T\PP
\end{align}
Here, we need \eqref{lkk2} in order to control the leading order term generated when $\TP^4\bp$ fall on $\TP \vrk$ (which is part of $\p_t \nnr$), i.e., 
$$
\int_0^T Q(|\erk|_{L^\infty}, |\TP\erk|_{L^\infty}) |\bp \TP \vrk|_4^2 \leq \int_0^T |b_0|_4^2 Q(|\erk|_{L^\infty}, |\TP\erk|_{L^\infty})|\TP^2 \vrk|_4^2\lesssim_{\kk^{-1}} \int_0^T |b_0|_4^2 Q(|\erk|_{L^\infty}, |\TP\erk|_{L^\infty})|\vr|_4^2. 
$$
In addition, 
\begin{align}
\left|\int_0^T \bp (v\cdot\nnr) \right|_4^2 \lesssim T\int_0^T |\bp (v\cdot\nnr)|_4^2, 
\end{align}
and the RHS can be controlled by studying the elliptic equation satisfied by $\bp (v\cdot \nnr)$. Taking $\bp$ on \eqref{ellll} and we get
\begin{equation}\label{ell bp v}
\TL \bp (v\cdot \nnr)=[\TL, \bp] (v\cdot \nnr)+\bp(v\cdot \nnr)-\kk^{-1}\left(\bp(\sqrt{\grk}q)+\sigma\bp(\sqrt{\gr}\Delta_{\gr}\er\cdot\nnr)\right),
\end{equation}
then the elliptic estimate implies
\begin{align}
\int_0^T |\bp (v\cdot\nnr)|_4^2 &\lesssim_{\kk^{-1}} P(|b_0|_4)\int_0^T \left(|v\cdot \nnr|_4^2 + \left|\sqrt{\grk} q\right|_3^2+\sigma P(|\er|_{H^4})|\bp \er|_4^2\right)\lesssim_{\kk^{-1}} \int_0^T\PP. 
\end{align}
Thus, 
\begin{equation}\label{b3lbdry}
|\bp (\eta\cdot \nnr)|_4^2 \lesssim_{\kk} \int_0^T\PP. 
\end{equation}
We can obtain the bounds for $|v^3|_4^2$ and $|\bp \eta^3|_4^2$ from \eqref{v3lbdry} and \eqref{b3lbdry}, respectively. Indeed, we have
\begin{align}
\label{from1}|v^3|_4^2 \leq&|v\cdot \nnr|_4^2 + |v\cdot (N-\nnr)|_4^2,\\
\label{from2}|\bp\eta^3|_4^2 \leq& |\bp(\eta\cdot\nnr)|_4^2 + |\bp (\eta\cdot (N-\nnr))|_4^2. 
\end{align}
Since 
$$
N-\nnr = -\int_0^T \p_t\nnr = \int_0^T Q(\TP \erk)\TP \vrk\cdot \nnr, 
$$
invoking \eqref{lkk2} and the proof of \eqref{comars N and nn in H^3} , we have
\begin{align} \label{N-nnr}
|N-\nnr|_5\lesssim_{\kk^{-1}} \int_0^T\PP. 
\end{align}
Therefore,
\begin{align} \label{to}
 |v\cdot (N-\nnr)|_4^2 +  |\bp (\eta\cdot (N-\nnr))|_4^2 \lesssim_{\kk^{-1}} \PP\int_0^T\PP.
\end{align}

Now, we can take time derivative $\p_t$ in \eqref{ellll} to get the elliptic equation of $\p_t (v\cdot\nnr)$ on the boundary, i.e.,
\begin{align}
\label{ellll1} \TL \p_t(v\cdot \nnr)=&\p_t(v\cdot \nnr)-\kk^{-1}\left(\p_t(\sqrt{\grk}q)+\sigma\p_t(\sqrt{\gr}\Delta_{\gr}\er\cdot \nnr)\right).
\end{align}
Then standard elliptic estimate gives
\begin{equation}\label{vt3lbdry}
\begin{aligned}
|\p_t (v\cdot \nnr)|_3^2\lesssim_{\kk^{-1}}&~|\p_t(v\cdot\nnr)|_1^2+\left|\p_t\left(\Ark^{33} q\right)\right|_1^2+\sigma P(|\TP \er|_{L^\infty}, |\TP^2 \er|_{L^\infty}, |\TP \vr|_{L^\infty})|\vr|_3^2\\
\lesssim_{\kk^{-1}}&~ \PP_0+\int_0^T\PP.
\end{aligned}
\end{equation}
This estimate implies the estimate for $|\p_t v^3|_3^2$ by writing $|\p_t v^3|_3^2 \leq |\p_t (v\cdot \nnr)|_3^2+|\p_t(v\cdot (N-\nnr))|_3^2$ and then adapting the arguments from \eqref{from1}-\eqref{to}. 
Moreover, in light of the estimate for $|v^3|_4^2$, we have
\begin{equation}\label{bt3lbdry}
|\p_t\bp \eta^3|_3= |\bp v^3|_3^2 \leq P(|b_0|_3) |v^3|_4^2 \lesssim_{\kk^{-1}} \PP_0 +\PP\int_0^T\PP.
\end{equation}

Similarly, by taking two time derivatives to \eqref{ellll}, we can control $|\p_t^2 (v\cdot \nnr)|_2$ by the standard elliptic estimate, i.e.,  
\begin{equation} 
|\p_t^2 (v\cdot \nnr)|_2^2
\lesssim_{\kk^{-1}} \PP_0+\int_0^T\PP,
\end{equation}
and this yields 
\begin{equation} \label{vtt3lbdry}
|\p_t^2 v^3|_2^2\lesssim_{\kk^{-1}} \PP_0 +\PP\int_0^T\PP. 
\end{equation}
In addition to this,   $|\p_t^2 \bp \eta^3|_2^2$ reduces to $|\p_t v^3|_3^2$, whose bound is given above. Also, in the case when there are three time derivatives, $|\p_t^3\bp \eta^3|_1^2$ reduces to $|\p_t^2 v^3|_2^2$, which is just \eqref{vtt3lbdry}.

Finally, $|\p_t^3 v^3|_1^2$ can be controlled with the help $\Er^{(2)}$. We can make use of the $\kk$-weighted higher-order terms to directly control the time-integrated terms on the boundary. Specifically, by writing
$
|\p_t^3 v^3|_1 \leq \PP_0 + \int_0^T |\p_t^4 v^3|_1,
$
we have
\begin{align}
|\p_t^3 v^3|_1^2 \lesssim \PP_0 + \left(\int_0^T |\p_t^4 v^3|_1\right)^2 \lesssim \PP_0 + T\int_0^T |\p_t^4 v^3|_1^2,
\end{align}
where 
\begin{align}
T\int_0^T |\p_t^4 v^3|_1^2 \leq T \int_0^T |\p_t^4 v\cdot \nnr|_1^2+ T\int_0^T |\p_t^4 v\cdot (N-\nnr)|_1^2.
\end{align}
Here, the second term on the RHS is $\lesssim\kk^{-1} TC(\eps)\int_0^T \|\sqrt{\kk}\p_t^4 v\|_{1.5}^2$ whereas the first term is $\lesssim\kk^{-1}T\int_0^T|\sqrt{\kk}\p_t^4v \cdot \nnr |_{1}^2$. Therefore, by choosing $T$ sufficiently small, we have
\begin{equation}
|\p_t^3 v^3|_1^2 \lesssim_{\kk^{-1}} \PP_0 +C(\eps) \Er(T) + \PP\int_0^T\PP. 
\label{vttt3lbdry}
\end{equation}

\subsubsection{Tangential estimate with four time derivatives}\label{lineartg}
We still need to control 
$$
\|\p_t^4 v\|_0^2,\quad \|\p_t^4 \bp \eta\|_0^2
$$ 
in order to finish the control of $\Er$. In fact, we only need to control $\|\p_t^4 v\|_0^2$ since $\|\p_t^4 \bp \eta\|_0^2$ reduces to $\|\p_t^3 v\|_1^2$ which has been done previously.

Now we compute the $L^2$-estimate of $\p_t^4 v$ and $\p_t^4\bp\eta$. Invoking \eqref{linearized} and integrating $\bp$ by parts, we get
\begin{equation}\label{tgtl40}
\begin{aligned}
&\frac12\int_0^T\frac{d}{dt}\io|\p_t^4 v|^2+\left|\p_t^4\bp\eta\right|^2\dy\dt=-\int_0^T\io\p_t^4v_{\alpha}\p_t^4(\Ark^{\mu\alpha}\p_{\mu} q)\dy\dt\\
=&-\int_0^T\io\p_t^4 v_{\alpha}\Ark^{\mu\alpha}\p_t^4\p_{\mu}q\dy\dt\underbrace{-\int_0^T\io\p_t^4 v_{\alpha}~\left[\p_t^4,\Ark^{\mu\alpha}\right]\p_{\mu}q\dy\dt}_{\Ir_1}\\
=&\int_0^T\io\Ark^{\mu\alpha}\p_t^4\p_{\mu}v_{\alpha}\p_t^4q\dy\dt\underbrace{-\int_0^T\ig\p_t^4v_{\alpha}\Ark^{3\alpha}\p_t^4 q\dS\dt}_{\Ir_B}\underbrace{+\int_0^T\int_{\Gamma_0}\p_t^4v_{\alpha}\Ark^{3\alpha}\p_t^4 q\dS\dt}_{=0}+\Ir_1\\
=&\int_0^T\io\underbrace{\p_t^4(\divAr v)}_{=0}\p_t^4 q\dy\dt\underbrace{+\int_0^T\io\left[\Ark^{\mu\alpha},\p_t^4\right]\p_{\mu}v_{\alpha}~\p_t^4 q\dy\dt}_{\Ir_2}+\Ir_B+\Ir_1.
\end{aligned}
\end{equation}
Here, $\Ir_1$ and $\Ir_2$ can be straightforwardly controlled by $\int_0^T\PP$.  We start to analyze the boundary integral $I_B$. 
\begin{equation}\label{IrB0}
\begin{aligned}
\Ir_B=&-\int_0^T\ig\p_t^4v_{\alpha}\Ark^{3\alpha}\p_t^4 q=-\int_0^T\ig\sqrt{\grk}(\p_t^4v\cdot \nnr)(\p_t^4 q)\\
=&\sigma\int_0^T\ig\sqrt{\grk}(\p_t^4v\cdot \nnr)\p_t^4(\sqrt{\gr}\grk^{-\frac{1}{2}}\Delta_{\gr}\er\cdot \nnr)\\
-&\kk\int_0^T\ig\sqrt{\grk}(\p_t^4v\cdot \nnr)\p_t^4\left(\grk^{-\frac{1}{2}}(1-\TL)(v\cdot \nnr)\right)
:=\Ir_{B1}+\Ir_{B2}.
\end{aligned}
\end{equation}
Invoking the identity \eqref{lapg}, we have
\begin{equation}
\begin{aligned}
\Ir_{B1} =& \sigma\int_0^T\ig\sqrt{\grk}(\p_t^4v\cdot \nnr)\p_t^4(\sqrt{\gr}\grk^{-\frac{1}{2}}\gr^{ij}\TP_i\TP_j\er\cdot \nnr)\\
&-\sigma\int_0^T\ig\sqrt{\grk}(\p_t^4v\cdot \nnr)\p_t^4(\sqrt{\gr}\grk^{-\frac{1}{2}}\gr^{ij}\gr^{kl}\TP_l\er^\mu\TP_i\TP_j\er_{\mu}\TP_k\er\cdot \nnr)\\
=&\Ir_{B11}+\Ir_{B12}.
\end{aligned}
\end{equation}
Since
$$
\Ir_{B11} \LL\sigma\int_0^T\ig(\p_t^4v\cdot \nnr)(\sqrt{\gr}\gr^{ij}\TP_i\TP_j\p_t^3 \vr\cdot \nnr)
$$
 we integrate $\TP_i$ by parts and get
\begin{align*}
\Ir_{B11} \LL& -\sigma\int_0^T\ig(\TP_i\p_t^4v\cdot \nnr)(\sqrt{\gr}\gr^{ij}\TP_j\p_t^3 \vr\cdot \nnr)\\
\lesssim_{\kk^{-1}}& \eps \int_0^T \|\sqrt{\kk}\p_t^4v\|_{1.5}^2 + \int_0^T\PP \leq \eps \Er(T)+\int_0^T\PP,
\end{align*}
and $\Ir_{B12}$ can be treated in the same fashion. 

Next we study $\Ir_{B2}$. We have
\begin{align}
\Ir_{B2}\LL \kk\int_0^T\ig(\p_t^4v\cdot \nnr)\TL(\p_t^4 v\cdot \nnr)+\kk\int_0^T\ig(\p_t^4v\cdot \nnr)\TL( v\cdot \p_t^4\nnr):=\Ir_{B21}+\Ir_{B22},
\end{align}
where $\Ir_{B21}$ contributes to the positive energy term $\int_0^T|\p_t^4 v\cdot \nnr|_1^2$ after integrating $\TP$ by parts and moving the resulting term to the LHS. In addition, since $\p_t^4\nnr= Q(\TP \erk) \TP \p_t^3 \vrk\cdot \nnr+$ lower-order terms, 
\begin{align}
\Ir_{B22} \LL & -\kk\int_0^T\ig(\TP\p_t^4v\cdot \nnr)( v\cdot  Q(\TP \erk) \TP^2 \p_t^3 \vrk\cdot \nnr)\nonumber\\
\lesssim& \eps \int_0^T \|\p_t^4 v\|_{1.5}^2 + \int_0^T |v|_{L^\infty}^2 Q(|\TP\erk|_{L^\infty})|\TP^2\p_t^3\vrk|_0^2\nonumber\\
\lesssim_{\kk^{-1}}& \eps \Er(T) + \int_0^T |v|_{L^\infty}^2 Q(|\TP\erk|_{L^\infty})|\TP\p_t^3\vr|_0^2 \leq \eps\Er(T)+\int_0^T\PP,
\end{align}
where we used \eqref{lkk2} in the second to the last inequality to control $|\TP^2\p_t^3 \vrk|_0^2\lesssim \kk^{-1}|\TP \p_t^3 \vr|_0^2$. 

\subsubsection{Interior estimates}\label{Subsect. 8.1.3}
To control $\|\p_t^k v\|_{4.5-k}^2$ and $\|\p_t^k \bp \eta\|_{4.5-k}^2$,  $k=0,1,2,3,$ we only need to apply the div-curl estimate:
\begin{align}
\|\p_t^k v\|_{4.5-k}^2 \lesssim& \|\p_t^k\di v\|_{3.5-k}^2+\|\p_t^k \curl v\|_{3.5-k}^2+ |\p_t^k v^3|_{4-k}^2,\\
\|\p_t^k \bp\eta\|_{4.5-k}^2 \lesssim& \|\p_t^k\di \bp \eta\|_{3.5-k}^2+\|\p_t^k \curl \bp\eta\|_{3.5-k}^2+ |\p_t^k \bp \eta^3|_{4-k}^2.
\end{align}
The boundary terms $|\p_t^k v^3|_{4-k}^2$ and $|\p_t^k \bp \eta^3|_{4-k}^2$ are treated in Lemma \ref{lemma 8.2}.
The estimates for the divergence and curl of $v$ and $\bp\eta$, together with their time derivatives are identical to those in Section \ref{divcurlkk}, and so we shall not repeat the proofs. Furthermore, the estimate for the top order interior term in $\Er^{(2)}$, i.e.,   
\begin{align}
\kk\Big(\int_0^T \|\p_t^4 v\|_{1.5}^2+\int_0^T \|\p_t^4 \bp \eta\|_{1.5}^2\Big) \leq \PP_0+C(\eps)\Er(T)+\PP\int_0^T\PP
\end{align}
is identical to what has been done in Section \ref{sect E3kk}. 

We also need the estimates for the interior Sobolev norms of the pressure $q$. $q$ satisfies the following elliptic system
\begin{equation}
\begin{cases}
-\lapArk q=\p_t\Ark^{\mu\alpha}\p_\mu v_{\alpha}+\divAr\bp^2\eta &\text { in }\Omega,\\
\frac{\p q}{\p N}=(\delta^{\mu 3}-\Ark^{\mu\alpha}\Ark^{3\alpha})\p_{\mu} q-\Ark^{3\alpha}\p_tv_{\alpha}+\Ark^{3\alpha}\bp^2\eta_{\alpha} &\text { on }\Gamma,\\
\frac{\p q}{\p N}=0 &\text { on }\Gamma_0.
\end{cases}
\end{equation} The elliptic estimates are identical to \eqref{ellq} in Section \ref{ellkk}, so we omit the details here. This concludes the proof of Proposition \ref{linearkk}.

\subsection{The Picard iteration}
We now prove that the sequence $\{(\eta_{\mm},v_\mm,q_\mm)\}_{m\in\N^*}$ has a strongly convergent subsequence. We define $[f]_{\mm}:=f_{\mmn}-f_{\mm}$ for any function $f$ and then $([\eta]_\mm,[v]_\mm,[q]_\mm)$ satisfies the following system
\begin{equation}
\begin{cases}
\p_t[\eta]_\mm=[v]_\mm ~~~&\text{ in }\Omega,\\
\p_t[v]_\mm-\bp^2[\eta]_{\mm}+\nabla_{\Ak_{\mm}}[q]_{\mm}=-\nabla_{[\Ak]_{\mml}}q_\mm ~~~&\text{ in }\Omega,\\
\dive_{\Ak_{\mm}}[v]_\mm=-\dive_{[\Ak]_{\mml}}v_\mm~~~&\text{ in }\Omega,\\
[q]_\mm=\kk(1-\TL)([v]_\mm\cdot\nnk_\mm)+h_{\mm}~~~&\text{ on }\Gamma,\\
([\eta]_\mm,[v]_{\mm})|_{t=0}=(\mathbf{0},\mathbf{0}).
\end{cases}
\end{equation}where
\begin{align*}
h_\mm=&\kk(1-\TL)(v_\mm\cdot[\nnk]_{\mml})\\
&-\sigma\left(\sqrt{\tilde{g}_{\mm}}g_{\mm}^{ij}\Pi_{\mm\alpha}^{\lambda}\TP_i\TP_j\eta_{\mm\lambda}\nnk_{\mm}^{\alpha}-\sqrt{\tilde{g}_{\mml}}g_{\mml}^{ij}\Pi_{\mml\alpha}^{\lambda}\TP_i\TP_j\eta_{\mml\lambda}\nnk_{\mml}^{\alpha}\right)
\end{align*} and $\Pi_{(m)}:=\nn_{(m)}\otimes\nn_{(m)}$.

We also define the energy functional of $([\eta]_\mm,[v]_\mm,[q]_\mm)$ to be
\begin{equation}\label{gapEr}
[\Er]_{\mm}:=[\Er]_{\mm}^{(1)}+[\Er]_{\mm}^{(2)},
\end{equation}where
\begin{equation}
\begin{aligned}
[\Er]_{\mm}^{(1)}(T)&:=\left\|[\eta]_\mm\right\|_{3.5}^2+\left\|[v]_\mm\right\|_{3.5}^2+\left\|\p_t [v]_\mm\right\|_{2.5}^2+\left\|\p_t^2 [v]_\mm\right\|_{1.5}^2+\left\|\p_t^3 [v]_\mm\right\|_{0}^2\\
&+\left\|\bp[\eta]_\mm\right\|_{3.5}^2+\left\|\p_t \bp[\eta]_\mm\right\|_{2.5}^2+\left\|\p_t^2 \bp[\eta]_\mm\right\|_{1.5}^2+\left\|\p_t^3 \bp[\eta]_\mm\right\|_{0}^2\\
[\Er]_{\mm}^{(2)}(T)&:=\frac{\kk}{\sigma}\int_0^T\left|\p_t^3 [v]_\mm\cdot\nnk_\mm\right|_1^2\dt+\kk\bigg(\int_0^T \|\p_t^3 [v]_{\mm}\|_{1.5}^2+\int_0^T \|\p_t^3 \bp [\eta]_{\mm}\|_{1.5}^2\bigg).
\end{aligned}
\end{equation}

\subsubsection{The div-curl estimates}

 For $k=0,1,2$
\begin{align}
\label{gapdivcurl}\|\p_t^{k}[v]_\mm\|_{3.5-k}\lesssim&\|\p_t^k[v]_\mm\|_0+\|\dive\p_t^k[v]_\mm\|_{2.5-k}+\|\curl\p_t^k[v]_\mm\|_{2.5-k}+|\TP\p_t^k[v]_\mm\cdot N|_{2-k},\\
\|\p_t^{k}\bp[\eta]_\mm\|_{3.5-k}\lesssim&\|\p_t^k\bp[\eta]_\mm\|_0+\|\dive\p_t^k\bp[\eta]_\mm\|_{2.5-k}\\
&+\|\curl\p_t^k\bp[\eta]_\mm\|_{2.5-k}+|\TP\p_t^k\bp[\eta]_\mm\cdot N|_{2-k}.\nonumber
\end{align}
Again, each part in the div-curl estimates should follow in the same way as in Section \ref{divcurlkk} so we omit the proof. For example, one can take $\curl_{\Ak_{\mm}}$ in the equation of $[v]_\mm$ to get the evolution equation of the curl:
\begin{align*}
\p_t(\curl_{\Ak_{\mm}} [v]_{\mm})&-\bp\left(\curl_{\Ak_{\mm}}\bp[\eta]_\mm\right)\\
=&-\curl_{\Ak_{\mm}}(\nabla_{[\Ak]_{\mml}}q_\mm)-\epsilon_{\alpha\beta\gamma}\p_t\Ak_{\mm}^{\mu\beta}\p_{\mu}[v]_{\mm\gamma}-[\bp,\curl_{\Ak_{\mm}}](\bp[\eta]_{\mm}).
\end{align*} Then standard energy estimates give the control of curl part. One can control the divergence of $[v]$ via $\dive_{[\Ak]_{\mml}}v_\mm$, and control the divergence of $\bp[\eta]$ in the same way as in Section \ref{divcurlkk}. To control the interior terms in $[\Er]_\mm^{(2)}$, we also need a similar div-curl decomposition for the $\kk$-weighted terms and follow the method in Section \ref{sect E3kk}. Then we have
\begin{align}
&\kk\bigg(\int_0^T \|\p_t^3 [v]_{\mm}\|_{1.5}^2+\int_0^T \|\p_t^3 \bp [\eta]_{\mm}\|_{1.5}^2\bigg)\nonumber\\ 
\label{gapEr2i}\lesssim& \PP_0+\eps[\Er]_{\mm}(T)+P([\Er]_{\mm}(T),\Er_{\mm,\mml}(T))\int_0^T P([\Er]_{\mm,\mml}(t),\Er_{\mm,\mml}(t))\dt. 
\end{align}

\subsubsection{Elliptic estimates of pressure}
 The quantity $[q]_\mm$ verifies the elliptic equation equipped with the Neumann boundary condition:
\begin{align}
\Delta_{\Ak_{\mm}}[q]_\mm=-\dive_{\Ak_{\mm}}(\nabla_{[\Ak]_{\mml}}q_\mm-\bp^2[\eta]_\mm)+\p_t\Ak^{\mu\alpha}_\mm \p_\mu[v]_{\mm\alpha}-\p_t(\dive_{[\Ak]_\mml}v_\mm), &\text{ in }\Omega,\label{ell diff}\\
\Ak^{3\alpha}_{\mm}\Ak^{\mu}_{\mm\alpha}\p_\mu[q]_m=\Ak^{3\alpha}_{\mm}(-\nabla_{[\Ak]_{\mml}}q_\mm-\p_t[v]_\mm+\bp^2[\eta]_{\mm})_{\alpha},&\text{ on }\Gamma.\label{ell bdy}
\end{align}
Compared to the estimates in Section \ref{ellkk}, we need to control the contribution of $\di_{[\Ak]_{\mm}}(\nabla_{[\Ak]_{\mml}}q_\mm)$ on the RHS of \eqref{ell diff}, as well as the contribution of $\Ak^{3\alpha}_{\mm}\nabla_{[\Ak]_{\mml}}q_\mm$ on the RHS of \eqref{ell diff}, when estimating $\|[q]_\mm\|_{3.5}$, but this is straightforward. 
\begin{align}
\|\dive_{\Ak_{\mm}}(\nabla_{[\Ak]_{\mml}}q_\mm)\|_{1.5}\lesssim P(\|[\Ak]_{\mml}\|_{2.5},\|q_\mm\|_{3.5},\|\Ak_\mm\|_{2.5}),
\end{align}
In addition to this, for the boundary contribution, we have
\begin{align}
|\Ak_{\mm} N\cdot \nabla_{[\Ak]_{\mml}}q_\mm|_{2}\lesssim P(\|[\Ak]_{\mml}\|_{2.5},\|q_\mm\|_{3.5},\|\Ak_\mm\|_{2.5}).
\end{align}
The quantities $\|[\p_t q]_{\mm}\|_{2.5}$, $\|[\p_t^2 q]_{\mm}\|_{1.5}$ and $\|[\p_t^3 q]_{\mm}\|_{0}$ are treated analogously by invoking the arguments in Section \ref{ellkk}.

\subsubsection{Boundary estimates}
The boundary estimates are parallel to those in Section \ref{sect boundary est simplified}. Analogous to \eqref{ellll}, we have
\begin{equation}\label{elllll}
\kk\TL([v]_{\mm}\cdot\nnk_{\mm})=\kk([v]_{\mm}\cdot\nnk_{\mm})+h_{\mm}-[q]_{(m)}.
\end{equation}
Then using the boundary elliptic estimates, we get
\begin{equation}\label{gapvbdry}
\begin{aligned}
|[v]_{\mm}\cdot\nnk_{\mm}|_3\lesssim_{\kk^{-1}}&|[v]_{\mm}\cdot\nnk_{\mm}|_1+|h_{\mm}|_1+\|[q]_\mm\|_{1.5}\\
\lesssim&|[v]_{\mm}\cdot\nnk_{\mm}|_1+\|[q]_\mm\|_{1.5}+ |v_{(m)}|_3P(|\TP[\ek]_{(m-1)},\TP^2[\eta]_\mml|_1,|\TP\eta_{\mml}|_{2})\\
\lesssim& \PP_0+P(\Er_{\mm,\mml}(T))\int_0^T P([\Er]_{\mm,\mml}(t),\Er_{\mm,\mml}(t))\dt.
\end{aligned}
\end{equation}
As for the magnetic field, we use the fact that $\bp=b_0^j\TP_j$ on $\Gamma$ to get
\[
\bp[\eta]_{\mm}\cdot\nnk_{\mm}=0+\int_0^T\bp [v]_{\mm}\cdot\nnk_\mm+\bp[\eta]_\mm\cdot\p_t\nnk_\mm.
\]
Similarly as in Section \ref{sect boundary est simplified}, one can directly control the $H^3(\Gamma)$-norm of the second term. Then the first term can be controlled by using elliptic estimates in $\bp$-differentiated elliptic equation \eqref{elllll}. We omit the detailed proof because there is no essential difference from the arguments in Section \ref{sect boundary est simplified}.
\begin{equation}\label{gapbbdry}
|\bp[\eta]_{\mm}\cdot\nnk_{\mm}|_3\lesssim_{\kk^{-1}} \int_0^T P([\Er]_{\mm}(t),\Er_{\mm}(t))\dt.
\end{equation}
Taking one time derivative, we can similarly control the boundary norm of $\p_t[v]_\mm$ and $\p_t\bp[\eta]_\mm$. We skip the details.
\begin{equation}\label{gaptbdry}
\left|\p_t[v]_{\mm}\cdot\nnk_{\mm},\p_t\bp[\eta]_{\mm}\right|_2\lesssim_{\kk^{-1}} \PP_0+P(\Er_{\mm,\mml}(T))\int_0^T P([\Er]_{\mm,\mml}(t),\Er_{\mm,\mml}(t))\dt.
\end{equation}
For the $H^1(\Gamma)$-norm of $\p_t^2[v]_\mm$ and $\p_t^2\bp[\eta]_\mm$, one can use the $\kk$-weighted interior terms in $[\Er]_{\mm}^{(2)}$ and Sobolev trace lemma to get the control
\begin{equation}\label{gapttbdry}
\begin{aligned}
&\left|\p_t^2[v]_{\mm}^3,\p_t^2\bp[\eta]_{\mm}^3\right|_1\lesssim\left\|\p_t^2[v]_{\mm},\p_t^2\bp[\eta]_{\mm}\right\|_{1.5}\\
\lesssim& \PP_0+\int_0^T\left\|\p_t^3[v]_{\mm}\cdot\nnk_{\mm},\p_t^3\bp[\eta]_{\mm}\right\|_{1.5}\dt\\
\lesssim& \PP_0+\sqrt{\frac{T}{\kk}}\left\|\sqrt{\kk}\p_t^3[v]_{\mm},\sqrt{\kk}\p_t^3\bp[\eta]_{\mm}\right\|_{L_t^2H_y^{1.5}}\\
\lesssim_{\kk^{-1}}&\PP_0+\sqrt{T}P([\Er]_{\mm}^{(2)}(T)).
\end{aligned}
\end{equation}
Finally, we need to control the difference between $X\cdot N$ and $X\cdot\nnk_\mm$, which should be done in the same way as \eqref{from1}-\eqref{N-nnr}, so we do not repeat the calculations. For $k=0,1$, we have for $X=[v]_\mm,\bp[\eta]_\mm$
\begin{equation}\label{gapnormal}
|\p_t^k X^3-\p_t^k(X\cdot\nnk_\mm)|_{3-k}\lesssim_{\kk^{-1}}\int_0^TP([\Er]_{\mm},\Er_{\mm,\mml}(t))\dt.
\end{equation}

Combining \eqref{gapvbdry}-\eqref{gapnormal}, we get the boundary estimates as
\begin{equation}\label{gapbdry}
\sum_{k=0}^2\left|\p_t^k([v]_\mm^3,\bp[\eta]_{\mm}^3)\right|_{3-k}\lesssim_{\kk^{-1}}\PP_0+P(\Er_{\mm,\mml}(T))\int_0^T P([\Er]_{\mm,\mml}(t),\Er_{\mm,\mml}(t))\dt.
\end{equation}

\subsubsection{Estimates of full-time derivatives}
Now it remains to control the $L^2$-norm of full time derivatives. By replacing $\p_t^4$ in Section \ref{lineartg} by $\p_t^3$, we can do analogous computation to control $\|\p_t^3 [v]_\mm\|_0$ and $\|\p_t^3\bp[\eta]_\mm\|_0$. The $\kk$-weighted boundary terms in $[\Er]_{\mm}^{(2)}$ are produced in the analogues of \eqref{IrB0}. The only difference is that we should control the extra contribution (under time integral) of $\nabla_{[\Ak]_{\mml}}q_\mm$ in the interior  and the $\sigma$-coefficient part in the term $h_{(m)}$ on the boundary. These quantities can all be directly controlled
\[
\|\p_t^3 \nabla_{[\Ak]_{\mml}}q_\mm\|_0\lesssim P\left(\|[v]_\mml,\p_t[v]_\mml, \p_t^2[v]_\mml\|_2,\|\p_t^3 q_\mm\|_1,\|\p_t^2 q_\mm,\p_t q_\mm,q_\mm\|_2\right).
\]
\[
|\p_t^3h_{(m),\sigma}|_0\lesssim P\left(|\p_t^2v_{\mm,\mml}|_2,|\TP\eta_{\mm,\mml}, \TP v_{\mm,\mml},\TP \p_tv_{\mm,\mml}|_{L^{\infty}}\right).
\]
Therefore, one can get
\begin{equation}\label{gapttt}
\begin{aligned}
&\|\p_t^3[v]_\mm\|_0^2+\|\p_t^3\bp[\eta]_\mm\|_0^2+\frac{\kk}{\sigma}\int_0^T\left|\p_t^3 [v]_\mm\cdot\nnk_\mm\right|_1^2\dt\\
\lesssim&\PP_0+\int_0^T P([\Er]_\mm(t),\Er_{\mm,\mml}(t))\dt
\end{aligned}
\end{equation}

\subsection{Well-posedness of the nonlinear approximate problem}
 We conclude this section with the following proposition.

\begin{prop}[\textbf{Local well-posedness of the nonlinear $\kk$-approximation problem}]\label{nonlinearkksol}
For each fixed $\kk>0$, there exists $T_\kk'>0$ such that the nonlinear $\kk$-approximation problem \eqref{MHDLkk} has a unique strong solution  $(\eta(\kk),v(\kk),q(\kk))$ in $[0,T_\kk']$ that satisfies
\begin{equation}
\sup_{0\leq t\leq T_\kk'}\Er'(t)\leq\mathcal{C}
\end{equation}
where $\Er'(t)=\Er^{(1)'}(t)+\Er^{(2)'}(t)$
\begin{equation}
\begin{aligned}
\Er^{(1)'}(t)&:=\left\|\eta\right\|_{4.5}^2+\left\|v\right\|_{4.5}^2+\left\|\p_t v\right\|_{3.5}^2+\left\|\p_t^2 v\right\|_{2.5}^2+\left\|\p_t^3 v\right\|_{1.5}^2+\left\|\p_t^4 v\right\|_{0}^2\\
&+\left\|\bp\eta\right\|_{4.5}^2+\left\|\p_t \bp\eta\right\|_{3.5}^2+\left\|\p_t^2 \bp\eta\right\|_{2.5}^2+\left\|\p_t^3 \bp\eta\right\|_{1.5}^2+\left\|\p_t^4 \bp\eta\right\|_{0}^2\\
\Er^{(2)'}(t)&:=\frac{\kk}{\sigma}\int_0^T\left|\p_t^4 v\cdot\nnk\right|_1^2\dt+\kk\bigg(\int_0^T \|\p_t^4 v\|_{1.5}^2+\int_0^T \|\p_t^4 \bp \eta\|_{1.5}^2\bigg).
\end{aligned}
\end{equation}
\end{prop}
\begin{proof}
Summarizing \eqref{gapdivcurl}-\eqref{gapEr2i}, \eqref{gapbdry}-\eqref{gapttt}, we can get the following inequality
\[
\begin{aligned}
[\Er]_{\mm}(T)\lesssim_{\kk^{-1}}&~ \PP_0+\eps[\Er](T)+TP([\Er]_{\mm}(T))\\
&+P([\Er]_{\mm}(T),\Er_{\mm,\mml}(T))\int_0^TP([\Er]_{\mm,\mml}(t),\Er_{\mm,\mml}(t)).
\end{aligned}
\]
By Gronwall-type inequality in Tao \cite{tao2006nonlinear} and the conclusion of Proposition \ref{linearkk}, there exists some $T_{\kk}'>0$, such that $\forall t\in[0,T_\kk']$
\[
[\Er]_{\mm}(t)\leq \frac{1}{4}[\Er]_{\mml}(t),
\]which implies $[\Er]_{\mm}(t)\leq 4^{-m} \PP_0$. Let $m\to\infty$, we know the sequence $\{(\eta_\mm,v_\mm,q_\mm)\}$ must strongly converge. The strong limit is denoted by $(\eta(\kk),v(\kk),q(\kk))$ which exactly solves the nonlinear $\kk$-approximation problem \eqref{MHDLkk}. By taking $m\to\infty$ in the energy of linearized equation \eqref{linearized}, one can also get the energy estimates.
\end{proof}

\section{Local well-posedness}\label{lwp1}
\subsection{Uniqueness and well-posedness}\label{unique}

Combining the conclusions of Proposition \ref{nonlinearkk} and Propostion \ref{nonlinearkksol} and letting $\kk\to 0_+$, we actually prove that there exists some time $T'>0$ (only depends on the initial data), such that the original system \eqref{MHDL} has solution $(\eta,v,q)$ satisfying the energy estimates
\[
\sup_{0\leq t\leq T} E(t)\leq \mathcal{C},
\]where $\mathcal{C}=\mathcal{C}(\|v_0\|_{4.5}, \|b_0\|_{4.5}, |v_0|_5)$, and the energy functional $E$ is defined to be
\begin{equation}\label{MHDLenergy}
\begin{aligned}
E(t):=&\left\|\eta\right\|_{4.5}^2+\sum_{j=0}^3\left\|\left(\p_t^jv,\p_t^j \bp\eta\right)\right\|_{4.5-j}^2+\left\|\left(\p_t^4 v,\p_t^4 \bp\eta\right)\right\|_{0}^2\\
&+\sum_{j=0}^3\left|\TP\left(\Pi\TP^{3-j}\p_t^j v\right)\right|_0^2+\left|\TP\big(\Pi\TP^3 \bp\eta\big)\right|_0^2, \\
\end{aligned}
\end{equation}

To establish the local well-posedness, it remains to prove the uniqueness. Let $\{(\eta_{\mm},v_\mm,q_\mm)\}_{m=1,2}$ be two solutions of \eqref{MHDL} satisfying the energy estimates. Then we define
\[
[\eta]:=\eta_{(1)}-\eta_{(2)},~[v]:=v_{(1)}-v_{(2)},~[q]:=q_{(1)}-q_{(2)},~[a]:=a_{(1)}-a_{(2)}.
\] Then $([\eta],[v],[q])$ satisfies the following system
\begin{equation}
\begin{cases}
\p_t[\eta]=[v]~~~& \text{in}~[0,T]\times\Omega;\\
\p_t[v]-\bp^2[\eta]+\nabla_{a_{(1)}} [q]=-\nabla_{[a]}q_{(2)}~~~& \text{in}~[0,T]\times\Omega;\\
\dive_{a_{(1)}} [v]=-\dive_{[a]}v_{(2)},&\text{in}~[0,T]\times\Omega;\\
\dive b_0=0~~~&\text{in}~[0,T]\times \Omega ;\\
[v^3]=b_0^3=0~~~&\text{on}~\Gamma_0;\\
[q]n_{(1)}=-\sigma g_{(1)}^{ij}\Pi_{(1)}\TP^2_{ij}[\eta]-\sigma\sqrt{g_{(1)}}\Delta_{[g]}\eta_{(2)} ~~~&\text{on}~\Gamma;\\
b_0^3=0 ~~~&\text{on}~\Gamma,\\
([\eta],[v])=(\mathbf{0},\mathbf{0})~~~&\text{on}~\{t=0\}{\times}\overline{\Omega}.
\end{cases}
\end{equation}

Define
\begin{equation}\label{MHDLgapenergy}
\begin{aligned}
[E](t):=&\left\|[\eta]\right\|_{3.5}^2+\left\|[v]\right\|_{3.5}^2+\left\|\p_t [v]\right\|_{2.5}^2+\left\|\p_t^2 [v]\right\|_{1.5}^2+\left\|\p_t^3 [v]\right\|_{0}^2\\
&+\left\|\bp[\eta]\right\|_{3.5}^2+\left\|\p_t \bp[\eta]\right\|_{2.5}^2+\left\|\p_t^2 \bp[\eta]\right\|_{1.5}^2+\left\|\p_t^3 \bp[\eta]\right\|_{0}^2\\
&+\left|\TP\left(\Pi_{(1)}\p_t^2[v]\right)\right|_0^2+\left|\TP\left(\Pi_{(1)}\TP\p_t [v]\right)\right|_0^2+\left|\TP\left(\Pi_{(1)}\TP^2 v\right)\right|_0^2+\left|\TP\big(\Pi_{(1)}\TP^2 \bp\eta\big)\right|_0^2. \\
\end{aligned}
\end{equation}

Then we can mimic the proof in Section \ref{sect nonlinearkk} to get the energy estimates of $[E]$
\[
[E](T)\lesssim P([E](T),E(T))\int_0^TP([E](t),E(t))\dt,
\]which together with Gronwall-type inequality yields
\[
\exists T\in[0,T'],~~[E](t)=0~~\forall t\in[0,T]
\]which establishes the local well-posedness of \eqref{MHDL} in $[0,T]$.

\subsection{Regularity of initial data and free surface}\label{data}

Finally, we need to prove that the norms of time derivatives can be controlled by $\|v_0\|_{4.5},\|b_0\|_{4.5}$ and $|v_0|_5$. This part is exactly the same as in Section \ref{datakk} or \cite[Section 7.1]{luozhangMHDST3.5}. Finally, the boundary condition of \eqref{MHDL} gives us an elliptic equation of $\eta$ on $\Gamma$
\[
-\sigma\sqrt{g}\Delta_g\eta^\alpha=a^{3\alpha}q.
\]By using elliptic estimates in Dong-Kim \cite{DKell} (see also \cite[Proposition 3.4]{DK2017}), one has
\[
|\eta|_5\lesssim |a^{3\alpha} q|_3\lesssim|(\TP\eta\times\TP\eta)q|_3\lesssim P(\|\eta\|_{4.5})\|q\|_{3.5}.
\]
Similarly, taking a time derivative gives us the elliptic equation of $v^\alpha$
\begin{align*}
\sqrt{g}g^{ij}\TP_{ij}^2v^\alpha=&\sqrt{g}g^{ij}\Gamma_{ij}^k\TP_k v^\alpha-\p_t(\sqrt{g}g^{ij})\TP_{ij}^2\eta^\alpha-\p_t(\sqrt{g}g^{ij}\Gamma_{ij}^k)\TP_k\eta^\alpha\\
&-\sigma^{-1}(\p_ta^{3\alpha}q+a^{3\alpha}\p_t q),
\end{align*}and thus by the similar argument in \cite[Section 5.1]{luozhangMHDST3.5} we get
\[
|v(t)|_5\lesssim P(E(t))~~~~\text{in }[0,T].
\]
This concludes the proof of Theorem \ref{lwp}.

\section*{Acknowledgment} We thank the anonymous referees for their helpful comments on improving the exposition of this manuscript. X. Gu was supported by the National Natural Science Foundation of China (12031006). C. Luo was supported by the Hong Kong RGC grant CUHK-24304621 and a direct grant of research 4053457.

\end{document}